\documentclass{amsart}

\usepackage[english]{babel}

\usepackage[letterpaper,top=2cm,bottom=2cm,left=3cm,right=3cm,marginparwidth=1.75cm]{geometry}

\usepackage{amssymb}
\usepackage{amsmath}
\usepackage{subfigure}
\usepackage{epsfig}
\usepackage[all]{xy}
\usepackage[colorlinks=true, allcolors=blue]{hyperref}

\newtheorem{theorem}{Theorem}[section] 
\newtheorem{lemma}[theorem]{Lemma}     

\newtheorem{proposition}[theorem]{Proposition}

\theoremstyle{definition}
\newtheorem{example}[theorem]{Example}

\title[Quotient graphs and random walks]
{Eichler orders, quotient graphs and random walks}

\author{Luis Arenas-Carmona}
\author{Marco Godoy}

\newcommand\classgp{\mathfrak{g}}
\newcommand\anything{*}
\newcommand\sgn{\textnormal{sgn}}
\newcommand\diag{\textnormal{diag}}
\newcommand\disc{\textnormal{disc}}
\newcommand\ad{\mathbb{A}}
\newcommand\Z{\mathbb{Z}}
\newcommand\lu{\mathfrak{L}}
\newcommand\compleji{\mathbb{C}}
\newcommand\alge{\mathfrak{A}}
\newcommand\tr{\textnormal{ tr}}
\newcommand\la{\Lambda}
\newcommand\ideala{\mathcal{A}}
\newcommand\idealb{\mathcal{B}}
\newcommand\bark{\bar{k}}
\newcommand\uno{\{1\}}
\newcommand\quadc{k^*/(k^*)^2}
\newcommand\Ba{\mathfrak B}
\newcommand\oink{\mathcal O}
\newcommand\D{\mathcal D}
\newcommand\nuD{\nu_{\D}}
\newcommand\Q{\mathbb Q}
\newcommand\enteri{\mathbb Z}
\newcommand\gal{\mathcal G}
\newcommand\dete{\textnormal{det}}
\newcommand\ovr[2]{#1|#2}
\newcommand\vaa{\longrightarrow}
\newcommand\pws{\mathcal{P}}
\newcommand\Da{\mathfrak{D}}
\newcommand\Ha{\mathfrak{H}}
\newcommand\Ea{\mathfrak{E}}
\newcommand\Ra{\mathfrak{R}}
\newcommand\Fa{\mathfrak{F}}
\newcommand\Ma{\mathfrak{M}}
\newcommand\ma{\mathfrak{m}}
\newcommand\Txi{\lceil}

\usepackage{xcolor}
\newcommand{\red}[1]{{\color{red}#1}}
\newcommand{\green}[1]{{\color{teal}#1}}
\newcommand{\blue}[1]{{\color{cyan}#1}}
\newcommand{\orange}[1]{{\color{orange}#1}}

\newcommand\bmattrix[4]{\left(\begin{array}{cc}#1&#2\\#3&#4\end{array}\right)}
\newcommand\sbmattrix[4]{\textnormal{\scriptsize$\left(\begin{array}{cc}#1&#2\\#3&#4\end{array}\right)$\normalsize}}
\newcommand\blattice[4]{\left<\begin{array}{cc}#1&#2\\#3&#4\end{array}\right>}
\newcommand\bspace[4]{\left[\begin{array}{cc}#1&#2\\#3&#4\end{array}\right]}

\newcommand\wq{\mathfrak{q}}
\newcommand\uink{\mathcal U}
\newcommand\qunitaryomega{\uink}
\newcommand\quniti[4]{\rm{SU}^{#4}}
\newcommand\qunitarylam{\quniti nkh{\la}}

\newcommand\spinomega{Spin_{n}(Q)}
\newcommand\hspinomega{Spin_{n}(\D,h)}

\newcommand\matrici{\mathbb{M}}
\newcommand\finitum{\mathbb{F}}

\newcommand\splitt[1]{\matrici_2(#1)}
\newcommand\splitk{\splitt k}

\newcommand\Hgot{\mathfrak{H}}
\newcommand{\powerset}{\raisebox{.15\baselineskip}{\Large\ensuremath{\wp}}}

\begin{document}
\maketitle

\begin{abstract}
We study the extent to which the quotient of the Bruhat-Tits tree
at one place $Q$, associated to a genus of orders of maximal rank,
can be computed from the analogous quotient at a different 
place $P$.
We show that this computation can be carried out, except for a
small set of vertices depending on $P$, but not on $Q$.
We give some geometrical conditions on the quotient at $P$ that
ensure that this
exceptional set is empty. This generalizes the formulas from a
previous work that allow the computation of the quotient graph
at all places, for the genus of maximal orders over the
projective line. The methods presented here yield
similar results
for other genera or other curves.
\end{abstract}

\section{Introduction}\label{intro}

In all of this work $K$ is a global function field, 
i.e. the field of rational functions on a smooth irreducible
projective curve $X$ over a finite field $\finitum$. 
We let $\oink_X$ denote the structure sheaf of $X$, 
and we let $|X|$
be the set of closed points in $X$. Furthermore,
we write $\alge=\mathbb{M}_2(K)$ for the $2$-by-$2$ matrix algebra, while $\Ra$, $\Da$ and $\Ea$ denote
$X$-orders of maximal rank in $\alge$, i.e. sheaves of $\oink_X$-algebras whose generic fiber is $\alge$.
Throughout this work, $K_P$ is the completion of $K$ for the valuation defined by the place $P\in|X|$.
 A similar convention applies to completions of vector spaces, algebras, and orders. 
For example $\alge_P= \mathbb{M}_2(K_P)$.
The completion $\Da_P$ of an order $\Da$ in $\alge$ is defined as the closure in $\alge_P$ of the ring
$\Da(U)\subseteq K$, for any affine open set $U$ of $X$ containing $P$.
This definition is independent of the choice of $U$. 
Recall that the adele ring of an algebra is the restricted topological
product
\begin{equation}\label{rtp}
\alge_{\ad}=\left\{a=(a_P)_P\in\prod_{P\in|X|}\alge_P
\Big|a_P\in\Da_P\textnormal{ for almost all }P\right\}.
\end{equation}
  This definition is independent of $\Da$. The global algebra $\alge$
is identified with the image of the diagonal embedding $\mathrm{diag}:\alge\rightarrow\alge_\ad$
defined by $\mathrm{diag}(a)=(a)_P$.
Similarly, we define the adele ring $\ad=K_\ad$, replacing  
$\Da$  in (\ref{rtp}) by the order $\oink_X$.
The group $J_X=\ad^*$ is called the idele group of $K$ in the sequel.
 More generally, we write $R^*$ for the group of units of an arbitrary ring $R$.

Two orders $\Da$ and $\Ea$ are said to be locally conjugate,
 or to belong to the same genus, if any of the following equivalent statements hold:
\begin{enumerate}
\item The stalks $\Da_P^o,\Ea_P^o\subseteq\alge$ are conjugates at every place $P\in|X|$. 
\item The completions $\Da_P,\Ea_P\subseteq\alge_P$ are conjugate at every place $P\in|X|$. 
\item There is an element $a\in\alge_{\ad}^*$ satisfying $a_P\Da_P a_P^{-1}=\Ea_P$
at every place $P$.
\end{enumerate}
When $(3)$ is satisfied we write simply $a\Da a^{-1}=\Ea$ by abuse of notation. Similar conventions 
apply to other linear algebraic groups acting on lattices.
Being locally conjugate is an equivalence relation.
As usual (see \cite{abelianos} or \cite{cqqgvro}), by a genus of orders of maximal rank, we mean an equivalence class
for this relation.
The study of genera is an important part of the theory of orders, since local classification is much simpler
than global classification. On the other hand, over a Dedekind domain, like $A=\oink_X(U)$ for an affine set
$U\subseteq X$, two $A$-orders in $\alge$, 
like $\Da(U)$ and $\Ea(U)$, are conjugate precisely when they
are in the same spinor genus. Two $X$-orders in a genus are said to be in the same spinor genus
if the adelic element $a$ in (3) above can be chosen of the form $a=bc$, where $b\in\alge$,
 and the local coordinate $c_P$ of $c=(c_P)_P$ has trivial reduced norm at each place $P$.
Spinor genera of $A$-orders are defined analogously. 

Determining whether two $X$-orders in the same spinor genus are conjugate
can be studied through classifying graphs. Fix a place $P\in |X|$, let $U=\{P\}^c\subset X$ be its complement,
and set $A=\oink_X(U)$ as before.
Fix a genus $\mathbb{O}$ whose orders are maximal at $P$. The classifying graph $C_P(\mathbb{O})$ is a graph,
with a vertex for each conjugacy class in  $\mathbb{O}$, whose connected
components $C_P(\Da)$ are in correspondence with the conjugacy classes of $A$-orders of the form $\Da(U)$ for 
$\Da\in\mathbb{O}$. Furthermore, each connected component is the quotient of the Bruhat-Tits tree 
$\mathfrak{t}(K_P)$ for $\mathrm{PSL}_2(K_P)$ by the 
normalizer in $\mathrm{PSL}_2(K)$ of the 
corresponding $\oink_X(U)$-order $\Da(U)$. 
Fix one such $A$-order $\Da(U)$.
For any local maximal order $\Da'(P)\subseteq\alge_P$,
there is an order $\Da'\in \mathbb{O}$ satisfying
$\Da'(U)=\Da(U)$ and $\Da'_P=\Da'(P)$.
Such orders are called $P$-variants of $\Da$.
The vertices in $C_P(\Da)$ are in correspondence with the conjugacy classes of these $P$-variants.

Computing the classifying graph for a given genus $\mathbb{O}$
at a single place $P$ is sufficient for the purpose of  
determine all
conjugacy classes. There are, however, additional reasons to compute these classifying graphs at different places:
\begin{enumerate}
\item Basse-Serre' Theory  allows us to determine generators for the 
associated arithmetic subgroups (c.f. \cite[\S5]{trees}).
\item Genera of Eichler orders can be studied in terms of genera of simpler 
Eichler orders via suitable quotient graphs (c.f. \cite{ogcseooff}
and \cite{bravo1}).
\item Selectivity of certain suborders is related to the structure of the graphs in several ways (c.f. \cite{cqqgvro}, \cite{ruidqo}). 
\end{enumerate}
For these reasons, it would be desirable to be able to compute the classifying graph $C_P(\mathbb{O})$ in terms
of the corresponding graph $C_Q(\mathbb{O})$ at a different place. 
In \cite{cqqgvro}, we were able to compute
all quotient graphs for maximal orders over the projective line 
$\mathbb{P}^1$ using this idea. 
An alternative approach, for the same curve and genus appears 
in \cite{Kohl}. In general terms, the strategy used in the 
former reference can be summarized as follows:
\begin{itemize}
    \item From the classifying graph $C_P(\mathbb{O})$
    we compute a ``neighborhood matrix $N_P(\mathbb{O})$''.
    \item The neighborhood matrix $N_Q(\mathbb{O})$, at a different
    place, can be computed from $N_P(\mathbb{O})$.
    \item The neighborhood matrix $N_Q(\mathbb{O})$ can be used
    to obtain the quotient graph $C_Q(\mathbb{O})$.
\end{itemize}
Above, the neighborhood matrix $N_P(\mathbb{O})=(m_{i,j})_{i,j}$ 
is an infinite matrix whose entry
$m_{i,j}$ is the number of  neighbors in 
$\mathfrak{t}(K_P)$ of $\Da_j$ that are
isomorphic to $\Da_i$, for some set $\{\Da_1,\Da_2,\dots\}$
of representatives of all conjugacy classes in $\mathbb{O}$.  
Our ability to compute the 
classifying graph from the corresponding neighborhood matrix 
depends on our understanding
on the action of the stabilizer of a vertex in its neighbor set. 
This is simple for some vertices, namely those
corresponding to split orders, as defined in Proposition \ref{p41},
to the point where we can
give explicit formulas for the number of edges joining two 
vertices from the corresponding coefficient $m_{i,j}$
and conversely, but potentially much harder for vertices 
corresponding to non-split orders (c.f. \cite[\S6]{cqqgvro}).
 In this work we concentrate in the central step, 
 i.e., computing one neighborhood matrix in terms of another.  
With this tool alone, we can compute which pairs of vertices are 
neighbors, but not the multiplicity of edges,
and we cannot distinguish half edges from loops (see \S2). 
We note, however, that our neighborhood
matrix encodes the same information that the weighted graph for 
Hecke operators described in \cite{Alvetall},
where we also find some computations for
the case of orders in a three dimensional
matrix algebra.
These Hecke operators play a role in Drinfeld and Lafforgue' proof of  Langland's correspondence,
see \cite[\S3]{Frenkel} for an account on that connection.

The present work can be seen as a partial extension
of the computations in \cite[\S6]{Alvetall}.
We say ``partial'' mainly for the following
reason:
\begin{quote}
Sometimes there is a small set of conjugacy classes 
for which our methods cannot be applied to compute
the weights of the edges joining them, 
so a full computation requires
another method, as the ones used in \cite{Alvetall}.
However, such exceptional set depends only on the
initial place $P$ and can be bounded by an examination
of the geometry of the quotient graph at $P$.
In any case, our method reduces significantly the effort
needed to obtain the full graph.
\end{quote}

 In \cite{cqqgvro} we were able to compute the neighborhood matrix, at all places, for the genus 
$\mathbb{O}_0$ of maximal orders over the  projective line $X=\mathbb{P}^1$. There are two facts that allowed us
 to perform these computations:
\begin{enumerate}
\item The neighborhood matrices $N_P(\mathbb{O})$ and $N_Q(\mathbb{O})$ at different places commute.
\item The graph $C_Q(\mathbb{O})$ is well understood beyond a finite set
$T=T(\mathbb{O},Q)$.
\end{enumerate}
It would be useful to be able to extends this method to a larger family of genera. In order to do this we need to
extend the description of the graph outside a finite set, a result proved by Serre in \cite{trees} for
the genus $\mathbb{O}_0$ of maximal orders.  We can achieve this, at this point, for Eichler orders of multiplicity free level,
although sometimes we can extend it to Eichler orders
whose level has small multiplicities. The meaning of this 
condition should  be clear in the examples.

In this work, we show that conditions (1) and (2) above, plus the knowledge of a single matrix $N_P(\mathbb{O})$,
  is sufficient to compute the neighborhood  matrix $N_Q(\mathbb{O})$ at any other place $Q$, 
except for the coefficients $m_{i,j}$ with $\Da_i$ and $\Da_j$ in some finite set $S=S(\mathbb{O},P)$ of vertices 
depending only on $P$. 
Since the classifying graph is defined only for places $P$ 
in the cofinite set
\begin{equation}\label{piq}
\Pi_{\mathbb{O}}=\left\{P\in|X|\Big|\Da_P\textnormal{ is maximal
for some (and therefore any) }\Da\in\mathbb{O}\right\},
\end{equation}
 we assume throughout that $P,Q\in\Pi_{\mathbb{O}}$. 
 We give sufficient conditions for the set $S$ to be empty, in terms of the geometry of the classifying graph $C_P(\mathbb{O})$. This is to be expected in some cases, e.g., when $\mathbb{O}=\mathbb{O}_0$, $X=\mathbb{P}^1$
and $\mathrm{deg}(P)=1$, which is precisely the main  result in \cite{cqqgvro}. However, we provide a few examples
where this is not the case. This holds in particular when the classifying graph $C_P(\mathbb{O})$, and the 
corresponding matrix $N_P(\mathbb{O})$, have a symmetry that coincide with the identity outside a finite subgraph.
See \S\ref{sec6} for details. When $S$ is not empty, our method can probably complemented with Hall algebra
computations as in \cite{Alvetall}, but the number of such computations needed is greatly reduced.

In \cite{Kohl}, all quotient graphs for the case where $X=\mathbb{P}^1$ and $\mathbb{O}=\mathbb{O}_0$
are described via simultaneous actions at different places. Their method requires the ring 
$A=\oink_X\Big(X-\{P\}\Big)$ to be a PID, as it is pointed out in the reference. We rely mostly on the geometry of the 
graph, which makes our result more general. 

\section{A few conventions on graphs}

 In all this work, we use the following notations concerning graphs: A graph $\mathfrak{g}$ is a $5$-tuple 
$\mathfrak{g}=(V,A,s,t,r)$ consisting on the following data:
\begin{itemize} \item Two disjoint sets $V=V_\mathfrak{g}$ and $A=A_\mathfrak{g}$,
 called the vertex set and the edge set of $\mathfrak{g}$.
\item Two maps $s,t:A\rightarrow V$ called source and target.
\item A map  $r:A\rightarrow A$ is called the reverse.
\end{itemize}
 They are assumed to satisfy the following compatibility properties:
$$r\big(r(a)\big)=a,\quad r(a)\neq a,\quad s\big(r(a)\big)=t(a).$$
In particular, we admit multiple edges and loops. 
The number of edges $a$ satisfying $s(a)=v$, for a fixed vertex
$v$ is called the valency of $v$. Valency one vertices are 
called leaves, although this convention is mostly used for
trees, as defined below.
If there is at most one edge $a$ satisfying $s(a)=v$ and 
$t(a)=w$ for each pair
of vertices $(v,w)$ the graph is called reduced.
An important example of reduced graph is the generic ray 
$\mathfrak{c}$ with 
vertex and edge sets given by
$V=\{v_0,v_1,\dots\}$ and $A=\{a_0,\bar{a}_0,a_1,\bar{a}_1,\dots\}$,
where
$$r(a_i)=\bar{a}_i,\qquad s(a_i) =v_i, \qquad t(a_i) =v_{i+1}.$$

 A morphism of graphs
$\delta:\mathfrak{g}\rightarrow \mathfrak{g}'$, where $\mathfrak{g}'=(V',A',s',t',r')$,
 is a pair of maps $\delta_V:V\rightarrow V'$ and $\delta_A:A\rightarrow A'$ satisfying
$s'\circ \delta_A=\delta_V\circ s$,
$t'\circ \delta_A=\delta_V\circ t$  and $r'\circ \delta_A=\delta_A\circ r$.
There is a category called $\mathfrak{Graphs}$, whose objects
are graphs in our sense, and its morphisms are those just
described. A ray is any isomorphic image, in $\mathfrak{Graphs}$,
of the generic ray.
A subgraph $\mathfrak{h}$ of a graph $\mathfrak{g}$ is a graph admitting a morphism 
$\iota:\mathfrak{h}\rightarrow\mathfrak{g}$ for which $\iota_V$ and $\iota_A$ are inclusions.

The category $\mathfrak{Graphs}$ has coproducts and fibered coproducts,
but not products.
The only fibered coproduct that we use in the sequel is the attachment of a graph $\mathfrak{e}$ 
with a distinguished initial vertex $v_0$  to a graph 
$\mathfrak{g}$ at a vertex $v\in V_{\mathfrak{g}}$,
 which is defined as the fibered coproduct of $\mathfrak{g}$ and $\mathfrak{e}$
with respect to the maps $\gamma:\mathfrak{p}\rightarrow \mathfrak{g}$ and  
$\delta:\mathfrak{p}\rightarrow \mathfrak{e}$, where
$\mathfrak{p}=(\{p\},\varnothing,\varnothing,\varnothing,\varnothing)$, $\gamma_V(p)=v$ and
$\delta_V(p)=v_0$. The attached graph  $\mathfrak{e}$ is usually a ray.
For example, the real line graph is the graph
obtained by attaching to rays to each other
by their initial vertex.
In the sequel, we often consider a finite graph with finitely many
attached rays called cusps. Figure \ref{sdoq} in \S3 shows a graph 
with a few cusps.

A graph is connected if it is not isomorphic to the coproduct of two non-empty graphs.
 The interval $\mathfrak{i}_n$ is defined as the smallest connected
subgraph of $\mathfrak{c}$ containing both $v_0$ and $v_n$.
 A path of length $n$ is an injective
 morphism of graphs $\gamma:\mathfrak{i}_n\rightarrow \mathfrak{g}$. A non-necessarily injective
morphism $\mu:\mathfrak{i}_n\rightarrow \mathfrak{g}$ is called a promenade. The integer
$n$ is called the length of the path or promenade. 
We usually emphasize the initial vertex $w_0=\mu_V(v_0)$ and the final vertex $w_n=\mu_V(v_n)$
by saying a promenade (or path) from $w_0$ to $w_n$ or between $w_0$ and $w_n$.
A graph $\mathfrak{g}$ is connected if and only if there is a 
promenade between every two vertices in $V_\mathfrak{g}$. For every two vertices $v$ and $w$
in a connected graph, the distance $d(v,w)$ is the minimal length of a promenade from $v$ to $w$.
A promenade, from $v$ to $w$, of minimal length is a path. Similarly, an injective
 morphism of graphs $\gamma:\mathfrak{c}\rightarrow \mathfrak{g}$ is called an infinite path, while
a non-necessarily injective morphism 
$\mu:\mathfrak{c}\rightarrow \mathfrak{g}$ is called an infinite promenade. 
For any infinite path $\gamma:\mathfrak{c}\rightarrow \mathfrak{g}$, we say that 
$\gamma_V(v_{n+1})$ is the vertex following 
$\gamma_V(v_n)$. This is used in particular for 
cusps. You can also define maximal paths or
promenades using the real line graph instead of
a ray. Details are left to the reader.
If a graph has a unique path from any
vertex $v$ to any vertex $w$, it is called a tree. The image of a maximal path, in a tree,
is call a full geodesic.

Let $G$ be a group.
 By a $G$-action by graph automorphisms,  on a graph $\mathfrak{g}$, 
 we mean an action of $G$ on both sets $A_\mathfrak{g}$ and 
 $V_\mathfrak{g}$ preserving the maps $s$, $t$ and $r$.
 With this definition, there is a well defined quotient graph 
 for the $G$-action provided that $a$ and $r(a)$ are in 
 different orbits for every edge $a$. In this case we say that 
 the action has no inversions.
Since the actions that concern us here typically 
have inversions, we define
quotient graphs in a slightly different way that allows them.
We define the barycentric subdivision $\mathfrak{g}'=(V'A',s',t',r')$ by
$$V'=V\cup A,\quad A'=\{0,1\}\times A,\qquad r(0,a)=\big(1,r(a)\big),\quad 
 r(1,a)=\big(0,r(a)\big),$$
$$s(0,a)=s(a),\quad 
 s(1,a)=a,\quad t(1,a)=t(a),\quad 
 t(0,a)=a.$$

Any action of a group $G$ on $\mathfrak{g}$ defines an action without inversions on $\mathfrak{g}'$,
 so the quotient graph $G\backslash\mathfrak{g}'$ is always defined.
We call $G\backslash\mathfrak{g}'$ the fine quotient of $G$ and denote it 
$G\backslash\backslash\mathfrak{g}$.
 For actions without inversions the fine quotient is just the barycentric
subdivision of the quotient. More generally, in the fine quotient graph, the elements of $V\subseteq V'$
 are called actual vertices, while the elements of $A\subseteq V'$ are called virtual
vertices. In the fine quotient we say an actual edge for a  length-$2$ path 
$\gamma:\mathfrak{i}_2\rightarrow
G\backslash\backslash\mathfrak{g}$ 
satisfying both $\gamma_A(a_0)=(0,a)$ and $\gamma_A(a_1)=(1,a)$ 
for some $a\in A$, while the actual distance between two actual vertices is one half of their usual graph distance.
 In particular, actual vertices are actual neighbors if the actual distance between them is $1$.
We use a few other conventions along these lines, so that, aside of a few words \emph{actual},
there is no difference between a statement on the quotient graph an its equivalent on
 the fine quotient graph for actions without inversions. By abuse of notations, we ignore the
word actual if there is no risk of confusion, which is usually the case since we make almost no use
of standard quotients, the exceptions being the formula (\ref{doubleqt}) below and some computations in \S\ref{intro}.
We also ignore virtual vertices in notations and drawings, except for the following
exception: An inversion produce a virtual vertex of valency $1$ in the fine quotient,
 joined to a vertex by an edge called a half-edge in the sequel. The latter virtual vertex
is drawn explicitly, and denoted by "$*$" in all figures. More generally, any bipartite graph
is called a fine graph whenever vertices on one side of the partition are called virtual vertices, and their
valencies are either $1$ or $2$. In this case, all of the preceding conventions apply.

 A multiple edge on a graph is a pair of vertices 
$(v,w)$ with more than one edge $a$ satisfying
$s(a)=v$ and $t(a)=w$. The analogous definition in the fine quotient graph goes as follows:
A multiple edge in the fine quotient graph is  a pair of actual vertices $(v,w)$ with more than one actual
edge from $v$ to $w$. The pair $(v,v)$ is considered a multiple
edge if there is, at least, either an actual  loop
or more than one half edge starting at $v$. For example, $(v_1,v_1)$ and $(v_2,v_3)$ are multiple
edges in Figure \ref{F2}, but $(v_4,v_4)$  is not, as there is only a  half edge at $v_4$.  If a fine quotient 
has no multiple edges, in this sense, we call it reduced. Reduced fine quotients can have half edges 
but not loops. Note that every fine quotient has a reduction, i.e., an associated reduced fine graph
with the same neighborhood relation
between actual vertices.

\begin{figure}
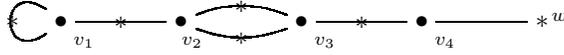

\[  \xygraph{
!{<0cm,0cm>;<.8cm,0cm>:<0cm,.8cm>::}
!{(0,1)}*+{\bullet}="a" !{(0.4,0.6)}*+{{}^{v_1}}="a1"
!{(2,1)}*+{\bullet}="b" !{(2.2,0.6)}*+{{}^{v_2}}="b1"
!{(4,1)}*+{\bullet}="c" !{(4.4,0.6)}*+{{}^{v_3}}="c1"
!{(6,1)}*+{\bullet}="x" !{(6.4,0.6)}*+{{}^{v_4}}="x1"
!{(1,0.975)}*+{*}="l1" !{(3,1.25)}*+{*}="l2"
!{(3,0.725)}*+{*}="l3" !{(5,.975)}*+{*}="l4"
!{(-0.805,1)}*+{*}="l5"
!{(8,1)}*+{*}="d" !{(8.3,1)}*+{{}^{w}}="d1"
"a"-@(lu,ld) "a" "a"-"b" "b"-@/^/"c" "b"-@/_/"c" "c"-"x" "x"-"d"} 
\]
\caption{A fine graph with a loop at $v_1$, a double edge between $v_2$ and $v_3$, and a half edge at $v_4$.
In latter pictures, the virtual vertices, which are denoted by $*$, are ignored unless they have valency 1,
 as it is the case for $w$.}\label{F2}
\end{figure}

A simplicial map $\gamma: \mathfrak{g}\rightarrow\mathfrak{g}'$ 
is called
surjective if both $\gamma_V$ and $\gamma_A$ are surjective. 
It is called
locally surjective at a vertex $v$ if the induced map
$\mathrm{nb}_\gamma(v)$
 from the set $\mathrm{nb}_{\mathfrak{g}}(v)$, of neighbors of $v$, 
to $\mathrm{nb}_{\mathfrak{g}'}\big(\gamma_V(v)\big)$ is surjective. 
It is called locally
surjective, or LC, if it is locally surjective at every vertex.
A ramified covering is a surjective LC simplicial map. 
A vertex $v$ is unramified if
$\mathrm{nb}_\gamma(v)$ is injective, and ramified otherwise. If $N\subseteq G$ are groups
acting on a graph $\mathfrak{g}$, 
there is a natural ramified covering $c:N\backslash\backslash\mathfrak{g}\rightarrow 
G\backslash\backslash\mathfrak{g}$.
Furthermore, if $N$ is normal in $G$, then $G/N$ acts without inversions on the fine quotient
$N\backslash\backslash\mathfrak{g}$,
and the corresponding quotient is
\begin{equation}\label{doubleqt}
  (G/N)\backslash(N\backslash\backslash\mathfrak{g})\cong G\backslash\backslash\mathfrak{g}.
\end{equation}
When this is the case, the covering is said to be regular.
In this case, a vertex $v$ is ramified precisely when its stabilizer
$\mathrm{Stab}_{G/N}(v)$ acts non-trivially on its neighbors.  

We need one final piece of notation for quotient graphs. 
Recall that the weight $m_{v,w}$ is the number of neighbors
of one (any) pre-image of $v$, in $\mathfrak{g}$, whose image is $w$.
The Weighted Reduced fine quotient
(WRFQ) of $\mathfrak{g}$ by $G$, denoted $G|\mathfrak{g}$, is the reduction of the fine quotient
$G\backslash\backslash\mathfrak{g}$, together with the 
data of the weights $m_{v,w}$, for all pairs of vertices
$(v,w)$ of $G\backslash\backslash\mathfrak{g}$. 
These are the quotients that our tools
allows us to compute directly. For instance, the WRFQ in Figure \ref{F3} does not allows us to discriminate
if the fine quotient $G\backslash\backslash\mathfrak{g}$
have a loop at $v_1$ (as it is the case in Fig. \ref{F2}), or two half edges. The only thing we can say for sure
if we only have the WRFQ is that the multiplicity of the (actual) edge joining $v$ and $w$ in
$G\backslash\backslash\mathfrak{g}$ is bounded by $\min\{m_{v,w},m_{w,v}\}$. Note that,
in this regard, loops count as edges of multiplicity $2$, so we know there is no loop, only a half edge,
at $v_4$. These weighted graph are those computed in \cite{Alvetall}.

\begin{figure}
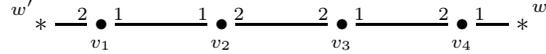

\[  \xygraph{
!{<0cm,0cm>;<.8cm,0cm>:<0cm,.8cm>::}
!{(-1,1)}*+{*}="u" !{(-1.3,1.2)}*+{{}^{w'}}
!{(0,1)}*+{\bullet}="a" !{(0,0.6)}*+{{}^{v_1}}
!{(2,1)}*+{\bullet}="b" !{(2,0.6)}*+{{}^{v_2}}
!{(4,1)}*+{\bullet}="c" !{(4,0.6)}*+{{}^{v_3}}
!{(6,1)}*+{\bullet}="x" !{(6,0.6)}*+{{}^{v_4}}
!{(7,1)}*+{*}="d" !{(7.3,1.2)}*+{{}^{w}}
!{(6.3,1.1)}*+{{}^{1}} !{(5.7,1.1)}*+{{}^{2}}
!{(4.3,1.1)}*+{{}^{1}} !{(3.7,1.1)}*+{{}^{2}}
!{(2.3,1.1)}*+{{}^{2}} !{(1.7,1.1)}*+{{}^{1}}
!{(.3,1.1)}*+{{}^{1}} !{(-.3,1.1)}*+{{}^{2}}
"a"-"u" "a"-"b" "b"-"c" "c"-"x" "x"-"d"} 
\]
\caption{One possible WRFQ $G|\mathfrak{t}$ corresponding to a fine quotient 
$G\backslash\backslash\mathfrak{g}$ isomorphic to the fine graph
in Figure \ref{F2}. No virtual vertices are depicted except for the leaves at the half edges. If this picture represents a fine
quotient of a tree $\mathfrak{t}$ where
every vertex has valency $2$,
the weights indicate
that the only ramified (actual) vertices in the map $\mathfrak{t}^{[1]}\twoheadrightarrow
 G\backslash\backslash\mathfrak{g}$ are the 
 pre-images of $v_4$, where a single edge in Figure \ref{F2} corresponds to an edge having wheight $2$ in this picture.}\label{F3}
\end{figure}

\section{Statement of the main results}

In all of this section we let $X$, $K$, $\oink_X$, and $\alge=\matrici_2(K)$ be as defined at the beginning of 
\S1. Let $\Da$ be an
$X$-order of rank $4$ in a genus $\mathbb{O}$, and let $P\in\Pi_{\mathbb{O}}$, as defined in Equation
(\ref{piq}), so that $\Da_P$
is maximal. In this work, the Bruhat-Tits tree 
$\mathfrak{t}=\mathfrak{t}(P)=(V_\mathfrak{t},A_\mathfrak{t},r,s,t)$ 
is a graph whose vertex set $V_\mathfrak{t}$ is 
the set of local maximal orders in the completion $\alge_P$. 
Two such orders $\Da_P$
and $\Da'_P$ are neighbors if and only if, in some basis, they have the form
$$\Da_P=\bmattrix {\oink_P}{\oink_P}{\oink_P}{\oink_P}\quad\textnormal{and}\quad 
\Da'_P=\bmattrix {\oink_P}{\pi_P\oink_P}{\pi_P^{-1}\oink_P}{\oink_P},$$
for some local uniformizing parameter $\pi_P$. 
This is the same definition we used in \cite{Eichler2},
\cite{cqqgvro}, \cite{ruidqo}, \cite{scffgeo} or \cite{ogcseooff}.
Structurally, $\mathfrak{t}$  is a homogeneous tree having vertices of
valency  $p+1$, where $p=|\finitum(P)|$ is the cardinality of the residue field.
 The set $V_{\mathfrak{t}}$ can also be identified as the set of $P$-variant of a single $X$-order $\Da$ in $\alge$.
A $P$-variant of $\Da$ is an order $\Da'$ satisfying the following conditions:
\begin{enumerate}
\item $\Da_Q=\Da'_Q$ for any place $Q\in |X|\smallsetminus\{P\}$.
\item $\Da'_P$ is maximal.
\end{enumerate}
If $\Da_P$ and $\Da'_P$ are neighbors we say that $\Da$ and $\Da'$ are $P$-neighbors.
Any $P$-variant of $\Da$ is contained in the genus $\mathbb{O}$. The fact that every 
vertex corresponds to an $X$-order is a consequence of the fact that orders can be defined
locally. To make this statement precise, we fix an order $\Da$ and say that a family
$\{\Da'(Q)\}_Q$ is coherent whenever $\Da_Q=\Da'(Q)$  at all places $Q\in|X|\smallsetminus S$,
for some finite set $S\subseteq |X|$. This property is independent of the choice of $\Da$,
and the main property of coherent families reads as follows:
\begin{quote} The map sending every order $\Da'$ to the family $\{\Da'_Q\}_Q$ is a bijection
between the set of orders of maximal rank, and the coherent families of local orders of maximal rank.
\end{quote}
This tells us that we can define every maximal order by fixing one such order and computing successive
$P_i$ variants  at a finite family $\{P_1,\dots,P_N\}$ of places. Furthermore, at each step the
coordinate $\Da_P'$ is defined arbitrarily.
Another consequence of this fact is that, when $\Da'$
belong to the genus of $\Da$, the element $a_Q$ satisfying $\Da'_Q=a_Q\Da_Q a_Q^{-1}$ can be assumed 
to be trivial for almost every place $Q$. In particular, $a=(a_Q)_Q$ is an adele, and we write
$\Da'=a\Da a^{-1}$. Additionally, the conjugate $a\Da a^{-1}$ is a well defined lattice in the genus of $\Da$ 
for every adele $a$, so genera can be studied as adelic orbits 
of orders. Adeles $a=(a_Q)_Q$ satisfying $a_Q=\tilde a$ for 
every place $Q$ and a fixed element $\tilde a\in \alge$  
are called diagonal adeles or global adeles, and the
map $\tilde a\mapsto a$ 
is called the diagonal embedding. 
The orders $\Da$ and $a\Da a^{-1}$ in the same genus
are conjugate if the adelic element $a$ can be assumed to be
in the image of the diagonal embedding, i.e.,
if the coordinate $a_Q$ can be chosen independently of $Q$.
In all that follows, any element $\tilde{a}\in\alge$ is identified  with the corresponding
adele $a$ as above. 

Note that  two  $P$-variants $\Da,\Da'$ are conjugates if and 
only if there is an element $a$ in the normalizer 
$N\subseteq PGL_2(K)$ of  $\Da(U)=\Da'(U)$ satisfying $a\Da_Pa^{-1}=\Da'_P$. 
In particular, conjugacy classes of such orders are in correspondence with actual vertices of the fine quotient graph
$\mathfrak{c}_P(\Da)=N\backslash\backslash\mathfrak{t}$. In the nineteen seveties, J.-P. Serre gave a precise 
description of these  quotient graphs, or more precisely a closely related graph which we call the S-graph 
$\mathfrak{s}_P(\Da)$, whenever $\Da$ is maximal
 \cite{trees}. We rephrase his description as follows:
\begin{quote}
The quotient graph $\mathfrak{s}_P(\Da)=\Gamma\backslash\mathfrak{t}$,
 where $\Gamma=\Da(U)^*$ is the group of units,
 is a combinatorially finite (CF) graph,
i.e., the graph obtained by attaching a finite
number of cusps to a finite graph $Y$ as in
Figure \ref{sdoq}(A). Furthermore, the cusps 
are indexed by the Picard group of the 
Dedekind domain 
$A=\oink_X(U)$.
\end{quote}
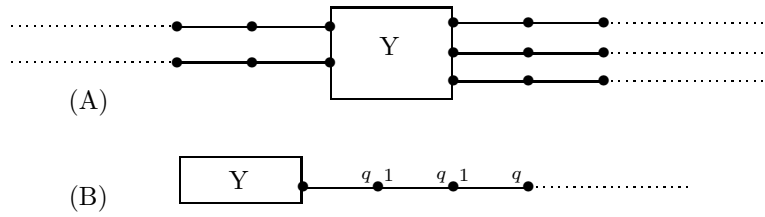
\begin{figure}
\unitlength 1mm 
\linethickness{0.4pt}
\ifx\plotpoint\undefined\newsavebox{\plotpoint}\fi 
\begin{picture}(80.25,17.5)(0,0)
\put(27.25,5.5){\framebox(16,12)[cc]{}}
\put(43.25,15.5){\line(1,0){20}}
\put(6.75,15){\line(1,0){20}}
\put(43.5,7.75){\line(1,0){20}}
\put(43.25,11.5){\line(1,0){20}}
\put(6.75,10.25){\line(1,0){20}}
\put(35,12.5){\makebox(0,0)[cc]{Y}}
\put(-5,5){\makebox(0,0)[cc]{(A)}}
\put(43.5,15.5){\makebox(0,0)[cc]{$\bullet$}}
\put(43.5,7.75){\makebox(0,0)[cc]{$\bullet$}}
\put(43.5,11.5){\makebox(0,0)[cc]{$\bullet$}}
\put(53.5,15.5){\makebox(0,0)[cc]{$\bullet$}}
\put(53.5,7.75){\makebox(0,0)[cc]{$\bullet$}}
\put(53.5,11.5){\makebox(0,0)[cc]{$\bullet$}}
\put(63.5,15.5){\makebox(0,0)[cc]{$\bullet$}}
\put(63.5,7.75){\makebox(0,0)[cc]{$\bullet$}}
\put(63.5,11.5){\makebox(0,0)[cc]{$\bullet$}}
\put(6.75,15){\makebox(0,0)[cc]{$\bullet$}}
\put(16.75,15){\makebox(0,0)[cc]{$\bullet$}}
\put(27.1,15){\makebox(0,0)[cc]{$\bullet$}}
\put(6.75,10.25){\makebox(0,0)[cc]{$\bullet$}}
\put(16.75,10.25){\makebox(0,0)[cc]{$\bullet$}}
\put(27.1,10.25){\makebox(0,0)[cc]{$\bullet$}}
\multiput(63.5,15.5)(1,0){22}{{\rule{.4pt}{.4pt}}}
\multiput(63.5,11.5)(1,0){22}{{\rule{.4pt}{.4pt}}}
\multiput(63.5,7.75)(1,0){22}{{\rule{.4pt}{.4pt}}}
\multiput(-15.25,15)(1,0){22}{{\rule{.4pt}{.4pt}}}
\multiput(-15.25,10.25)(1,0){22}{{\rule{.4pt}{.4pt}}}
\end{picture}
\unitlength 1mm 
\linethickness{0.4pt}
\ifx\plotpoint\undefined\newsavebox{\plotpoint}\fi 
\begin{picture}(80.25,17.5)(0,0)
\put(7.25,9.5){\framebox(16,6)[cc]{}}
\put(23.25,11.5){\line(1,0){30}}
\put(15,12.5){\makebox(0,0)[cc]{Y}}
\put(-5,10){\makebox(0,0)[cc]{(B)}}
\put(23.5,11.5){\makebox(0,0)[cc]{$\bullet$}}
\put(33.5,11.5){\makebox(0,0)[cc]{$\bullet$}}
\put(43.5,11.5){\makebox(0,0)[cc]{$\bullet$}}
\put(53.5,11.5){\makebox(0,0)[cc]{$\bullet$}}
\multiput(53.5,11.5)(1,0){22}{{\rule{.4pt}{.4pt}}}
\put(32,12.5){\makebox(0,0)[cc]{${}^{q}$}}
\put(35,12.5){\makebox(0,0)[cc]{${}^{1}$}}
\put(42,12.5){\makebox(0,0)[cc]{${}^{q}$}}
\put(45,12.5){\makebox(0,0)[cc]{${}^{1}$}}
\put(52,12.5){\makebox(0,0)[cc]{${}^{q}$}}
\end{picture}
\caption{In (A), an S-graph with $5$ cusps according to Serre's description.
In (B), a cusp in the WRFQ. Here $q=\sharp\mathbb{F}(P)$.}\label{sdoq}
\end{figure}
Note that $\Gamma K^*/K^*$ is a normal subgroup of $N$ of finite index, so that 
$\mathfrak{c}_P(\Da)$  is the fine quotient of $\mathfrak{s}_P(\Da)$
by the natural action of the quotient group $N/\Gamma$. If $\mathfrak{g}^{[1]}$ 
denotes the barycentric subdivision of $\mathfrak{g}$, the map
$\mathfrak{s}_P(\Da)^{[1]}\twoheadrightarrow \mathfrak{c}_P(\Da)$
 is a regular ramified covering.
In this context, the first result of the present work is an extension of Serre's result to all genera of orders
of maximal rank. It is unsurprising and probably known to experts. Furthermore, the particular case of
Eichler orders is discussed, in a slightly different language, in \cite{bravo1}, where a precise bound on the number
of cusps is given. We decided to include it here as an independent result, mainly because it is critical in what follows.  

\begin{theorem}\label{t1}
Let $\Ra$ be an order of maximal rank in 
$\alge$, and let $P$
be a place where $\Ra$ is maximal. 
Set $U=X-\{P\}$ and 
$\Gamma=\Ra(U)^*K^*/K^*$. Let $N$ be the
normalizer of the 
order $\Ra(U)$ in the projective general 
linear group $\mathrm{PGL}_2(K)$. 
  Then either quotient graph 
  $\Gamma\backslash\mathfrak{t}$ or 
  $N\backslash\backslash\mathfrak{t}$ is CF.
  Furthermore,
if  $\Ra$ is an Eichler order of 
multiplicity-free level the cusps are 
indexed by the Picard group of 
$A=\mathcal{O}_X(U)$,
and the set of orders in each cusp 
can be explicitly described
as split orders.
\end{theorem}

When seen in the WRFQ,  the cusps looks like 
the one in  Figure \ref{sdoq}(B). See Prop. 
\ref{p41} below for details.

It is customary to extend the classifying graph 
$\mathfrak{c}_P(\Ra)=
N\backslash\backslash\mathfrak{t}$ to
a non-connected graph 
$\mathfrak{c}_P(\mathbb{O})$, depending only on 
the genus $\mathbb{O}$ of $\Ra$,
defined as the coproduct in $\mathfrak{Graphs}$ 
of a finite number of graphs 
$\mathfrak{c}_P(\Ra_1),\dots,
\mathfrak{c}_P(\Ra_m)$, where 
$\Ra_1(U),\dots,\Ra_m(U)$ are
representatives of all conjugacy classes in the 
corresponding genus. See for example 
\cite{cqqgvro}.
 The advantage of this construction
is that set of actual vertices of the graph
$\mathfrak{c}_P(\mathbb{O})$ is precisely 
the set of
conjugacy classes of orders in $\mathbb{O}$, 
and therefore it is independent of the 
choice of the
place $P\in\Pi_{\mathbb{O}}$. Analogously, 
to state our main result in a simpler way, 
we define
$\mathfrak{wq}_P(\Ra)=N|\mathfrak{t}_P$, 
where $\mathfrak{t}_P$ is the Bruhat-Tits 
tree at $P$
and $N$ the normalizer of $\Ra(U)$
as above. Then we 
define the full WRFQ as the coproduct
$\mathfrak{wq}_P(\mathbb{O})=
\coprod_{i=1}^m\mathfrak{wq}_P(\Ra_i)$. 
These are weighted graphs, no graphs,
but extending the coproduct to this setting 
should cause no confusion. Now our main result 
is the following:

\begin{theorem}\label{t2b}
Let $\Ra$ be an order of maximal rank in 
$\alge$, and let $P$ and $Q$ be two places 
where $\Ra$ is maximal. 
Assume that we have the following  data:
\begin{itemize}
\item the full weighted quotient  $\mathfrak{wq}_P(\mathbb{O})$ and
\item the cusps of  $\mathfrak{wq}_Q(\mathbb{O})$.
\end{itemize}
Then, we can compute every edge, and the corresponding  weights,
of the graph
$\mathfrak{wq}_Q(\mathbb{O})$, outside a finite set of 
``bad'' pairs $(v,w)$ that depends
on  $\mathfrak{wq}_P(\mathbb{O})$, but not on $Q$.
\end{theorem}

A few comments are in order here, knowing the cusps of $\mathfrak{wq}_Q(\mathbb{O})$, do indeed mean knowing
every edge, and the corresponding weights, except for a finite 
set $S'$ of pairs $(v,w)$, but this set depends on
$Q$ and grows with the degree of $Q$, so having a uniform bound
on our ``hard to compute'' set is generally useful,
making more likely that we can finish the computation with 
some sort of ``ad-hoc' argument. This is illustrated by
the examples in \S8.
More importantly,
this set of bad pairs is quite explicit, and depends heavily on the geometry of  $\mathfrak{wq}_P(\mathbb{O})$,
so we can frequently prove it to be empty in explicit examples.

Another useful remark is that knowing the weighted graph  $\mathfrak{wq}_Q(\mathbb{O})$ is the
same as knowing the weights $m_{v,w}$ provided that we set $m_{v,w}=0$ when $w$ and $v$ are
not neighbors. The reduced graph is recovered by joining the 
pairs $(v,w)$ for which $m_{v,w}\neq0$.
For this reason it is useful to re-state our main result 
in terms of matrices, which requires that we define 
cusps in an infinite matrix.  As in \S1,
the neighborhood matrix is written as $N_P(\mathbb{O})=(m_{i,j})_{i,j=1}^\infty$,
where $m_{i,j}$ is the number of neighbors of an order $\Da_j$ that are conjugate
to $\Da_i$, for some sequence $\Da_0,\Da_1,\Da_2,\dots$ of representatives of
all conjugacy classes in $\mathbb{O}$. 

An infinite matrix $A=(a_{i,j})_{i,j=0}^\infty$ is said to be graphic if the relation
$a_{i,j}\neq0$ is symmetrical in $i$ and $j$. If this is the case, there is a reduced fine
graph whose edges are precisely the pairs $(\Da_i,\Da_j)$ for which
$a_{i,j}\neq0$, and we call it the associated graph for $A$. 
An infinite graphic matrix is CF if so is its associated graph.
In this case, there is a finite number of increasing sequences 
$\{r_{1,t}\}_{t=0}^\infty,\dots,\{r_{N,t}\}_{t=0}^\infty$
 of positive integers, called cusp-sequences, satisfying the following conditions:
\begin{enumerate}
\item The cusp-sequences are disjoint, i.e., $r_{i,t}\neq r_{j,s}$ whenever $i\neq j$.
\item If $t>0$ and $i=r_{m,t}$, then $a_{i,j}\neq0$ if and only if $j\in\{r_{m,t-1},r_{m,t+1}\}$.
\item Every positive integer, outside a finite set, is in the rank of some cusp-sequence.
\end{enumerate}
By a cusp of a CF matrix we mean one of these sequences $\{r_{i,t}\}_{t=0}^\infty$,
 together with the corresponding sequence of rows $\{R(r_{i,t})\}_{t=0}^\infty$
and columns $\{C(r_{i,t})\}_{t=0}^\infty$ of $A$,
so that the cusps of a CF matrix have all the information needed to describe it outside a finite set of
coordinates. The word ``explicitly'' in Theorem \ref{t1} refers to the fact that these rows and column are known
for some Eichler orders, as it will be clear on the examples. Next result is a re-statement of Theorem \ref{t2b}
in terms of infinite matrices. Recall that 
$\Pi_\mathbb{O}$ is the set of places of $K$ where the orders in $\mathbb{O}$ are maximal.

\begin{theorem}\label{t2}
Let $\mathbb{O}$ be a genus of orders of maximal rank, and let $P\in\Pi_{\mathbb{O}}$.
Then, there is an explicit finite set $T=T(P)\subseteq\mathbb{N}$ such that,
for any place $Q\in\Pi_{\mathbb{O}}$, 
every coefficient $m_{i,j}$ of the neighborhood matrix $N_Q(\mathbb{O})$
satisfying $(i,j)\notin T\times T$
can be computed from the following data:
\begin{enumerate}
\item The cusps of $N_Q(\mathbb{O})$.
\item The matrix $N=N_P(\mathbb{O})$.
\end{enumerate}
\end{theorem}

In fact, we prove that the set of  CF matrices commuting with $N$
that have a prescribed set of cusps lies on a coset of a finite dimensional vector
space whose dimension depends only on $N_P(\mathbb{O})$.
In \cite{cqqgvro} we managed to get a closed formula for all matrices $N_Q(\mathbb{O}_0)$,
when $\mathbb{O}_0$ is the genus of maximal orders on the projective line $X=\mathbb{P}^1$.
 This is in fact a consequence of the following result,
 which we prove in \S5:

\begin{theorem}\label{t3}
Let $\mathbb{O}$ be a genus of orders of maximal rank and let $\mathfrak{wq}=\mathfrak{wq}_P(\mathbb{O})$
its  corresponding full WRFQ at some place $P\in\Pi_{\mathbb{O}}$. Assume that each connected component of
the graph obtained by removing all half edges in $\mathfrak{wq}$ is a tree with at most one leaf. 
Then there is a unique CF matrix commuting with $N_P(\mathbb{O})$
with any prescribed set of cusps.
\end{theorem}

In other words, if the hypotheses are satisfied, the whole
neighborhood matrix $N_Q(\mathbb{O})$ can be computed from the cusps and $N_P(\mathbb{O})$
by applying the commuting relations.
We provide explicit examples of genera for which all neighborhood matrices can be computed from a unique
matrix by this procedure and one for which we have an actual obstruction. The preceding result indicates that
this fact is, sometimes, a consequence of the geometry of the graph. Indeed, we provide graphs $\mathfrak{g}$
such that any neighborhood matrix whose associated graph is $\mathfrak{g}$ would have an obstruction for 
this computation, and others for which the existence of such an obstruction depends on the matrix.  

Recall that the divisor valued distance between two maximal orders
$\Da$ and $\Da'$ is defined as 
$D=D(\Da,\Da')=\sum_{P\in|X|}d_P(\Da_P,\Da'_P)P$,
where $d_P$ is the usual graph distance as vertices of the
Bruhat-Tits tree at $P$. An Eichler order of level $D$
is, by definition, the intersection of two maximal orders
whose divisor valued distance is $D$. All Eichler orders
of level $D$ belong to a unique genus denoted $\mathbb{O}_D$.

Applying the commuting relation can be a long process in general,
but, for the particular case of Eichler orders,
we have a closed formula that allows us to compute
the WRFQ at a second place $Q$ whose image $\bar{Q}$
in the Picard group is a multiple of $\bar{P}$. 
It can be stated in terms of the polynomials $\tilde{f}_i$
defined by the following recursive relation:
$$ \tilde{f}_1(\mathbf{x},\mathbf{y})=\mathbf{x},
\qquad
\tilde{f}_2(\mathbf{x},\mathbf{y})=
\mathbf{x}^2-2\mathbf{y}, 
\qquad \tilde{f}_{n+2}(\mathbf{x},\mathbf{y})=
\mathbf{x}\tilde{f}_{n+1}(\mathbf{x},\mathbf{y})-
\mathbf{y}\tilde{f}_n(\mathbf{x},\mathbf{y}).$$
Now set $f_i(\mathbf{x})=
\tilde{f}_i(\mathbf{x},q)$, where 
$q=|\mathbb{F}(P)|$.

\begin{theorem}\label{t5}
Let $X$ be a curve and let $D$ be a multiplicity free divisor over $X$. Let $P$ and $Q$ be two places of $X$
satisfying the following conditions:
\begin{itemize}
\item Neither $P$ nor $Q$ is in the support of $D$.
\item $Q-nP$ is principal. 
\end{itemize}
Then the matrix 
$f_n\big(N_P(\mathbb{O}_D)\big)-N_Q(\mathbb{O}_D)$
vanishes, except maybe in an
explicit finite dimensional invariant subspace 
depending only on $P$. Furthermore, if 
$\mathfrak{wq}=\mathfrak{wq}_P(\mathbb{O}_D)$ satisfies 
the hypotheses of Theorem \ref{t3}, then
$N_Q(\mathbb{O}_D)=f_n\big(N_P(\mathbb{O}_D)\big)$.
\end{theorem}

We show in the examples that the hypotheses for the
last conclusion of the preceding result cannot
be fully removed, as there are indeed cases where
$f_n\big(N_P(\mathbb{O}_D)\big)-N_Q(\mathbb{O}_D)$
is non-zero.

Note that a different recurrence for the polynomials $f_i$
was given in \cite{cqqgvro}. The one presented here is simpler and
more natural from the condition $f_n(\mathbf{x}+q\mathbf{x}^{-1})=
\mathbf{x}^n+q^n\mathbf{x}^{-n}$,
which is what it is actually used, see \S\ref{sec5}.

The last two sections deal with explicit computations of examples.
In \S7 we compute, by an ad-hoc method, a few particular quotient 
graphs to which the results of this work can be applied. In \S8
we apply our reults to these graphs to actually obtain
some weighted reduced quotients at places of higher degree.
An aditional example from the literature is used to illustrate a case
where actual barriers to our method do exist.

\section{quotient graphs for orders of maximal rank}

The purpose of this section is to prove Theorem \ref{t1}. We need to recall a few facts from
the Theory of Spinor Genera. For a more detailed account on this topic see \cite{abelianos}.

Let $\mathbb{O}$ a genus of orders of maximal rank. We define the spinor class field 
$\Sigma=\Sigma(\mathbb{O})$ as the class field corresponding to the group $K^*H(\Da)$,
where $H(\Da)\subseteq J_X$ is the group of reduced norms of adeles $a\in\alge^*_\ad$ satisfying
$a\Da a^{-1}=\Da$ for some fixed order $\Da\in\mathbb{O}$. Note that $\Sigma$ is independent of the choice
of $\Da$ since the reduced norm is invariant under conjugation. There exists, furthermore,
 a well defined distance map
$\rho:\mathbb{O}^2=\mathbb{O}\times\mathbb{O}\rightarrow\mathrm{Gal}(\Sigma/K)$ satisfying the identities:
\begin{itemize}
\item $\rho(\Da,\Da'')=\rho(\Da,\Da')\rho(\Da',\Da'')$, for any triplet
$(\Da,\Da',\Da'')\in\mathbb{O}^3$,  and 
\item $\rho(\Da,a\Da a^{-1})=[n(a),\Sigma/K]$, where $x\mapsto[x,\Sigma/K]$ 
denotes the Artin map on ideles, and where 
$n:\alge_{\mathbb{A}}^*\rightarrow J_X$ denotes 
the reduced norm.
\end{itemize}
 Two orders $\Da$ and $\Da'$ are in the same 
spinor genus whenever any of the following equivalent conditions hold:
\begin{itemize}
\item $\Da'=a\Da a^{-1}$, 
for $a=bc$ with $c\in\alge$ and $n(b)=1$, 
\item $\Da'(U)$ and $\Da(U)$ are conjugate for any affine open set $U\subseteq X$, or
\item the automorphism $\rho(\Da,\Da')$ is the identity in $\mathrm{Gal}(\Sigma/K)$. 
\end{itemize}
It is easy to see from the preceding definitions that whenever two orders $\Da$ and
$\Da'$ are $P$-neighbors, their distance is
$$\rho(\Da,\Da')=[e(P),\Sigma/K],\textnormal{ where }e(P)_Q=\left\{\begin{array}{ccl}
1&\textnormal{if}& P\neq Q\\
\pi_P&\textnormal{if}& P=Q\end{array}\right.$$
and $\pi_P$ is a local uniformizer. 
This map is, by definition, the Frobenius homomorphism 
at the place $P$.
Recall that the $P$-variants of a fixed order $\Da$ correspond to
vertices of the Bruhat-Tits at $P$ via $\Da'\mapsto\Da'_P$.
In particular, the global orders $\Da$ with a fixed restriction
$\Da(U)$ belong to either one or two spinor genera,
according to whether $[e(P),\Sigma/K]$ is the identity or not. In the latter case the
classifying graph is  bipartite, as
neighboring (actual) vertices belong
to different spinor genera. In particular,
it has no half edges.

We need also to recall the cusp structure for the classifying
graphs $\mathfrak{c}_P(\mathbb{O}_0)$ of the genus
of maximal orders, which we re-state here in our notations. 

\begin{proposition}\label{p41}
For any smooth curve $X$, each cusp of $\mathfrak{wq}_P(\mathbb{O}_0)$ has the form shown
in Figure \ref{nf3}, where
\begin{figure}
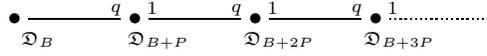

\[  \xygraph{
!{<0cm,0cm>;<.8cm,0cm>:<0cm,.8cm>::}
!{(0,1)}*+{\bullet}="a" !{(0.4,0.6)}*+{{}^{\Da_B}}="a1"
!{(2,1)}*+{\bullet}="b" !{(2.4,0.6)}*+{{}^{\Da_{B+P}}}="b1"
!{(4,1)}*+{\bullet}="c" !{(4.4,0.6)}*+{{}^{\Da_{B+2P}}}="c1"
!{(6,1)}*+{\bullet}="x" !{(6.4,0.6)}*+{{}^{\Da_{B+3P}}}="x1"
!{(8,1)}*+{}="d"
!{(1.7,1.1)}*+{{}^q} !{(2.3,1.1)}*+{{}^1}
!{(3.7,1.1)}*+{{}^q} !{(4.3,1.1)}*+{{}^1}
!{(5.7,1.1)}*+{{}^q} !{(6.3,1.1)}*+{{}^1}
 "a"-"b" "b"-"c" "c"-"x" "x"-@{.}"d"} \]
\caption{Structure of a cusp in a classifying graph of maximal orders.}\label{nf3}
\end{figure}
$\Da_C=\bmattrix{\oink_X}{\mathcal{L}^C}{\mathcal{L}^{-C}}{\oink_X}$,
and $\mathcal{L}^C$ is the line bundle defined at any affine open set
$U\subseteq X$ by the identity $$\mathcal{L}^C(U)=
\left\{f\in K\Big|\mathop{\mathrm{div}}(f)|_U+C|_U\geq0\right\}.$$
These maximal orders are called split.
\end{proposition} 

\begin{proof}
Note that the weights in the picture show that
each neighbor of the vertex $\Da_{B+tP}$, for $t\geq 1$,
is conjugate to $\Da_{B+(t-1)P}$ except for $\Da_{B+(t+1)P}$.
It is immediate that successive orders in  Figure \ref{nf3} 
are $P$-neighbors. 
Assume first that $\deg(B)\geq 0$.  Then the divisors of the form $B+tP$, for $t\geq 1$ all have positive degree.
It is known from \cite[Prop. 4.1]{cqqgvro}, that $\Da_B$ can be a conjugate of $\Da_C$ only if $C$ is linearly equivalent 
to either $B$ or $-B$. In particular, no two vertices in   Figure \ref{nf3} are conjugates.
Note that the group of units $\Da_{B+tP}(X)^*$ acts
transitively on the set of lattices of the form
$$\Lambda_v=
\oink_X\left(\begin{array}{c}v\\
1\end{array}\right)+ 
\mathcal{L}^{B+(t-1)P}
\left(\begin{array}{c}1\\
0\end{array}\right),$$
for $v\in\mathcal{L}^{B+tP}(X)$. The sheaves 
of endomorphisms of these lattices
are $P$-neighbors of $\Da_{B+tP}$.
Riemann-Roch Theorem implies that 
$\mathcal{L}^{B+tP}(X)/
\mathcal{L}^{B+(t-1)P}(X)$ is an 
$\mathbb{F}$-vector
space of dimension $\deg(P)$, for $t$ 
large enough. Then, there are sufficiently 
many of these
lattices $\Lambda_v$ to obtain all 
$|\mathbb{F}(P)|=q^{\deg(P)}$ neighbors 
of $\Da_{B+tP}$
apart from $\Da_{B+(t+1)P}$. This proves that
we obtain a cusp of the required form in 
this case.

The preceding argument shows that every split order of the form $\Da_C$ for a divisor $C$ of large enough
degree belong to a cusp of the form already described. For split orders of the form $\Da_B$ with
$\deg(B)< 0$, we have the isomorphism $\Da_B\cong\Da_{-B}$, so in particular, the same holds for 
orders of the form $\Da_C$ for a divisor $C$ of sufficiently negative degree. Since $\mathrm{Pic}_0(X)$ is finite,
as it is isomorphic to the set of $\mathbb{F}$-points of an abelian variety, all split orders except for a finite number
of conjugacy classes belong to a cusp of the type described above. 
Now \cite[Lem. 7.5]{ogcseooff} guarantees that every cusp consist of split orders, so there are no additional cusps.
\end{proof}

The weights in Figure \ref{nf3} can be reinterpreted in
terms of matrices as follows:
If we write
$N_P(\mathbb{O})=(m[\Da,\Da'])_{\Da,\Da'}$, then
\begin{equation}\label{r33}
m[\Da_{B+tP},\Da_{B+(t+1)P}]=1,\quad
m[\Da_{B+tP},\Da_{B+(t-1)P}]=|\finitum(P)|=q^{\deg(P)}
\end{equation}
 for large values of the degree $\deg(B)$. 
 We certainly can assume this to hold for $t\geq1$ by 
replacing the
initial point $\Da_B$ by another vertex $\Da_{B+nP}$
for a large enough value of $n$. Such order is already
in the cusp in Figure \ref{nf3}, so this replacement makes 
indeed the cusp smaller. We assume these
conventions in all that follows. In fact, Riemann-Roch
Theorem gives a precise bound for the required 
degree
of $B$ so that $\Da_B$ is in a cusp, but we do 
not need it here.

Before we prove  Theorem \ref{t1}, we need a 
lemma on finite coverings of graphs.
To simplify the statement, we say that a cover 
$\gamma:\mathfrak{g}\rightarrow\mathfrak{g}'$
is smooth on cusps if no cusp of 
$\mathfrak{g}'$ can have more than one lifting
in $\mathfrak{g}$ with a given initial vertex.
This condition is immediate for a cover of 
the form $\mathfrak{t}_P^{[1]}\rightarrow
\mathfrak{c}_P(\Da)$,
if the initial point $\Da_B$ is set far enough, 
from the condition (\ref{r33}) above. 
It follows that it also 
applies to intermediate covers.

\begin{lemma}\label{l41p}
Let $\gamma:\mathfrak{g}\rightarrow\mathfrak{g}'$ be a finite cover of graphs that is smooth on cusps. 
Then only a finite number of points in any lifting of a cusp $\mathfrak{c}'$ can be ramified.
\end{lemma}

\begin{proof}
Let $v'$ be the initial vertex of the cusp.
 Consider the pre-images $v_1,\dots,v_n$ of $v'$
in $\mathfrak{g}$, and for each of these vertices consider the corresponding lifting
$\mathfrak{c}_i$  of $\mathfrak{c}'$.
 If any non-initial vertex $w$ in one of these rays is a ramification
point  for the cover,
then two of these rays must coincide at $w$, and therefore at every subsequent point, as in Figure 5A,
or the covering would not be smooth on cusps. 
Furthermore, the vertices preceding $w$ in 
either cusp are different.
This can happen only once for each pair 
of liftings $(\mathfrak{c}_i,\mathfrak{c}_j)$, whence the result follows.
\begin{figure}
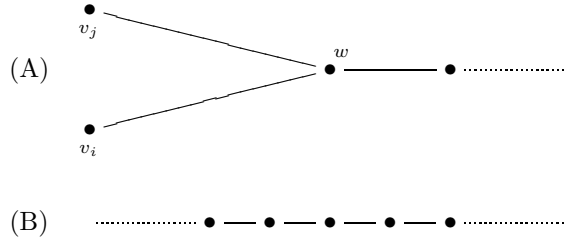

\[  \xygraph{
!{<0cm,0cm>;<.8cm,0cm>:<0cm,.8cm>::}
!{(0,1)}*+{\bullet}="a" !{(0,0.6)}*+{{}^{v_i}}="a1"
!{(0,3)}*+{\bullet}="b" !{(0,2.6)}*+{{}^{v_j}}="b1"
!{(4,2)}*+{\bullet}="c" !{(4.2,2.2)}*+{{}^{w}}="c1"
!{(6,2)}*+{\bullet}="x" 
!{(-1,2)}*+{\textnormal{(A)}}="name"
!{(8,2)}*+{}="d"
 "a"-"c" "b"-"c" "c"-"x" "x"-@{.}"d"} \]
\[  \xygraph{
!{<0cm,0cm>;<.8cm,0cm>:<0cm,.8cm>::}
!{(3,1)}*+{\bullet}="a" 
!{(5,1)}*+{\bullet}="b" 
!{(4,1)}*+{\bullet}="c" !{(6,0.6)}*+{}="c1" !{(2,1.4)}*+{}="c2" 
!{(6,1)}*+{\bullet}="x" 
!{(2,1)}*+{\bullet}="y" 
!{(8,1)}*+{}="d"
!{(0,1)}*+{}="e"
!{(-1,1)}*+{\textnormal{(B)}}="name"
 "a"-"c" "b"-"c" "b"-"x" "a"-"y" "x"-@{.}"d" "y"-@{.}"e"
} \]
\caption{Two merging rays, as in the proof of Lemma \ref{l41p} (A), and the line mentioned in
the proof of Theorem \ref{t1} (B).}\label{fi6}
\end{figure}
\end{proof}

The preceding result shows that some final section of any cusp is covered by
a finite number of isomorphic copies of itself. We might express this fact by saying
that finite coverings that are smooths on cusps are trivial near the end of a cusp.

$\phantom{X}$

\paragraph{Proof of Theorem \ref{t1}}
 Let $\Ra\in\mathbb{O}$, and let $\Da$ be
a maximal order containing $\Ra$. It is well known that the group of units of an order
is open in the adelic topology. In other words,
if $A=\oink_X(U)$ for $U=X-\{P\}$, there is an ideal
$J\subseteq A$ such that:
$$\Gamma_J=\left\{a\in\Da(U)^*\Big|a\equiv\bmattrix 1001 \ \mod J\Da(U)\right\}\subseteq \Ra(U)^*.$$
Recall that $\mathfrak{s}_P(\Da)=\Da(U)^*\backslash \mathfrak{t}$ is CF, 
as shown by Serre in \cite[Theorem II.9]{trees}.
The fine quotient is also CF, if the definition of cusp and CF
graph for fine quotients follows our usual conventions,
as we assume in all that follows.
We make two claims:
\begin{enumerate}
\item If  $\Delta'$ is any subgroup of finite index in
$\Da(U)^*$, then $\Delta'\backslash\backslash \mathfrak{t}$ is CF.
\item If $\Delta'\backslash\backslash \mathfrak{t}$ is CF,  
and if it $\Delta'$ is a normal subgroup of finite index in
$\Delta$, then $\Delta\backslash\backslash \mathfrak{t}$ is CF.
\end{enumerate}
The results follows from the claims by the following observations:
\begin{itemize}
\item  The subgroup $\Gamma_J$ is normal of  finite index in either $\Da(U)^*$ or $\Ra(U)^*$, and
\item the subgroup $\Ra(U)^*$ is normal of  finite index in the normalizer of the order $\Ra(U)$.
\end{itemize}

 It follows from the previous lemma, and the discussion preceding it, that any cusp
has a sub-ray whose pre-image is the coproduct of a finite number of rays, 
while the pre-image of the rest of the graph is finite. This proves the first claim.
To prove the second claim we look at the action of the finite group $\Psi=\Delta/\Delta'$ on 
the quotient graph $\Delta'\backslash \mathfrak{t}$. 
By redefining a cusp, we mean replacing a cusp by a proper subray.
 We claim that, by  redefining the cusps if needed, we can assume that the following 
statement holds:
\begin{quote}
For any element $\psi\in\Psi$ and any cusp $\mathfrak{c}$, the graph $\psi.\mathfrak{c}$ is a cusp.
\end{quote}
This can be done choosing a cusp $\mathfrak{c}$, with vertices $v_0,v_1,v_2,\dots$,
in order,  and redefining $\psi.\mathfrak{c}$ as a cusp
for any element $\psi\in \Psi$. This could fail if $\mathfrak{c}$ and $\psi.\mathfrak{c}$ intersect non-trivially.
If such intersection is finite, we can avoid it by redefining the cusps. In fact, this can happen only
for a $(\mathbb{Z}/2\mathbb{Z})$-action on a line, which is the graph depicted in Figure \ref{fi6}(B),
 but we do not need this fact here. When the intersection
is infinite, we have a relation of the form $\psi(v_k)=v_{k+t}$, for some integer $t$ and any large enough
integer $k$. By replacing $\psi$ by $\psi^{-1}$ if needed, we can assume $t\geq0$. This implies that
$\psi^n(v_k)=v_{k+nt}$ which, given the fact that $\psi$ belongs to a finite group, implies that $t=0$. 
Hence $\psi(v_k)=v_k$ for $k$ large enough, and we assume it to hold for all $k$ by redefining the cusps.
It follows that every ray in the orbit of  $\mathfrak{c}$ is a cusp. 
 Now we  can repeat this argument for the remaining cusps. This proves the claim.
Since $\Psi$ simply permutes a set of cusps, the result follows.

As for the cusps of the classifying graph for a genus of multiplicity-free level Eichler orders to be explicit,
recall that the maximal orders containing an Eichler order of the form $\Ea_B=\Da_B\cap\Da_{B+D}$,
for an effective divisor $D$, have the form $\Da_{B+D'}$ for $0\leq D'\leq D$. Riemann-Roch' Theorem
implies that, for divisors $B$ of large degree, $\Da_B$ can be characterized as the maximal order
containing $\Ea_B$ with the smallest ring of global sections. This implies that $\Da_B\mapsto\Ea_B$
is an injective map preserving neighbors, near the end of a cusp,
in the sense defined in the paragraph preceding the proof. 
It follows from
 \cite[Prop. 7.6]{ogcseooff} that, when $D$ is multiplicity free, 
 every cusp (except for a finite initial sub-path) is in the 
 image of such map, so the
description of cusps can be imported from the classifying 
graph of maximal orders to the classifying
graph of Eichler orders of level $D$.
\qed

\section{On obstruction sets.}\label{sec5}

At this point, we need to introduce a more general notation
that will be used in the sequel.
By a weighted CF graph $\mathfrak{wg}$,  or a WCFG, we mean a 
reduced fine graph $\mathfrak{g}$ 
that can be obtained by attaching
a finite number of cusps (rays)  to a finite graph $\mathfrak{z}$ 
called the center, together with an assignation
of weights $m_{v,w}\in\mathbb{R}$ for any pair of 
neighboring vertices 
$v,w\in V=V(\mathfrak{g})$.
If the WCFG is defined from a quotient of the Bruhat-Tits 
tree at a place $P$, then  the weights are non-negative integers,
and we have $\sum_wm_{v,w}=
|\mathbb{F}(P)|+1=q^{\deg(P)}+1$ for every vertex $v\in V$, 
but we do not assume this in the sequel.
The reason is that we want to be able to draw WCFGs associated to more 
general matrices. In general the graph $\mathfrak{g}$ need not be connected. 
As usual we set $m_{v,w}=0$ if $v$ and $w$ are not neighbors.
We say that $\mathfrak{wg}$ is regular if $m_{v,w}>0$ for every
pair of neighbors $(v,w)$, and we call it
normalized when $\sum_wm_{v,w}=1$ 
for every vertex $v$. Every regular WCFG can be normalized, replacing $m_{v,w}$ by 
$a_{v,w}=m_{v,w}/\left(\sum_wm_{v,w}\right)$. When the WCFG has the 
form $\mathfrak{wq}_P(\mathbb{O})$,
then this relation becomes $a_{v,w}=m_{v,w}/\left(q^{\deg(P)}+1\right)$. 
We denote the weights as
 $a_{v,w}$ whenever the WCFG is regular and normalized to avoid confusion.
By the neighborhood matrix of the WCFG, we mean the infinite matrix $N=(m_{v,w})_{v,w}$. We assume that
all graph considered here are locally finite, so that products of such matrices 
are well defined.
When the WCFG is regular and normalized, the matrix is written 
$T=(a_{v,w})_{v,w\in V}$. Since 
$\sum_wa_{v,w}=1$, the infinite matrix $T$ is the transition matrix 
of a Markov chain \cite{wiki1}. 
Markov Chains defined on quotient graphs can be used to 
study Diophantine approximation on local fields, see \cite{dalffvbtt}, but 
here they are just a tool to help intuition and simplify notations.

By a charge, we
mean a finitely supported signed measure $\mu:\powerset(V)\rightarrow\mathbb{R}$, where $\powerset(V)$
is the collection of all subsets of $V$. 
The set of all charges is a vector space $C_V$, which we call
the charge space for $V$. We denote by $\mathbf{T}$ the linear map
 $\mathbf{T}:C_V\rightarrow C_V$ defined by
 $$\mathbf{T}\mu(A)=
\sum_{x\in V}\sum_{y\in A}a_{x,y}\mu\big(\{x\}\big).$$
This formula defines a linear map for every WCFG $\mathfrak{wg}$,
but when $\mathfrak{wg}$ is regular and normalized some additional
language of Markov chain theory can be applied.
For instance, assume we have a probability space 
$(\Omega,\mathfrak{F},\mathbf{P})$, where $\Omega$ is a set,  
$\mathfrak{F}\subseteq\powerset(\Omega)$ is a $\sigma$-algebra 
and $\mathbf{P}:\mathfrak{F}\rightarrow\mathbb{R}$
is a probability measure. Assume further the existence of a sequence 
of random elements 
$X_n:\Omega\rightarrow V=V(\mathfrak{g})$,
whose associated conditional probabilities 
satisfy the following conditions:
\begin{itemize}
\item  $\mu_0(A)=\mathbf{P}\Big(X_0\in A\Big)$ defines a probability 
measure $\mu_0\in C_V$, i.e.,
$\mu_0$ is finitely supported.
\item $\mathbf{P}\Big(X_{n+1}=w\Big|X_n=v\Big)=a_{v,w}=\mathbf{P}\Big(X_{n+1}=w\Big|X_n=v,X_m=u\Big)$,
for any triplet of vertices $(v,w,u)$ and for any pair of positive integers $(n,m)$ with $m<n$.
\end{itemize}
Then $\mathbf{T}\mu_0(A)=\sum_{x\in V}\mathbf{P}\Big(X_1\in A\Big|X_0=x\Big)\mathbf{P}\Big(X_0=x\Big)
=\mathbf{P}\Big(X_1\in A\Big)$.
By iteration we obtain $\mathbf{T}^n\mu_0(A)=\mathbf{P}\Big(X_n\in A\Big)$.
Note that we have $\mathbf{P}\Big(X_{n+1}=w\Big|X_n=v\Big)>0$ precisely 
when $w$ and $v$ are neighbors. 

For any charge $\mu\in C_V$, we write $\mathrm{Supp}(\mu)$ for the support of $\mu$,
 i.e., $\mathrm{Supp}(\mu)=\{v\in V|\mu(\{v\})\neq0\}$. 
 An important 
 but straightforward property of the map $\mathbf{T}$, that is critical 
 in what follows, is given in Lemma \ref{l51p} below:
\begin{lemma}\label{l51p}
Let $\mathbf{T}$ and $\mu$ be as above. Assume that there is a 
cusp $\mathfrak{d}$ in $\mathfrak{g}$ such that
$V(\mathfrak{d})\cap\mathrm{Supp}(\mu)\neq\varnothing$.
 Then $$V(\mathfrak{d})\cap\left(\bigcup_{k=0}^\infty\mathrm{Supp}(\mathbf{T}^k\mu)\right)$$ is unbounded.
\end{lemma}

\begin{proof}
Note that the vertex set of a cusp has a natural order,
 where the vertex following $v\in V(\mathfrak{d})$ 
 is the neighbor of $v$
 that remains in the connected component containing the visual limit of $\mathfrak{d}$ when $v$ is removed.
By hypothesis, the set $V'$ of vertices $v\in V(\mathfrak{d})$ 
 satisfying $\mu(\{v\})\neq0$ is not empty.
 Since $\mu$ is finitely supported, the set $V'$ has a last
 element, with respect to the natural order.
 Call this last element $v_0$.
Let $v'$ be vertex following $v_0$ in the cusp,
and let $v''$ the vertex following $v'$.
Then
\[
\mathbf{T}\mu(\{v'\})=a_{v_0,v'}\mu(\{v_0\})+
a_{v'',v'}\mu(\{v''\})=a_{v_0,v'}\mu(v_0)\neq0,
\]
as the transitional probability between neighbors is non-zero by definition
of regular WCFG.
\end{proof}

For any subset $Y\subseteq V$ we denote by $C_Y$ the set of charges supported on $Y$.  By a shell for
a map $\mathbf{F}:C_V\rightarrow C_V$, we mean a set $W\subseteq V$ satisfying the following conditions:
\begin{enumerate}
\item $\mathbf{F}(C_W)\subseteq C_W$ and
\item $\mathbf{F}(C_{W^c})=0$.
\end{enumerate}

\begin{lemma}\label{l52}
Let $\mathbf{T}$, $\mathfrak{g}$ and $\mu$ be as before, and let $\mathbf{F}:C_V\rightarrow C_V$ be any map satisfying 
the following conditions:
 \begin{enumerate}
\item $\mathbf{F}$ and $\mathbf{T}$ commute.
\item $\mathbf{F}$ has a finite shell.
\end{enumerate}
Let $W\subseteq V$
be a minimal shell for $\mathbf{F}$. Then 
$W\cap V(\mathfrak{d})=\varnothing$ for every cusp 
$\mathfrak{d}$ of $\mathfrak{g}$. 
\end{lemma}

\begin{proof}
 Let $\mathfrak{d}=\gamma(\mathfrak{N})$ be a cusp, where $\gamma:\mathfrak{N}\rightarrow\mathfrak{g}$
 is the associated ray.
We write $v_r=\gamma(n_r)$ for simplicity.
Assume $V(\mathfrak{d})\cap W\neq\varnothing$. 
  Since $W$ is finite, there is a last vertex
 $v_t$ in this intersection. 
We claim that $W'=W-\{v_t\}$ is also a shell for $\mathbf{F}$. This contradiction proves the lemma. 

To prove the claim we observe that, if $\mu_0=\delta_i\in C_V$ denotes the only probability measure supported
on $\{v_i\}$, and if $t\geq0$, then we have
\begin{equation}\label{probstep}
    \mathbf{T}(\delta_{t+1})=p\delta_t+q\delta_{t+2},
\end{equation}
 where $p$ and $q$ denote the following conditional 
probabilities:
 $$p=\mathbf{P}\Big(X_{n+1}=v_t\Big|X_n=v_{t+1}\Big), \qquad 
q=\mathbf{P}\Big(X_{n+1}=v_{t+2}\Big|X_n=v_{t+1}\Big).$$
Note that they are non-zero, since the WCFG is 
assumed to be regular.
Since $\mathbf{F}(\delta_{t+1})=\mathbf{F}(\delta_{t+2})=0$, it follows by 
applying $\mathbf{F}$ to the identity (\ref{probstep}) that
$\mathbf{F}(\delta_t)=0$. Since $C_{(W')^c}=C_{W^c}\oplus\mathbb{R}\delta_t$, 
we conclude that $W'$ satisfies the second condition in the definition of shell.
 Similarly, write $\mathbf{T}(\delta_t)=q'\delta_{t+1}+\nu$, where $\nu$ is a multiple of $\delta_{t-1}$
for $t>0$, and it is a measure supported on $ V(\mathfrak{d})^c\cup\{v_0\}$ when $t=0$.
Note that $q'\neq0$ by regularity. Now take any measure
$\mu\in C_{W'}$ and write $F(\mu)=\mu_0+s\delta_t$, 
where $\mu_0$ is supported on $W'$. This can be done since
$W$ is a shell. Then
$$\mathbf{F}\Big(\mathbf{T}(\mu)\Big)=
\mathbf{T}(\mu_0)+s\mathbf{T}(\delta_t)
=\mathbf{T}(\mu_0)+s\nu+sq'\delta_{t+1}.$$
Since $\mu$ is supported on $W'$, then $\mathbf{T}(\mu)$ is supported on $W$.
This contradicts the fact that $W$ is a shell, unless $s=0$.
The first condition in the definition of shell follows
for $W'$. This concludes the proof.
\end{proof}

Let $\mathfrak{wg}$ be a normalized regular WCFG
whose associated graph is $\mathfrak{g}$, and let $T$ be 
the corresponding transition matrix.
The obstruction space for $T$ is the vector space
defined by 
$$O_T=\left\{\mathbf{F}:C_V\rightarrow C_V| \mathbf{T}\mathbf{F}=\mathbf{F}\mathbf{T},
\ \mathbf{F}\textnormal{ has a finite shell}\right\}.$$
If we want to emphasize the map  $\mathbf{T}$ instead of the matrix $T$,
we write $O_\mathbf{T}$ instead.
A shell for $O_T$ is a set $W\subseteq V$ that is a shell for every element in $O_T$. A shell for
$O_T$ can be constructed as a union of minimal shells for its elements, whence next result follows:

 \begin{lemma}\label{l54}
The obstruction space $O_T$ has a shell $W$ intersecting trivially every
cusp of the graph. In particular, $W$ is finite.\qed
\end{lemma}

\begin{example}\label{ex1}
If $\mathfrak{g}$ is a cusp, then the empty set $W=\varnothing$ is a 
shell for the corresponding 
obstruction space $O_T$. We conclude that $O_T=\{0\}$.
This is the case, for example, when $\mathfrak{wg}=
\mathfrak{wq}_P(\mathbb{O})$  for the genus 
$\mathbb{O}=\mathbb{O}_0(\mathbb{P}^1)$ of maximal orders in the projective line.
See \cite[Th. 1.3]{cqqgvro} and \cite[\S II.2.4.1]{trees}.
\end{example}

\paragraph{Proof of Theorem \ref{t2}}
Let $\mathfrak{g}$ be the normalization of  $\mathfrak{wq}_P(\mathbb{O})$. Then 
$T=T_P(\mathbb{O})=\frac1{1+p}N_P(\mathbb{O})$ is the corresponding transition matrix. 
Let $N'_Q$ and $N''_Q$ be two possible candidates for the neighborhood matrix
$N_Q(\mathbb{O})$, i.e., two matrices having the right cusps, and commuting with $N_P(\mathbb{O})$.
Then their diference $F=N''_Q-N'_Q$ defines a map that belongs to the obstruction space $O_T$. 
It follows from Lemma \ref{l54} that both candidates
coincide in every matrix coordinate except maybe those corresponding to pairs of points
 in a minimal shell for $O_T$. The result follows.
\qed
 \vspace{2mm}
 
In order to prove Theorem \ref{t3}, we need a modified version of Lemma \ref{l54} that allows us to
ignore half edges. By a generalized cusp, or GC, we mean any graph obtained by attaching a finite number of 
half edges to the standard cusp.
 We add a half edge
to a cusp at a vertex $v_i$
 by forming the fibered co-product of the cusp and the 
 standard half edge, i.e., the graph shown
in Figure \ref{fi7}(A), by mean of the maps $\gamma$ and $\delta$
satisfying $\gamma_V(p)=v_i$, $\delta_V(p)=c$. 
This is similar to the process of adding cusps to a graph. 
See \S2 for the details. Generalized cusps
can be attached to the graph by exactly the same procedure as usual cusps. In fact, they have the same actual vertices
as a usual cusp. Next result is proven exactly as
Lemma \ref{l54}.

\begin{lemma}\label{l53}
Let $W\subseteq V$
be a minimal shell for $O_T$. Then $W$ cannot intersect non-trivialy any
GC in the graph. \qed
\end{lemma}

\begin{figure}
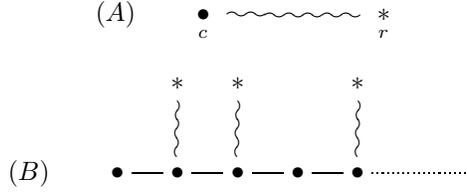

\[  (A) \qquad\xygraph{
!{<0cm,0cm>;<.8cm,0cm>:<0cm,.8cm>::}
!{(0,0)}*+{\bullet}="a" !{(0,-0.4)}*+{{}^{c}}="a1"
!{(3,0)}*+{*}="b" !{(3,-0.4)}*+{{}^{r}}="b1"
 "a"-@{~}"b"} \]
\[  (B) \qquad  \xygraph{
!{<0cm,0cm>;<.8cm,0cm>:<0cm,.8cm>::}
!{(3,0)}*+{\bullet}="a" !{(3,1.5)}*+{*}="a1"
!{(5,0)}*+{\bullet}="b" 
!{(4,0)}*+{\bullet}="c"  !{(4,1.5)}*+{*}="c1"
!{(6,0)}*+{\bullet}="x" !{(6,1.5)}*+{*}="x1"
!{(8,0)}*+{}="d"
!{(2,0)}*+{\bullet}="y"
 "a"-"c" "b"-"c" "b"-"x" "a"-"y" "x"-@{.}"d" 
"a"-@{~}"a1" "c"-@{~}"c1" "x"-@{~}"x1"
} \]
\caption{(A)The standard half edge. The half edge is denoted by a wiggly line 
and the virtual vertex by an asterisk. (B) Example of a cusp with three
half edges.}\label{fi7}
\end{figure}

\paragraph{Proof of Theorem \ref{t3}}
We say a vertex $v$ is an essential node if it satisfies $v=s(a)$ for 3 or more edges $a$, none of which is a half-edge.
The number of essential nodes is finite.
Reasoning as in the proof of Theorem \ref{t2}, 
it suffices to show that the empty set is a shell for  $O_T$. 
Let $\mathfrak{wg}$ be a connected component of  $\mathfrak{wq}_P(\mathbb{O})$. 
Let $v_f$ the only leaf of the graph obtained
by removing all half edges.  When there is no leaf $v_f$
 we can use instead an arbitrary (non-virtual) vertex in the tree that is not part
 of any cusp.
Since the associated graph $\mathfrak{g}$ 
is a tree, any essential node $w$ whose
distance to $v_f$ is maximal,
 connect two or more generalized cusps,
say $\mathfrak{gc}_1$ and  $\mathfrak{gc}_2$,
as in Figure \ref{fi8}(A), and a graph
$\mathfrak{gr}$ containing $v_f$.
Now we define a new graph $\mathfrak{g}'$
by deleting the GC $\mathfrak{gc}_1$
and redefining the transitional probabilities
at $w$ as illustrated in Figure \ref{fi8}(A). 
This produces a new WCFG with one less 
GC, and any 
shell of the original system is still a shell of this new system, as shown in the proof of Lemma \ref{l52}. 
By iterating this process we arrive eventually to either,
a WCFG with a single leaf and a single cusp, as in the example depicted in 
 Figure \ref{fi8}(B), or a full line, as the one shown in
Figure \ref{fi6}(B), with possibly a few attached half-edges in either case.
The first case  is handled as in Example \ref{ex1}. The second can be seen as a graph
 obtained by attaching two GCs at their initial vertices, so it can be treated analogously. The result follows.
\qed

\begin{figure}
\[  
\unitlength 0.8mm 
\linethickness{0.4pt}
\ifx\plotpoint\undefined\newsavebox{\plotpoint}\fi 
\begin{picture}(137.75,17)(0,0)
\put(17,12){\framebox(13.25,5)[cc]{$\mathfrak{gc}_1$}}
\put(17.25,3.75){\framebox(13.25,4.25)[cc]{$\mathfrak{gc}_2$}}
\put(46.5,11.75){\circle*{1.414}}
\put(30.5,14.5){\vector(-4,1){.07}}\multiput(46,12)(-.20666667,.03333333){75}{\line(-1,0){.20666667}}
\put(31.25,6){\vector(-3,-1){.07}}\multiput(46.25,12)(-.084269663,-.033707865){178}{\line(-1,0){.084269663}}
\put(61.5,9.5){\framebox(11,5)[cc]{$\mathfrak{gr}$}}
\put(46.5,12){\vector(1,0){14}}
\put(37.75,15){\makebox(0,0)[cc]{${}^a$}}
\put(39.75,6){\makebox(0,0)[cc]{${}^b$}}
\put(53.75,13){\makebox(0,0)[cc]{${}^c$}}
\put(48,9){\makebox(0,0)[cc]{${}^w$}}
\thicklines
\put(76.5,12.5){\vector(1,0){11}}
\put(76.5,15){\vector(1,0){11}}
\put(76.5,10){\vector(1,0){11}}
\put(93.25,10){\framebox(10.75,5)[cc]{$\mathfrak{gc}_2$}}
\put(116,12.25){\circle*{1.414}}
\put(116,12.25){\vector(-1,0){10.75}}
\put(127.75,9.25){\framebox(11,5)[cc]{$\mathfrak{gr}$}}
\put(116,12.25){\vector(1,0){11}}
\put(111,14){\makebox(0,0)[cc]{${}^{a+b}$}}
\put(121.5,14){\makebox(0,0)[cc]{${}^c$}}
\put(116,9){\makebox(0,0)[cc]{${}^w$}}
\put(5,10){\makebox(0,0)[cc]{(A)}}
\end{picture}
 \]
\[
\unitlength 1mm 
\linethickness{0.4pt}
\ifx\plotpoint\undefined\newsavebox{\plotpoint}\fi 
\begin{picture}(95.75,18.25)(0,180)
\put(7.25,194){\vector(-3,4){.07}}\multiput(10.75,189.5)(-.033653846,.043269231){104}{\line(0,1){.043269231}}
\put(11,189.25){\vector(-2,-3){3.5}}
\put(11,189){\line(1,0){11}}
\put(16,189){\line(0,1){7}}
\put(24,194.5){\vector(1,3){.07}}\multiput(21.75,189)(.03358209,.08208955){67}{\line(0,1){.08208955}}
\put(24.75,185.5){\vector(3,-4){.07}}\multiput(21.75,189.25)(.03370787,-.04213483){89}{\line(0,-1){.04213483}}
\put(38.25,191.25){\vector(-2,-3){3.5}}
\put(38.25,191){\line(1,0){11}}
\put(43.25,191){\line(0,1){7}}
\put(51.25,196.5){\vector(1,3){.07}}\multiput(49,191)(.03358209,.08208955){67}{\line(0,1){.08208955}}
\put(52,187.5){\vector(3,-4){.07}}\multiput(49,191.25)(.03370787,-.04213483){89}{\line(0,-1){.04213483}}
\put(64,191.25){\vector(-2,-3){3.5}}
\put(64,191){\line(1,0){11}}
\put(69,191){\line(0,1){7}}
\put(77,196.5){\vector(1,3){.07}}\multiput(74.75,191)(.03358209,.08208955){67}{\line(0,1){.08208955}}
\put(90.5,191.5){\vector(-2,-3){3.5}}
\put(95.5,191.25){\line(0,1){7}}
\put(90.75,191.5){\line(1,0){5}}
\put(30,190){\makebox(0,0)[cc]{$\Rightarrow$}}
\put(56,190){\makebox(0,0)[cc]{$\Rightarrow$}}
\put(82,190){\makebox(0,0)[cc]{$\Rightarrow$}}
\put(16,196){\makebox(0,0)[cc]{$\bullet$}}
\put(43.4,197.5){\makebox(0,0)[cc]{$\bullet$}}
\put(69.2,197.5){\makebox(0,0)[cc]{$\bullet$}}
\put(95.5,198){\makebox(0,0)[cc]{$\bullet$}}
\put(-3,190){\makebox(0,0)[cc]{(B)}}
\end{picture}
\]
\caption{Cusp reducing step in the proof of Theorem \ref{t3} (A),
and an example of succesive reductions (B).}\label{fi8}
\end{figure}
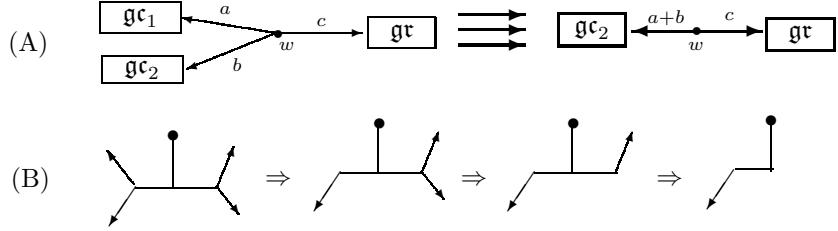
\vspace{2mm}

\paragraph{Proof of Theorem \ref{t5}}
This result is immediate from Theorem \ref{t2} and Theorem \ref{t3}
if we prove that the WRFQ $\mathfrak{wq}_Q(\mathbb{O}_D)$,
corresponding to the matrix 
$N_Q(\mathbb{O}_D)$, and the WCFG corresponding to the
matrix $f_n\big(N_P(\mathbb{O}_D)\big)$
have the same cusps. Now, if $v_D$ is a cusp vertex in
$\mathfrak{wq}_P(\mathbb{O}_D)$ corresponding to a maximal order 
of the form $\mathfrak{D}_D$, then the linear map $\mathbf{T}_P$
corresponding to $\mathfrak{wq}_P(\mathbb{O}_D)$ satisfies
$\mathbf{T}_P(\delta_{v_D})=\delta_{v_{D+P}}+
q\delta_{v_{D-P}}$.
On the other hand, the map  $\mathbf{T}_Q$
corresponding to $\mathfrak{wq}_Q(\mathbb{O}_D)$ satisfies
$\mathbf{T}_Q(\delta_{v_D})=\delta_{v_{D+nP}}+
q^n\delta_{v_{D-nP}}$,
since $Q$ is linearly equivalent to $nP$ and $\deg(Q)=n\deg(P)$.
It suffices now to check that both $\mathbf{T}_Q$
and $f_n(\mathbf{T}_P)$ have the same effect on 
$\delta_{v_D}$ whenever the degree of $D$ is large enough.
This follows from the fact that the polynomials $f_n$ have the property
$f_n(x+qx^{-1})=x^n+q^nx^{-n}$. The
latter result is proved by induction. It is trivial when $n=1$
or $n=2$, and the inductive step is handled
by a computation:
$$f_{n+2}\left(x+qx^{-1}\right)=
\left(x+qx^{-1}\right)f_{n+1}\left(x+qx^{-1}\right)
-qf_n\left(x+qx^{-1}\right)$$
$$=\left(x+qx^{-1}\right)\left(x^{n+1}+q^{n+1}x^{-n-1}\right)
-q\left(x^n+q^nx^{-n}\right)=x^{n+2}+q^{n+2}x^{-n-2}.$$
The result follows.

\section{Explicit computations of the obstruction}\label{sec6}

The purpose of this section is to examine some graph shapes 
that fail to allow an empty shield, thus illustrating
how the obstruction is explicitly computed in those 
examples. This means that presence of such shapes in the WRFQ  
$\mathfrak{wq}_P(\mathbb{O})$ indicates that
 $\mathfrak{wq}_Q(\mathbb{O})$ cannot be computed 
 from $\mathfrak{wq}_P(\mathbb{O})$
and the cusps alone.  In Ex. \ref{e65}, we also show
an example where the geometry of the graph alone fails to 
yield enough information to determine whether the 
obstruction set is trivial or not.

In all of this section we work with a normalized regular WRFQ
denoted by
$\mathfrak{wg}$, which is not assumed to come from a 
genus of orders. Some weights are not made explicit,
specially in the cusp at the right of every graph
in Figure \ref{soex},
but we note that the weights are assumed to be
positive for all pairs of neighboring vertices.
This assumption is all that is relevant for the examples.
A subgraph $\mathfrak{h}$ of  the associated graph
$\mathfrak{g}$ is called co-finite if $\mathfrak{g}$
is the union of $\mathfrak{h}$ and a finite graph.

\begin{example}
Assume that the normalized regular WCFG $\mathfrak{wg}$
has a symmetry $\sigma$ with a co-finite invariant subgraph.
Note that we assume $\sigma$ to preserve weights in the graph.
If $T_\sigma$ is the permutation matrix corresponding to this 
symmetry, and if $\mathbf{T}_\sigma$ is the corresponding
linear map, then $\mathbf{T}_\sigma-\mathbf{Id}$ is a 
linear map with a finite shell that commutes 
with the map $\mathbf{T}$ associated to $\mathfrak{wg}$.
Furthermore,
this is indeed the difference of two transition matrices, 
$\mathbf{T}_\sigma$ and the identity, so the sum
of the coefficients in each row vanishes. This is the case 
when $\mathfrak{wg}$ is as depicted in Figure \ref{soex}(A),
provided that the weights satisfy the relation $a=b$.
\end{example}

\begin{example}
We consider the WCFG depicted in Figure \ref{soex}(A)
and let $\mathbf{T}$ the corresponding
linear map $\mathbf{T}:C_V\rightarrow C_V$.
Any commuting map $\mathbf{F}$ with a finite
shell, must have a minimal shell $W$ contained in
$\{z,v\}$ by Lemma \ref{l54}. 
Since any set containing a shell
is a shell, it does not hurt if we assume $W=\{z,v,w\}$.
We let $\mathbf{P}=\mathbf{P}_W$ be the projection matrix
whose image is $C_W$ and whose kernel is $C_{W^c}$.
Then clearly $\mathbf{P}$ and $\mathbf{F}$ commute.
It follows that $\mathbf{F}$ is a map whose restriction
to $C_W$ has a matrix commuting with
$$T=\left(
\begin{array}{ccc}
0&0&a\\
0&0&b\\
1&1&0\\
\end{array}
\right),$$
with respect to the basis $\{\delta_v,\delta_z,\delta_w\}$.
Note that $T$ has always $0$ as an eigenvalue, regardless of the 
coefficients $a$ and $b$. In fact
the characteristic polynomial is $xF(x)$, where $F(x)=x^2-(a+b)$, 
whence we see that
$$F(T)=T^2-(a+b)\mathrm{Id}=
\left(
\begin{array}{ccc}
-b&a&0\\
b&-a&0\\
0&0&0\\
\end{array}
\right),$$
 is a matrix commuting with $T$ that fails to vanish 
 and have rows adding to $0$.
 Furthermore, the corresponding linear map has $\{z,v\}$ 
as a shell. We conclude that
whenever this graph arises as a quotient at a place $P$, 
the commuting relation and cusp structure 
fails to provide sufficient information to compute 
the neighborhood matrix at another place $Q$. 
Specifically, we cannot compute the weights $m_{v,z}$ and $m_{z,v}$, at $Q$, from the given data.
As before,
 $F(T)=\Big(F(T)+\mathrm{Id}\Big)-\mathrm{Id}$ is the 
difference of two transition matrices whenever $0\leq a,b\leq1$.
If $a+b\neq0$, so that $T$ has different eigenvalues, the matrix $T$
has a three dimensional centralizer, but only the multiples of
$F(T)$ correspond to linear maps with shell $\{z,v\}$.
In fact, in this case the map $\mathbf{F}$ corresponding to 
$F(T)$ spans the obstruction space $O_{\mathbf{T}}$.

If instead of working with the shell $W=\{v,z,w\}$ we use simply
$W'=\{v,z\}$, we arrive at the same result by searching for a
linear map $\mathbf{F}$ that commutes with
$\mathbf{P}_{W'}\mathbf{T}$ on $C_{W'}$ and satisfy
$$\mathbf{F}(a\delta_v+b\delta_z)=\mathbf{F}\mathbf{T}(\delta_w)
=\mathbf{T}\mathbf{F}(\delta_w)=0.$$
Note that the first condition is trivial as 
$\mathbf{P}_{W'}\mathbf{T}$ vanishes on $C_{W'}$.
\end{example}

\begin{figure}
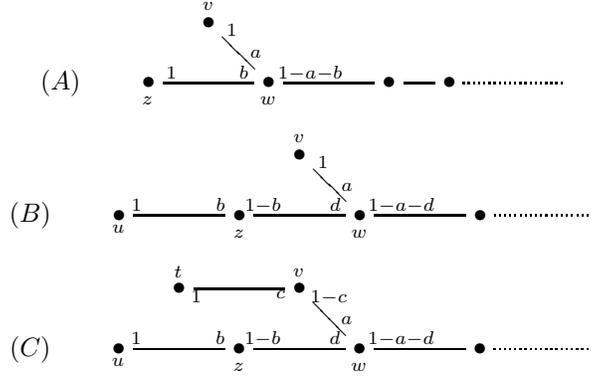

\[  (A) \qquad  \xygraph{
!{<0cm,0cm>;<.8cm,0cm>:<0cm,.8cm>::}
!{(2,0)}*+{\bullet}="a" !{(2,-0.3)}*+{{}_z}
!{(6,0)}*+{\bullet}="b" 
!{(4,0)}*+{\bullet}="c" !{(4,-0.3)}*+{{}_w}
  !{(3,1)}*+{\bullet}="c1" !{(3,1.2)}*+{{}^v}
!{(7,0)}*+{\bullet}="x"
!{(9,0)}*+{}="d"
 !{(2.4,.1)}*+{{}^1} !{(3.6,.1)}*+{{}^b} !{(3.8,.4)}*+{{}^a}  !{(3.4,.8)}*+{{}^1}
!{(4.7,.1)}*+{{}^{1-a-b}} 
 "a"-"c" "b"-"c" "b"-"x"  "x"-@{.}"d" 
"c"-"c1" 
} \]
\[  (B) \qquad  \xygraph{
!{<0cm,0cm>;<.8cm,0cm>:<0cm,.8cm>::}
!{(0,0)}*+{\bullet}="y" !{(0,-0.3)}*+{{}^u}
!{(2,0)}*+{\bullet}="a" !{(2,-0.3)}*+{{}_z}
!{(6,0)}*+{\bullet}="b" 
!{(4,0)}*+{\bullet}="c" !{(4,-0.3)}*+{{}_w}
  !{(3,1)}*+{\bullet}="c1" !{(3,1.2)}*+{{}^v}
!{(8,0)}*+{}="d"
 !{(2.4,.1)}*+{{}^{1-b}} !{(3.6,.1)}*+{{}^d} !{(3.8,.4)}*+{{}^a}  !{(3.4,.8)}*+{{}^1}
 !{(1.7,.1)}*+{{}^b}  !{(.3,.1)}*+{{}^1}
!{(4.7,.1)}*+{{}^{1-a-d}} 
 "a"-"c" "b"-"c" "b"-@{.}"d" "y"-"a"
"c"-"c1" 
} \]
\[  (C) \qquad  \xygraph{
!{<0cm,0cm>;<.8cm,0cm>:<0cm,.8cm>::}
!{(1,1)}*+{\bullet}="t" !{(1,1.2)}*+{{}^t}
!{(0,0)}*+{\bullet}="y" !{(0,-0.3)}*+{{}^u}
!{(2,0)}*+{\bullet}="a" !{(2,-0.3)}*+{{}_z}
!{(6,0)}*+{\bullet}="b" 
!{(4,0)}*+{\bullet}="c" !{(4,-0.3)}*+{{}_w}
  !{(3,1)}*+{\bullet}="c1" !{(3,1.2)}*+{{}^v}
!{(8,0)}*+{}="d"
 !{(2.4,.1)}*+{{}^{1-b}} !{(3.6,.1)}*+{{}^d} !{(3.8,.4)}*+{{}^a}  !{(3.5,.8)}*+{{}^{1-c}}
 !{(1.7,.1)}*+{{}^b}  !{(.3,.1)}*+{{}^1}
!{(4.7,.1)}*+{{}^{1-a-d}} 
!{(2.7,.9)}*+{{}_c}  !{(1.3,.8)}*+{{}^1}
 "a"-"c" "b"-"c" "b"-@{.}"d" "y"-"a" "t"-"c1"
"c"-"c1" 
}
 \]
\caption{Some examples of NWCFGs.}\label{soex}
\end{figure}

To simplify the task of actually checking that the obstruction space 
$O_{\mathbf{T}}$  is trivial, we can use next result:

\begin{proposition}
Let $\mathfrak{g}$ be a WCFG, and let $\mathbf{T}$ be the associated 
linear map. Let $W$ be a finite shell for $O_{\mathbf{T}}$. 
If $O_{\mathbf{T}}$ is non-trivial, then $\mathbf{T}$ has a
non-trivial eigenvector
in $\mathbb{C}\otimes_{\mathbb{R}}C_W$.
\end{proposition}   

\begin{proof}
Let $\mathbf{F}\in O_{\mathbf{T}}\smallsetminus\{0\}$.
By definition of shell we have $\mathbf{T}\mathbf{F}(C_V)
\subseteq \mathbf{F}(C_V)\subseteq C_W$, whence
$\mathbf{F}(C_V)$ is a non-null finite dimensional vector space fixed 
by $\mathbf{T}$. Then the result follows from elementary linear algebra.
\end{proof}

\begin{example}
We consider a transition matrix $T$ associated to the normalized regular
WCFG $\mathfrak{wg}$ depicted in Figure 7(B), with $a$ and $b$ nontrivial. 
The matrix of $T$ has the form
$$ T=\bmattrix {T'}XYZ,\quad\textnormal{where}\quad
T'=\left(
\begin{array}{cccc}
0&0&0&a\\
0&0&b&0\\
0&1&0&d\\
1&0&1-b&0\\
\end{array}
\right)$$
is the block corresponding to the subspace $\langle\delta_v,\delta_u,\delta_z,\delta_w\rangle$.
Note that $X$ has a non-zero coefficient only in the last row.
We claim that the associated linear map
$\mathbf{T}$ has no eigenvectors in
$\mathbb{C}\otimes_{\mathbb{R}}C_W$, where $W=\{v,u,z\}$. 
In fact, let $\mu$ be a complex valued measure supported
on $\{u,z,v\}$ that satisfies $\mathbf{T}(\mu)=\lambda\mu$. 
Note that 
$$\lambda\mu(\{v\})=\mathbf{T}(\mu)(\{v\})=a\mu(\{w\})=0.$$
If $\lambda\neq0$, we have $\mu(\{v\})=0$ and, since 
$$0=\mathbf{T}(\mu)(\{w\})=(1-b)\mu(\{z\})+1\cdot\mu(\{v\}),$$
we get $\mu(\{z\})=0$.
Note that we can ignore the remaining neighbor of $w$ 
since $\mu$ is supported on $\{u,v,z\}$.
 Furthermore, $$\lambda\mu(\{u\})=\mathbf{T}(\mu)(\{u\})=
 b\mu(\{z\})=0.$$ Hence, we get $\mu(\{u\})=0$, 
 and therefore $\mu=0$.
On the other hand, if $\lambda=0$, then $\mathbf{T}(\mu)=0$. 
If $\bar u$ denotes the row corresponding to
$\mu$, which has non-zero coefficients only in the 
first three positions, then we have the matrix identity 
$T\bar{u}=0$.
If $\hat{u}$ denotes the column of size $4$ obtained by 
cropping the first 
four coefficients of $\bar{u}$,
we claim that $T'\hat{u}=0$. The only non-trivial part 
of the claim is that $X$ has a non-zero coefficient in the fourth
row, as there is a vertex outside $W$ connected to $w$. 
However, the coefficient in $\bar{u}$ corresponding to
that vertex is zero. Since $T'$ is invertible, and since
$a$ and $b$ are nontrivial, it follows that $\hat{u}=0$, and hence
$\mu=0$. We conclude that the obstruction space $O_{\mathbf{T}}$
is trivial in this case. 
\end{example}

\begin{example}\label{e65}
We consider the WCFG depicted in Figure \ref{soex}(C),
with $0\notin\{a,b,c,d,1-b,1-c,1-a-d\}$.
Then, if the canonical basis of the charge space is ordered as
$(\delta_t,\delta_v,\delta_u,\delta_z,\delta_w,\dots)$,
then the matrix $T$ has a $5\times5$ block, on the upper left corner, of the form
 $\widetilde{A}=\sbmattrix A{\stackrel{\rightarrow}{r}^t}{\stackrel{\rightarrow}{s}}0$,
where the blocks are defined as follows:
\begin{itemize}
\item $A=\sbmattrix {A_1}00{A_2}$, where  $A_1=\sbmattrix 0c10$ and $A_1=\sbmattrix 0b10$.
\item $\stackrel{\rightarrow}{r}=(0\ a\ 0\ d)$, and
$\stackrel{\rightarrow}{s}=(0\ 1-c\ 0\ 1-b)$.
\end{itemize}

Again, we are reduced to find eigenvectors $\mu$ of  
$\mathbf{T}$ supported on $\{t,v,u,z\}$.  Assume $\mathbf{T}(\mu)=\lambda\mu$. 
Then $\lambda\mu(\{t\})=c\mu(\{v\})$ and, since $\mu(\{w\})=0$, we have
$\lambda\mu(\{v\})=\mu(\{t\})$. Combining these two equations we have
$$\lambda^2\mu(\{t\})=\lambda c\mu(\{v\})=c\mu(\{t\}),$$
 whence $\lambda^2=c$, unless $\mu(\{t\})=0$, and in the latter case we 
also have $\mu(\{v\})=0$, since $c\neq0$. Furthermore,
$\mu(\{v\})=0$ also implies $\mu(\{t\})=0$.
We prove similarly that $\lambda^2=b$,
unless $\mu(\{u\})=0$, and the latter case is equivalent to $\mu(\{z\})=0$.

Now, if either $\mu(\{v\})$ of $\mu(\{z\})$ vanishes,
 the condition 
$$0=\lambda\mu(\{w\})=(1-c)\mu(\{v\})+(1-b)\mu(\{z\})$$
shows that so does the other, whence we have $\mu=0$.
We conclude that $b=c$ is a necessary condition to have a non-trivial eigenvalue in this
case. Since the existence of a symmetry is a sufficient condition, we have shown that
the geometry of the graph fails to yield enough information 
to determine whether 
the obstruction space is trivial or not in this case. 
\end{example}

\begin{example}
Figure \ref{nf9} shows some aditional examples. In both (A) and (B) is possible to find weights
with either a trivial or a non-trivial obstruction space. In (D), the shell is empty since every vertex is in some cusp,
so there is no obstruction in this case. 
In (C) we arrive at the same result by showing that a non-trivial map 
whose shell is $\{v\}$
cannot commute with $\mathbf{T}$.
\end{example}

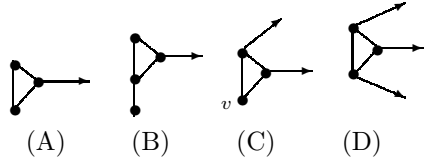
\begin{figure}
\[
\unitlength 1.4mm 
\linethickness{0.4pt}
\ifx\plotpoint\undefined\newsavebox{\plotpoint}\fi 
\begin{picture}(58.75,12)(0,152)
\put(19.75,158.5){\line(0,-1){4.75}}
\multiput(19.75,153.75)(.03333333,.03666667){75}{\line(0,1){.03666667}}
\multiput(22.25,156.5)(-.04166667,.03333333){60}{\line(-1,0){.04166667}}
\put(31.25,161){\line(0,-1){4.75}}
\multiput(31.25,156.25)(.03333333,.03666667){75}{\line(0,1){.03666667}}
\multiput(33.75,159)(-.04166667,.03333333){60}{\line(-1,0){.04166667}}
\put(41.5,159.5){\line(0,-1){4.75}}
\multiput(41.5,154.75)(.03333333,.03666667){75}{\line(0,1){.03666667}}
\multiput(44,157.5)(-.04166667,.03333333){60}{\line(-1,0){.04166667}}
\put(52,161.75){\line(0,-1){4.75}}
\multiput(52,157)(.03333333,.03666667){75}{\line(0,1){.03666667}}
\multiput(54.5,159.75)(-.04166667,.03333333){60}{\line(-1,0){.04166667}}
\put(22.3,156.55){\vector(1,0){4.75}}
\put(33.75,159){\vector(1,0){4}}
\put(44.25,157.5){\vector(1,0){4}}
\put(45.25,162.5){\vector(4,3){.07}}\multiput(41.75,159.75)(.04268293,.03353659){82}{\line(1,0){.04268293}}
\put(54.75,159.75){\vector(1,0){4}}
\put(57,164){\vector(2,1){.07}}\multiput(52,161.75)(.07462687,.03358209){67}{\line(1,0){.07462687}}
\put(57,155){\vector(2,-1){.07}}\multiput(52.25,157)(.07916667,-.03333333){60}{\line(1,0){.07916667}}
\put(31.25,156.25){\line(0,-1){3}}
\put(21.6,155.85){$\bullet$}\put(19.4,157.45){$\bullet$}\put(19.4,153.15){$\bullet$}
\put(33.2,158.25){$\bullet$}\put(30.8,160){$\bullet$}\put(30.8,153.15){$\bullet$}\put(30.8,156.15){$\bullet$}
\put(43.2,156.65){$\bullet$}\put(41,158.65){$\bullet$}\put(41,154.15){$\bullet$}
\put(53.8,158.85){$\bullet$}\put(51.5,160.95){$\bullet$}\put(51.5,156.55){$\bullet$}
\put(21,150){(A)}\put(31,150){(B)}\put(41,150){(C)}\put(51,150){(D)}
\put(39.5,153){${}^v$}
\end{picture}
\]
\caption{Some aditional examples.}\label{nf9}
\end{figure}

\section{Some explicit quotients at $P$}\label{intro}

At this point, we need to work out some explicit examples of 
quotient graphs for Eichler orders before we can apply our 
main results. 
For the explicit examples computed here, we set 
$X=\mathbb{P}^1$, and therefore  $K=\mathbb{F}(\mathbf{t})$. 
We compute the classifying graph at the degree-1 infinite place
$P=P_\infty$, for some genera of Eichler orders of small degree. 
In order to do this we recall the structure of the classifying 
graph $\mathfrak{c}_P(\mathbb{O}_0)$ for maximal orders, which
is the graph depicted in Figure \ref{f10}(A). 
This was computed by Serre in  \cite[\S II.1.6]{trees}.
We perform our computations by first looking at principal congruence 
subgroups $\Gamma(f)$ of 
$\mathrm{GL}_2(\mathbb{F}[\mathbf{t}])$, and then obtain classifying 
graphs as quotients of $\Gamma(f)\backslash\mathfrak{t}$.

Note that the completion at $P$ is the field $K_P=\mathbb{F}(\mathbf{t}^{-1}))$, the ring of integers is
$\mathcal{O}_P=\mathbb{F}[[\mathbf{t}^{-1}]]$, and $\mathbf{t}^{-1}$ is a uniformizing parameter.
We use $a\mapsto v_P(a)$ for the valuation and $a\mapsto |a|_P$ for the absolute value on $K_P$.
To make the Bruhat-Tits tree as explicit as possible, we identify it with the ball tree, whose vertices 
correspond to closed balls $B$ in $K_P$.
 Two such balls are neighbors if one is a maximal 
 proper sub-ball 
of the other. The local maximal order corresponding to the ball $B=B_a^{[r]}$ with center $a$ and radius
$|\mathbf{t}^{-r}|_\infty$ is
$$\Da[B]=\bmattrix a{\mathbf{t}^{-r}}10\mathbb{M}_2(\oink_P){\bmattrix a{\mathbf{t}^{-r}}10}^{-1}=
\mathrm{End}_{\oink_P}\left\langle\left(\begin{array}{c}a\\1\end{array}\right),
\left(\begin{array}{c}\mathbf{t}^{-r}\\0\end{array}\right)\right\rangle.$$
The natural action of $\mathrm{PGL}_2(K)$, by conjugation,
on maximal orders, corresponds to an action
on balls that we define by $\Da[a*B]=a\Da[B]a^{-1}$.
It is not hard to see that visual limits of the ball 
tree can be identified with  the elements of
the set $\mathbb{P}^1(K_P)$, in a way that 
$\infty$ is the visual limit of any ray whose
vertices, in the natural order, correspond to
an increasing sequence of balls,
while a finite element $a\in K_P$ is the visual limit of a 
ray whose vertices correspond to a decreasing sequence
of balls with intersection $\{a\}$.
With these conventions the group $\mathrm{PGL}_2(K)$
acts on visual limits via Moebius transformations.
The ball tree is usually drawn with the infinite 
visual limit on top and 
``finite'' visual limits at the bottom,
as shown in Figure \ref{f10}(B). We use the convention of 
denoting visual limits by stars.
This identification allows us to easily describe a
fundamental region for the action of
the normalizer $N=\mathbb{M}_2(\mathbb{F}[\mathbf{t}])^*$. 
It is the ray $\mathfrak{r}$ whose vertices are the 
balls of the form $B_0^{[-n]}$,
where $n$ is a positive integer, 
as Figure \ref{f10}(A) shows. 

We recall here a sketch of the proof of the fact that
$\mathfrak{r}$ is indeed a fundamental region.
One can prove that $\mathfrak{r}$ contains a fundamental region,
by finding explicit elements in $N$ that maps all edges outside 
this region
to an edge inside the region. By an inductive argument, 
it suffices to do it for edges that are adjacent to
$\mathfrak{r}$. The latter is easy if you recall that 
$\mathrm{GL}_2(K_P)$ acts on this
tree via Moebius transformations, and note the following fact:
\begin{quote}
\textbf{Fact 1:}
If $g\in \mathbb{F}[\mathbf{t}]$ has degree $d$, then the group
of transformations of the form $\mu(z)=z+ag$, for
$a\in\mathbb{F}$, fixes $B_0^{[-d]}$, and permutes transitively
all its maximal proper sub-balls.
\end{quote}
For $d=0$ you get the additional transformation $\eta(z)=\frac1{z}$,
so all neighbors of  $B_0^{[0]}$ are in the same orbit.
Finally, you can see 
that  $\mathfrak{r}$ is indeed a fundamental 
region, by observing that the global order corresponding 
to $B_0^{[-n]}$ is $\Da_{nP}$,
and this orders have non-isomorphic rings of global 
sections for different values of $n>0$, as it is proved in
 \cite[Lemma 6.3]{ogcseooff}.

\begin{figure}
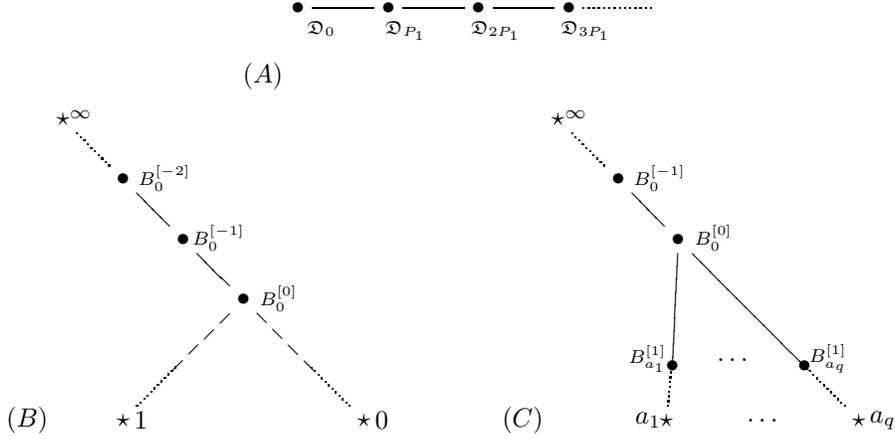

\[  (A)\xygraph{
!{<0cm,0cm>;<.8cm,0cm>:<0cm,.8cm>::}
!{(0,4)}*+{\bullet}="a" !{(0.4,3.6)}*+{{}^{\Da_0}}="a1"
!{(1.5,4)}*+{\bullet}="b" !{(1.8,3.6)}*+{{}^{\Da_{P_1}}}="b1"
!{(3,4)}*+{\bullet}="c" !{(3.3,3.6)}*+{{}^{\Da_{2P_1}}}="c1"
!{(4.5,4)}*+{\bullet}="x" !{(4.8,3.6)}*+{{}^{\Da_{3P_1}}}="x1"
!{(6,4)}*+{}="d" !{(2,3)}*+{}="e"
 "a"-"b" "b"-"c" "c"-"x" "x"-@{.}"d"
} \]
\[
 (B)\xygraph{
!{<0cm,0cm>;<.8cm,0cm>:<0cm,.8cm>::}
!{(3,2)}*+{\bullet}="a" !{(3.6,2)}*+{{}^{B_0^{[0]}}}="a1"
!{(2,3)}*+{\bullet}="b" !{(2.6,3)}*+{{}^{B_0^{[-1]}}}="b1"
!{(1,4)}*+{\bullet}="c" !{(1.7,4)}*+{{}^{B_0^{[-2]}}}="c1"
!{(0,5)}*+{\star}="x" !{(0.3,5)}*+{{}^{\infty}}="x1"
!{(1,0)}*+{\star}="d" !{(2,1)}*+{}="dm" !{(1.7,0.7)}*+{}="dn" !{(1.3,0)}*+{1}="d1"
!{(5,0)}*+{\star}="e" !{(4,1)}*+{}="em" !{(4.3,0.7)}*+{}="en" !{(5.3,0)}*+{0}="e1"
 "a"-"b" "b"-"c" "c"-@{.}"x" "a"-@{--}"dn" "a"-@{--}"en" "dm"-@{.}"d" "em"-@{.}"e"
} \qquad\qquad
 (C)\xygraph{
!{<0cm,0cm>;<.8cm,0cm>:<0cm,.8cm>::}
!{(2,3)}*+{\bullet}="b" !{(2.6,3)}*+{{}^{B_0^{[0]}}}="b1"
!{(1,4)}*+{\bullet}="c" !{(1.7,4)}*+{{}^{B_0^{[-1]}}}="c1"
!{(0,5)}*+{\star}="x" !{(0.3,5)}*+{{}^{\infty}}="x1"
!{(1.8,0)}*+{\star}="d" !{(1.9,1)}*+{}="dm" !{(1.9,0.9)}*+{\bullet}
!{(1.9,0.7)}*+{}="dn" !{(1.5,0)}*+{a_1}="d1" !{(1.5,1)}*+{{}^{B_{a_1}^{[1]}}}
!{(5,0)}*+{\star}="e" !{(4,1)}*+{}="em" !{(4.1,0.9)}*+{\bullet}
!{(4.3,0.7)}*+{}="en" !{(5.4,0)}*+{a_q}="e1" !{(4.5,1)}*+{{}^{B_{a_q}^{[1]}}}
!{(2.9,1)}*+{\dots} !{(3.4,0)}*+{\dots}
  "b"-"c" "c"-@{.}"x" "b"-"dn" "b"-"en" "dm"-@{.}"d" "em"-@{.}"e"
}
\]\caption{The quotient of the Bruhat-Tits tree by the modular group (A), and 
one fundamental region in the ball tree (B).
In (C) we have a fundamental region for the action of $\Gamma(N)$ for a 
linear polynomial $N$, if $\mathbb{F}=\{a_1,\dots,a_q\}$.
}\label{f10}
\end{figure}

Next, we describe the structure of the quotient graph for the principal congruence subgroup
$$\Gamma(f)=\left\{\bmattrix {1+fa}{fb}{fc}{1+fd}\in\mathrm{GL}_2(\mathbb{F}[\mathbf{t}])
\Big|\ a,b,c,d\in\mathbb{F}[\mathbf{\mathbf{t}}]\right\},$$
for some particular polynomials $f$. Note that $\Gamma(1)=\mathrm{GL}_2(\mathbb{F}[\mathbf{t}])$ and
 $\Gamma(0)=\{1_{\mathbb{M}_2(\mathbb{F}[t])}\}$.
The complexity of these graphs increase quickly with both 
$q=|\mathbb{F}|$ and $\deg(f)$, so we restrict the explicit 
examples to the cases where, either $\deg(f)=1$, or $\deg(f)=q=2$.
However, we start with some general considerations for arbitrary $f$, 
by studying the natural covering  
 $\phi:\Gamma(f)\backslash\mathfrak{t}
 \twoheadrightarrow
\Gamma(1)\backslash\mathfrak{t}$.
A vertex in $\Gamma(f)\backslash\mathfrak{t}$ is called a
type-$n$ vertex if its image in
the ray $\Gamma(1)\backslash\mathfrak{t}$ 
is the class of the ball $B_0^{[-n]}$.
We apply the same convention to the Bruhat-tits 
tree $\mathfrak{t}=\Gamma(0)\backslash\mathfrak{t}$.
A type-$n$ vertex, for $n>0$, can have only type-$(n-1)$ 
and type-$(n+1)$ neighbors.
Type-$0$ vertices only have type-$1$ neighbors. Furthermore,
by Fact 1 above, all but one of the neighbors of a type-$n$ 
vertex are type-$(n-1)$ vertices. It also follows 
from Fact 1 that
you can permute the sub-balls of $B_0^{[-n]}$ by a Moebius transformation corresponding to a matrix in $\Gamma(f)$ when
the polynomial $g$ with $\deg(g)=d$ can be a multiple
of $f$. We conclude that the class of the ball $B_0^{[-n]}$,
for $n\geq\deg(f)$, is unramified for $\phi$. 
A similar argument show that the balls $B_0^{[-k]}$ for 
$0\leq k<\deg(f)$ are unramified for the natural cover $\mathfrak{t}\twoheadrightarrow\Gamma(f)\backslash\mathfrak{t}$.
Let $M_n$ denote the number of type-$n$ vertices in $\Gamma(f)\backslash\mathfrak{t}$.
Then  $M_{n-1}=M_n=M_{n+1}=\dots$, 
as the corresponding vertices are contained in  the 
same amount of cusps. For $0<k<\deg(f)$, we have 
$M_{k-1}=qM_k$, while
$(q+1)M_0=qM_1$, for $\deg(f)>1$. When $\deg(f)=1$ the 
relation reads $(q+1)M_0=M_1$. 

If $n=\deg(f)=1$, the only ramified vertices for $\phi$ are type-$0$ vertices, so the quotient graph in this case
consist on $q+1$ cusps attached to a single vertex.  
In this case, the quotient graph looks like the one 
depicted in Figure \ref{f10}(C).
This is shown in more detail in 
\cite[\S II.1.6, Ex. 5]{trees}.

Now assume $n=\deg(f)=2$.
Let  $\mathfrak{a}=\mathfrak{a}_{0,1}(f)$ be the largest  subgraph of  $\Gamma(f)\backslash\mathfrak{t}$ 
having only type-$0$ and type $1$ vertices. Then  $\Gamma(f)\backslash\mathfrak{t}$  is obtained by attaching
a cusp to every type-$1$ vertex in $\mathfrak{a}$. If we further assume $q=2$, then type-$0$ vertices have
valency $3$ in $\mathfrak{a}$, while type-$1$ vertices have valency $2$. We conclude that $\mathfrak{a}$
is the baricentric subdivision of a homogeneous
graph $\mathfrak{o}$ with only 
valency-3 vertices. We identify 
this graph  $\mathfrak{o}$ for different values of $f$.

\begin{theorem}\label{t12}
For quadratic polynomials $f$  in $\finitum_2[\mathbf{t}]$, the graph $\mathfrak{o}$ defined above
is as follows:
\begin{itemize}
\item If $f=\mathbf{t}(\mathbf{t}+1)$, then $\mathfrak{o}$ is the full bipartite graph $K_{3,3}$.
\item If $f$ is a perfect square, then $\mathfrak{o}$ is the 1-skeleton of a cube.
\item If $f=\mathbf{t}^2+\mathbf{t}+1$, then $\mathfrak{o}$ is the Petersen graph.
\end{itemize}
\end{theorem}

The graph  $\Gamma(f)\backslash\mathfrak{t}$ , in each case covered by Theorem \ref{t12} above, can be seen
in Figure \ref{ft12}. They can be recovered by attaching a 
cusp at every
virtual vertex in the baricentric subdivision of $\mathfrak{o}$. 
Before we proceed to the proof of Theorem \ref{t12}, we need a
couple of lemmas.

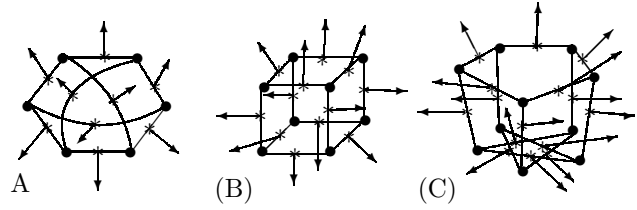
\begin{figure}
\[
\unitlength 1mm 
\linethickness{0.4pt}
\ifx\plotpoint\undefined\newsavebox{\plotpoint}\fi 
\begin{picture}(87.75,26.75)(0,0)
\put(42.5,11){\framebox(9.25,8.5)[cc]{}}
\put(38.25,6.75){\framebox(8.75,9)[cc]{}}
\multiput(38.5,15.5)(.033482143,.035714286){112}{\line(0,1){.035714286}}
\multiput(47,15.5)(.035714286,.033613445){119}{\line(1,0){.035714286}}
\put(42.5,11){\line(-1,-1){4.5}}
\multiput(51.75,11)(-.035447761,-.03358209){134}{\line(-1,0){.035447761}}
\put(37.75,21.5){\vector(-2,3){.07}}\multiput(40.5,17.5)(-.03353659,.04878049){82}{\line(0,1){.04878049}}
\put(46.75,24.5){\vector(0,1){.07}}\multiput(46.5,19.25)(.03125,.65625){8}{\line(0,1){.65625}}
\put(44,21.75){\vector(0,1){.07}}\multiput(43.75,15.75)(.03125,.75){8}{\line(0,1){.75}}
\put(52,23.75){\vector(1,3){.07}}\multiput(49.75,17.75)(.03358209,.08955224){67}{\line(0,1){.08955224}}
\put(57.5,15){\vector(1,0){.07}}\multiput(51.75,15.25)(.71875,-.03125){8}{\line(1,0){.71875}}
\put(52,12.75){\vector(1,0){.07}}\multiput(47,12.5)(.625,.03125){8}{\line(1,0){.625}}
\put(38,11.75){\vector(-1,0){5.75}}
\put(34,7){\vector(-3,-1){.07}}\multiput(40.75,9)(-.1125,-.03333333){60}{\line(-1,0){.1125}}
\put(42.5,6.75){\vector(0,-1){4.25}}
\put(53.5,5.5){\vector(1,-1){.07}}\multiput(49.75,9)(.036057692,-.033653846){104}{\line(1,0){.036057692}}
\put(45.75,11){\vector(0,-1){6.5}}
\put(42.5,14.5){\vector(-1,0){4}}
\put(69.75,21){\line(1,0){9.75}}
\put(79.5,21){\line(3,-4){3}}
\multiput(82.5,17)(-.091346154,-.033653846){104}{\line(-1,0){.091346154}}
\multiput(73,13.5)(-.065298507,.03358209){134}{\line(-1,0){.065298507}}
\multiput(64.25,18)(.06179775,.03370787){89}{\line(1,0){.06179775}}
\multiput(70,10.25)(.03353659,-.07317073){82}{\line(0,-1){.07317073}}
\multiput(72.75,4.25)(.046474359,.033653846){156}{\line(1,0){.046474359}}
\multiput(80,9.5)(-.19776119,-.03358209){67}{\line(-1,0){.19776119}}
\multiput(66.75,7.25)(.26923077,-.03365385){52}{\line(1,0){.26923077}}
\multiput(80.75,5.5)(-.078358209,.03358209){134}{\line(-1,0){.078358209}}
\multiput(66.75,7)(-.03333333,.15){75}{\line(0,1){.15}}
\multiput(72.75,4)(.03125,1.15625){8}{\line(0,1){1.15625}}
\multiput(80.75,5.25)(.03365385,.22596154){52}{\line(0,1){.22596154}}
\put(79.5,9.5){\line(0,1){11.75}}
\multiput(69.75,10.25)(-.03125,1.34375){8}{\line(0,1){1.34375}}
\put(67.25,19.25){\vector(-1,2){2.5}}
\put(75,26.75){\vector(0,1){.07}}\multiput(74.75,20.75)(.03125,.75){8}{\line(0,1){.75}}
\put(81,19.25){\vector(1,1){4}}
\put(86,20.25){\vector(3,2){.07}}\multiput(77.75,15)(.052884615,.033653846){156}{\line(1,0){.052884615}}
\put(61,16.25){\vector(-1,0){.07}}\multiput(68.5,15.5)(-.326087,.0326087){23}{\line(-1,0){.326087}}
\put(69.5,14){\vector(-1,0){6}}
\put(65.25,12.25){\vector(-1,0){6.25}}
\put(78.25,11){\vector(1,0){.07}}\multiput(73,10.5)(.35,.0333333){15}{\line(1,0){.35}}
\put(87.75,11.75){\vector(1,0){.07}}\multiput(82,12.25)(.3833333,-.0333333){15}{\line(1,0){.3833333}}
\put(79.5,14){\vector(1,0){6.5}}
\put(65.5,4){\vector(-2,-1){.07}}\multiput(71.25,7.25)(-.059278351,-.033505155){97}{\line(-1,0){.059278351}}
\put(85.5,9){\vector(4,1){.07}}\multiput(75.75,7.75)(.25657895,.03289474){38}{\line(1,0){.25657895}}
\put(79.75,3){\vector(1,-1){.07}}\multiput(76,6.5)(.036057692,-.033653846){104}{\line(1,0){.036057692}}
\put(78.75,1.75){\vector(1,-1){.07}}\multiput(74,6.25)(.035447761,-.03358209){134}{\line(1,0){.035447761}}
\put(70.75,12.75){\vector(-1,4){.07}}\multiput(72,8)(-.03289474,.125){38}{\line(0,1){.125}}
\put(12.25,19.5){\line(1,0){9.25}}
\multiput(21.5,19.5)(.033613445,-.054621849){119}{\line(0,-1){.054621849}}
\put(25.5,13){\line(-3,-4){4.5}}
\put(21,7){\line(-1,0){9}}
\multiput(12,7)(-.033687943,.046099291){141}{\line(0,1){.046099291}}
\multiput(7.25,13.5)(.033687943,.042553191){141}{\line(0,1){.042553191}}
\put(16.5,6.75){\vector(0,-1){5}}
\put(7.25,20.75){\vector(-2,3){.07}}\multiput(10,16.5)(-.03353659,.05182927){82}{\line(0,1){.05182927}}
\put(6,6){\vector(-3,-4){.07}}\multiput(9.5,10.25)(-.033653846,-.040865385){104}{\line(0,-1){.040865385}}
\put(27.25,6.75){\vector(4,-3){.07}}\multiput(23.5,9.75)(.04213483,-.03370787){89}{\line(1,0){.04213483}}
\put(26,19.25){\vector(1,2){.07}}\multiput(24,15.75)(.03333333,.05833333){60}{\line(0,1){.05833333}}
\put(17.25,19.5){\vector(0,1){4.75}}
\put(16.55,17.75){*}
\put(9.03,14.7){*}
\put(23.2,13.7){*}
\put(15.65,5){*}
\put(8.9,8.3){*}
\put(22.5,8.1){*}
\put(11.45,6.05){$\bullet$}
\put(6.3,12.15){$\bullet$}
\put(24.65,12.15){$\bullet$}
\put(19.9,6.05){$\bullet$}
\put(10.95,18.55){$\bullet$}
\put(20.85,18.55){$\bullet$}
\put(4.85,1.55){A}
\put(37.75,14.65){$\bullet$}
\put(37.5,6.1){$\bullet$}
\put(46.45,14.7){$\bullet$}
\put(46.5,6.1){$\bullet$}
\put(50.85,18.5){$\bullet$}
\put(51.15,10.25){$\bullet$}
\put(41.6,18.6){$\bullet$}
\put(41.75,10.2){$\bullet$}
\put(69.05,19.8){$\bullet$}
\put(78.65,19.75){$\bullet$}
\put(72.1,12.65){$\bullet$}
\put(81.6,15.95){$\bullet$}
\put(63.65,17){$\bullet$}
\put(72.1,3.5){$\bullet$}
\put(79.7,4.93){$\bullet$}
\put(66.1,6.25){$\bullet$}
\put(78.75,8.85){$\bullet$}
\put(69.15,9.2){$\bullet$}
\put(39.7,15.9){*}
\put(41.8,4.75){*}
\put(49.01,7.15){*}
\put(39.85,7.3){*}
\put(45.05,9.05){*}
\put(51,13.35){*}
\put(45.45,17.35){*}
\put(49,15.85){*}
\put(46.25,10.65){*}
\put(41.6,12.5){*}
\put(43,13.9){*}
\put(37.2,9.9){*}
\put(74,19.1){*}
\put(66.35,17.6){*}
\put(80.05,17.4){*}
\put(76.85,13.25){*}
\put(68,13.5){*}
\put(81.05,10.45){*}
\put(72,8.55){*}
\put(68.75,12.25){*}
\put(78.75,12.25){*}
\put(64.75,10.35){*}
\put(71.2,6.3){*}
\put(70.25,5.25){*}
\put(73.15,4.5){*}
\put(75.35,4.75){*}
\put(74.75,5.75){*}
\qbezier(7,13.5)(16.5,8)(25,12.75)
\qbezier(21,19)(9.375,16)(12.25,7.25)
\qbezier(11.75,20)(20.125,14.75)(21,6.5)
\put(21.5,15.75){\vector(3,2){.07}}\multiput(18,13.5)(.05223881,.03358209){67}{\line(1,0){.05223881}}
\put(11,16.75){\vector(-1,1){.07}}\multiput(13,14.75)(-.03333333,.03333333){60}{\line(0,1){.03333333}}
\put(13.75,8.15){\vector(-1,-1){.07}}\multiput(16.05,10.55)(-.03333333,-.03333333){60}{\line(0,1){.03333333}}
\put(12.3,12.55){*}
\put(17.65,11.95){*}
\put(15.25,8.75){*}
\put(32,1){(B)}
\put(59,1){(C)}
\end{picture}
\]
\caption{The graphs of Theorem \ref{t12} for $N=t(t+1)$ (A), $N$ a perfect square (B), or $N=t^2+t+1$ (C).
 We highlight the type-0 ($\bullet$) and type-1 (*) vertices. 
Arrows denote cusps.}\label{ft12}
\end{figure}

\begin{lemma}\label{l41}
Let $\bar a$ be an element in the quotient    $\Gamma(1)/\Gamma(f)$ that
has a fixed vertex $\bar v$ in the quotient
graph  $\Gamma(f)\backslash\mathfrak{g}$.
Then, for any lifting $v\in\mathfrak{t}$ of $\bar{v}$,
there is a lifting $a\in\Gamma(1)$
of $\bar a$ that fixes $v$.
\end{lemma}

\begin{proof}
If $\gamma.\bar{v}=\bar{v}$, then
$\gamma.v=\nu.v$ for some $\nu\in\Gamma(f)$.
It follows that $\nu^{-1}\gamma.v=v$, and
$\nu^{-1}\gamma$ is a lifting of
$\bar{\gamma}$.
\end{proof}

Recall that the group $\Delta=GL_2(\mathbb{F})$ is the stabilizer of $B_0^{[0]}$ in
$\Gamma(1)$.  We use $\bar{\Delta}$ for its image in $\Gamma(1)/\Gamma(f)$
in all that follows.

\begin{lemma}\label{ll51}
Let $a\in\Gamma(1)$ be an element
satisfying the following conditions:
\begin{itemize}
    \item The eigenvalues of $a$ are in
    $K_P$, and their valuations 
    differ by $2$.
    \item The image $\bar{a}$ of $a$ in
    $\Gamma(1)/\Gamma(f)$ has order
    $r$.
    \item No non-trivial power of $\bar{a}$ is
    conjugate to an element of $\bar\Delta$.
\end{itemize}
Then, the corresponding Moebius transformation
$\tau_a$ has precisely two fixed points
in $K_P$, and the action, on $\Gamma(f)\backslash\mathfrak{t}$,
of the matrix $a$ has no type-$0$ invariant vertex.
Furthermore, the image, in $\Gamma(f)\backslash\mathfrak{t}$, 
of the maximal geodesic joining the fixed
points of $\tau_a$ is a cycle
of length $2r$ containing no type-$n$ vertex for $n\geq2$.
\end{lemma}

\begin{proof}
The fact that $a$ has two different eigenvalues, and they belong 
to $K_P$, implies that the elements in the commutative algebra
$K_P[a]$ that are integral over $\mathcal{O}_P$ are contained,
precisely, in the maximal orders corresponding
to the vertices in a maximal geodesic 
$\mathfrak{p}$
(c.f. \cite[Prop 4.2]{Eichler2}).
If $\lambda_1$ and $\lambda_2$ are the eigenvalues  of $a$,
then $\tau_a$ is conjugate to the transformation
$\eta(z)=(\lambda_1/\lambda_2)z$, and the condition on their valuations
imply that $\tau_a$ acts as a shift by two on its unique
invariant maximal geodesic, which can be no other than $\mathfrak{p}$.
In particular, the visual limits of the geodesic
must be the only fixed points
of $\tau_a$. Since the order of $\bar{a}$ is $r$, any pair of
vertices at distance $2r$ on $\mathfrak{p}$
coincide on the quotient, but closer vertices
cannot coincide or some non-trivial power $\bar{a}^s$
would stabilize a type-0 vertex, and so it should be 
conjugate to an element of $\bar\Delta$
 by Lemma \ref{l41}. Finally, the fact that this cycle contains
 only type-0 and type-$1$ vertices
follows from the fact that type-$n$ vertices are located in
cusps for $n\geq2$, and therefore cannot be part of a cycle.
\end{proof}

\paragraph{Proof of Theorem \ref{t12}}
The strategy consists in finding, case by case, an
element $\bar{a}$  in the quotient $\Gamma(1)/\Gamma(f)$ that satisfies the hypotheses of Lemma \ref{ll51}.
This provides us with a suitable cycle that
can be easily completed to the whole graph
by ad-hoc arguments. In fact, in each case we have the same lifting $a=\sbmattrix {\mathbf{t}}110$.
Since the equation $x^2+\mathbf{t}x+1=0$ can be reduced to 
$y^2+y+\mathbf{t}^{-2}=0$ by the substitution $x=\mathbf{t}y$, the matrix $a$ has two different 
eigenvalues $\lambda_1$ and $\lambda_2$ whose valuations differ by $2$, as follows from Hensel's Lemma.
We conclude that multiplication by $a$ is a shift by two on the maximal geodesic whose visual limits are the roots of the equation for $x$ above. 

When $f=\mathbf{t}(\mathbf{t}+1)$, then $\Gamma(1)/\Gamma(f)$  is isomorphic to
the group $\mathrm{SL}_2(\mathbb{F}_2)\times
\mathrm{SL}_2(\mathbb{F}_2)$ by a Strong Approximation argument.
Note that the isomorphism is given by
$\bar{h}\mapsto\big(h(0),h(1)\big)$, for a polynomial
matrix $h\in\Gamma(1)$. In particular, $\Gamma(1)/\Gamma(f)$ 
has order 36, while the stabilizer of one type-$0$ 
element, which is 
$\bar\Delta\cong\mathrm{SL}_2(\mathbb{F}_2)\cong S_3$, 
has order $6$,
so there are $6$ type-$0$ vertices. We note that the element 
$$ \bar{a}=(a_1,a_2),\quad a_1=\left(\begin{array}{cc}0&1\\1&0\end{array}\right), \quad
a_2=\left(\begin{array}{cc}1&1\\1&0\end{array}\right),$$ has order $6$ and can be lifted to 
the matrix $a$ defined above. 
In this case, the group  $\bar\Delta$
embeds diagonally into
$\mathrm{SL}_2(\mathbb{F}_2)\times
\mathrm{SL}_2(\mathbb{F}_2)$.
Since $\mathrm{SL}_2(\mathbb{F}_2)$ has no element of order $6$,
$\bar{a}$ is not conjugate to an element of $\bar\Delta$.
For powers of $\bar{a}$ of degree $2$ or $3$, you can directly
see that they are not conjugate to a diagonal element since
they have a trivial coordinate.
It follows that $\bar{a}$ satisfies all the hypotheses 
of the preceding lemma.
 This gives us an hexagon
that contains all type-0 vertices,
and we just need a symmetry of the quotient graph 
that reverses this hexagon, 
since this proves that the remaining edge, for each
type-$0$ vertex in $\mathfrak{o}$, can only connect it 
to the opposite vertex of the hexagon
as in Fig. \ref{ft12}(A). Note that the element $b=\sbmattrix 0110$
 satisfies $bab^{-1}=a^{-1}$,
and therefore it must reverse the orientation of the hexagon, as required.

When $f=\mathbf{t}^2$, 
 the elements of $\Gamma(1)/\Gamma(f)$ can be written in the form $\bar{a}=\alpha+\bar{\mathbf{t}}\beta$, 
where $\alpha$ and $\beta$ are constant matrices, $\alpha$ is invertible and
$\alpha^{-1}\beta$ has zero trace.
In particular $\Gamma(1)/\Gamma(f)$ has order $48$.
Again the stabilizer of a type $0$-vertex has order
$6$, so we have now $8$ type-$0$ vertices.
If we take $\alpha=\sbmattrix0110$
and $\beta=\sbmattrix1000$, then the element $\bar{a}$, which
lifts to $a$ as before,  has order $4$. 
In fact, $\alpha^2=1_{\mathbb{M}}$ is the identity matrix, while
$\beta^2=\beta$ and $\alpha\beta+\beta\alpha=\alpha$.
 We conclude that 
 $\bar{a}^2=1_{\mathbb{M}}+
 \bar{\mathbf{t}}\alpha$,
 $\bar{a}^3=\alpha+\bar{\mathbf{t}}(1+\beta)$
 and
 $\bar{a}^4=1_{\mathbb{M}}$.
 Note that the homomorphism taking each element
 $\bar{a}'=\alpha'+\bar{\mathbf{t}}\beta'$ to
 the constant matrix $\alpha'$ fails to preserve
 the order of any non-trivial power of 
 $\bar{a}$,
 so they are not conjugate to constant
 matrices. This
implies that $a$ satisfies all condition in
the preceding proposition, and therefore the
quotient graph has a cycle of length $8$. 
Furthermore,
each type-$0$ vertex in this cycle is at distance $2$ of one of the remaining $4$ type-$0$ vertices,
by symmetry and the connectedness of the $0$-graph $\mathfrak{o}$.
Each of these remaining type-$0$ vertices, say $v$, must be similarly connected i.e., at distance $2$, to two 
more type-$0$ vertices $v'$ and $v''$, in a way that the whole graph is invariant by an order-4
 rotation of the square. Furthermore, $v'$ and $v''$ must be different, since the action of 
$\Gamma(1)/\Gamma(f)$ is transitive on type-$0$ vertices.  This can happen only
if the remaining type $0$-vertices are arranged in a second square that completes the cube in Figure \ref{ft12}(B).
 The case $f=\mathbf{t}^2+1$ can be deduced
from here by applying an automorphism to the ring $\mathbb{F}[\mathbf{t}]$.

Finally,  when $f=\mathbf{t}^2+\mathbf{t}+1$,  then $\Gamma(1)/\Gamma(f)$  is isomorphic to
the group $\mathrm{SL}_2(\mathbb{F}_4)$, and the matrix 
$\bar{a}=\sbmattrix{\omega}110$, where $\omega=\bar{\mathbf{t}}$ is a cubic
root of unity, lifts to the same matrix $a$ as before, and its powers are
$$\bar{a}^2=\left(\begin{array}{cc}\omega&\omega\\ \omega&1\end{array}\right),\quad
\bar{a}^3=\left(\begin{array}{cc}1&\omega\\ \omega&\omega\end{array}\right),\quad
\bar{a}^4=\left(\begin{array}{cc}0&1\\1&\omega\end{array}\right)$$
and finally $\bar{a}^5=1_{\mathbb{M}}$. All non-trivial
powers have order $5$, so the condition on their conjugates
is immediate. As before, this implies the existence of a cycle 
of length $10$ containing $5$
of the type-$0$ vertices, each at distance $2$ from one of the remaining $5$ type-$0$ vertices.
Now the rotational symmetry implies that the graph has one of the forms shown in Figure \ref{f13}. To determine
the correct graph we observe that there is no non-trivial action of $\mathrm{SL}_2(\mathbb{F}_4)\cong A_5$
on the prism in Figure \ref{f13}(B).\qed
\vspace{2mm}

\begin{figure}
\[
\unitlength 1mm 
\linethickness{0.4pt}
\ifx\plotpoint\undefined\newsavebox{\plotpoint}\fi 
\begin{picture}(20,26.75)(65,0)
\put(69.75,21){\line(1,0){9.75}}
\put(79.5,21){\line(3,-4){3}}
\multiput(82.5,17)(-.091346154,-.033653846){104}{\line(-1,0){.091346154}}
\multiput(73,13.5)(-.065298507,.03358209){134}{\line(-1,0){.065298507}}
\multiput(64.25,18)(.06179775,.03370787){89}{\line(1,0){.06179775}}
\multiput(70,10.25)(.03353659,-.07317073){82}{\line(0,-1){.07317073}}
\multiput(72.75,4.25)(.046474359,.033653846){156}{\line(1,0){.046474359}}
\multiput(80,9.5)(-.19776119,-.03358209){67}{\line(-1,0){.19776119}}
\multiput(66.75,7.25)(.26923077,-.03365385){52}{\line(1,0){.26923077}}
\multiput(80.75,5.5)(-.078358209,.03358209){134}{\line(-1,0){.078358209}}
\multiput(66.75,7)(-.03333333,.15){75}{\line(0,1){.15}}
\multiput(72.75,4)(.03125,1.15625){8}{\line(0,1){1.15625}}
\multiput(80.75,5.25)(.03365385,.22596154){52}{\line(0,1){.22596154}}
\put(79.5,9.5){\line(0,1){11.75}}
\multiput(69.75,10.25)(-.03125,1.34375){8}{\line(0,1){1.34375}}
\put(67.25,19.25){\vector(-1,2){2.5}}
\put(75,26.75){\vector(0,1){.07}}\multiput(74.75,20.75)(.03125,.75){8}{\line(0,1){.75}}
\put(81,19.25){\vector(1,1){4}}
\put(86,20.25){\vector(3,2){.07}}\multiput(77.75,15)(.052884615,.033653846){156}{\line(1,0){.052884615}}
\put(61,16.25){\vector(-1,0){.07}}\multiput(68.5,15.5)(-.326087,.0326087){23}{\line(-1,0){.326087}}
\put(69.5,14){\vector(-1,0){6}}
\put(65.25,12.25){\vector(-1,0){6.25}}
\put(78.25,11){\vector(1,0){.07}}\multiput(73,10.5)(.35,.0333333){15}{\line(1,0){.35}}
\put(87.75,11.75){\vector(1,0){.07}}\multiput(82,12.25)(.3833333,-.0333333){15}{\line(1,0){.3833333}}
\put(79.5,14){\vector(1,0){6.5}}
\put(65.5,4){\vector(-2,-1){.07}}\multiput(71.25,7.25)(-.059278351,-.033505155){97}{\line(-1,0){.059278351}}
\put(85.5,9){\vector(4,1){.07}}\multiput(75.75,7.75)(.25657895,.03289474){38}{\line(1,0){.25657895}}
\put(79.75,3){\vector(1,-1){.07}}\multiput(76,6.5)(.036057692,-.033653846){104}{\line(1,0){.036057692}}
\put(78.75,1.75){\vector(1,-1){.07}}\multiput(74,6.25)(.035447761,-.03358209){134}{\line(1,0){.035447761}}
\put(70.75,12.75){\vector(-1,4){.07}}\multiput(72,8)(-.03289474,.125){38}{\line(0,1){.125}}
\put(69.05,19.8){$\bullet$}
\put(78.65,19.75){$\bullet$}
\put(72.1,12.65){$\bullet$}
\put(81.6,15.95){$\bullet$}
\put(63.65,17){$\bullet$}
\put(72.1,3.5){$\bullet$}
\put(79.7,4.93){$\bullet$}
\put(66.1,6.25){$\bullet$}
\put(78.75,8.85){$\bullet$}
\put(69.15,9.2){$\bullet$}
\put(74,19.1){*}
\put(66.35,17.6){*}
\put(80.05,17.4){*}
\put(76.85,13.25){*}
\put(68,13.5){*}
\put(81.05,10.45){*}
\put(72,8.55){*}
\put(68.75,12.25){*}
\put(78.75,12.25){*}
\put(64.75,10.35){*}
\put(71.2,6.3){*}
\put(70.25,5.25){*}
\put(73.15,4.5){*}
\put(75.35,4.75){*}
\put(75,6){*}
\put(60,1.95){A}
\end{picture}\qquad\qquad\qquad
\unitlength 1mm 
\linethickness{0.4pt}
\ifx\plotpoint\undefined\newsavebox{\plotpoint}\fi 
\begin{picture}(20,26.75)(65,0)
\put(69.75,21){\line(1,0){9.75}}
\put(79.5,21){\line(3,-4){3}}
\multiput(82.5,17)(-.091346154,-.033653846){104}{\line(-1,0){.091346154}}
\multiput(73,13.5)(-.065298507,.03358209){134}{\line(-1,0){.065298507}}
\multiput(64.25,18)(.06179775,.03370787){89}{\line(1,0){.06179775}}
\multiput(66.75,7)(-.03333333,.15){75}{\line(0,1){.15}}
\multiput(72.75,4)(.03125,1.15625){8}{\line(0,1){1.15625}}
\multiput(80.75,5.25)(.03365385,.22596154){52}{\line(0,1){.22596154}}
\put(79.5,9.5){\line(0,1){11.75}}
\multiput(69.75,10.25)(-.03125,1.34375){8}{\line(0,1){1.34375}}
\put(67.25,19.25){\vector(-1,2){2.5}}
\put(75,26.75){\vector(0,1){.07}}\multiput(74.75,20.75)(.03125,.75){8}{\line(0,1){.75}}
\put(81,19.25){\vector(1,1){4}}
\put(86,20.25){\vector(3,2){.07}}\multiput(77.75,15)(.052884615,.033653846){156}{\line(1,0){.052884615}}
\put(61,16.25){\vector(-1,0){.07}}\multiput(68.5,15.5)(-.326087,.0326087){23}{\line(-1,0){.326087}}
\put(69.5,14){\vector(-1,0){6}}
\put(65.25,12.25){\vector(-1,0){6.25}}
\put(87.75,11.75){\vector(1,0){.07}}\multiput(82,12.25)(.3833333,-.0333333){15}{\line(1,0){.3833333}}
\put(79.5,14){\vector(1,0){6.5}}
\put(69.05,19.8){$\bullet$}
\put(78.65,19.75){$\bullet$}
\put(72.1,12.65){$\bullet$}
\put(81.6,15.95){$\bullet$}
\put(63.65,17){$\bullet$}
\put(72.1,3.5){$\bullet$}
\put(79.7,4.93){$\bullet$}
\put(66.1,6.25){$\bullet$}
\put(78.75,8.85){$\bullet$}
\put(69.15,9.2){$\bullet$}
\put(74,19.1){*}
\put(66.35,17.6){*}
\put(80.05,17.4){*}
\put(76.85,13.25){*}
\put(68,13.5){*}
\put(81.05,10.45){*}
\put(72,6.85){*}
\put(68.75,12.25){*}
\put(78.75,12.25){*}
\put(64.75,10.35){*}
\put(68.35,3.45){*}
\put(67.35,6.75){*}
\put(75,2.6){*}
\put(79.15,5.5){*}
\put(60,1.95){B}
\multiput(66.25,7.25)(.059278351,-.033505155){97}{\line(1,0){.059278351}}
\multiput(69.5,10)(-.03651685,-.03370787){89}{\line(-1,0){.03651685}}
\multiput(70,10.25)(.6166667,-.0333333){15}{\line(1,0){.6166667}}
\multiput(79.25,9.75)(.03289474,-.11184211){38}{\line(0,-1){.11184211}}
\multiput(80.5,5.5)(-.21052632,-.03289474){38}{\line(-1,0){.21052632}}
\put(78,.75){\vector(1,-2){.07}}\multiput(76,4.25)(.03333333,-.05833333){60}{\line(0,-1){.05833333}}
\put(79.75,7.75){\vector(2,-1){6}}
\put(66.5,1.25){\vector(-2,-3){.07}}\multiput(69,5.5)(-.03333333,-.05666667){75}{\line(0,-1){.05666667}}
\put(61.75,9.35){\vector(-1,0){.07}}\multiput(68,8.85)(-.4166667,.0333333){15}{\line(-1,0){.4166667}}
\put(70.5,5.25){\vector(-2,-3){.07}}\multiput(73,9.5)(-.03333333,-.05666667){75}{\line(0,-1){.05666667}}
\put(75,10){\vector(0,-1){4}}\put(74.1,8.2){*}
\end{picture}
\]
\caption{Two symmetric ways to complete the pentagon $10$-cycle in the last case in 
the proof of Theorem \ref{t12}.}\label{f13}
\end{figure}
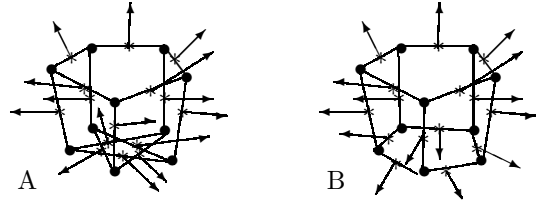

Next we use the previous result to compute quotient graphs for
the Hecke congruence subgroup of level $f$, namely
$$\Gamma_0(f)=\left\{\bmattrix abcd\in\Gamma(1)\Bigg| c\equiv 0\ (\mathrm{mod}\ f)\right\}.$$

\begin{example}
When $f=\mathbf{t}$, the graph $\Gamma_0(\mathbf{t})\backslash\mathfrak{t}(K)$ is obtained by letting the group
$$\Gamma_0(\mathbf{t})/\Gamma(\mathbf{t})\cong
\left\{\bmattrix ab0{a^{-1}}\Big|
a\in\mathbb{F}^*,b\in\mathbb{F}\right\}$$
 act on the graph in Fig. \ref{f10}(C). It is 
 not hard to see 
 that the quotient graph is isomorphic to the
 real line graph in this case. 
 Moreover, all cusps corresponding 
 to visual limits $a_i\in\mathbb{F}$, for the fundamental region 
 in Figure \ref{f10}(C), belong to the same 
 orbit.
\end{example}

\begin{lemma}\label{l51}
$\Gamma_0(\mathbf{t}^2)$ is a normal subgroup 
of index $2$ in $\Gamma_0(\mathbf{t})$.
\end{lemma}

\begin{proof}
It suffices to compute the index, since every 
subgroup of index $2$
is normal. Recall that 
$R=\mathcal{O}/\mathbf{t}^2\mathcal{O}
\cong\mathbb{F}_2[u]$,
where $u^2=0$. We can work in the quotient 
$\Gamma(1)/\Gamma(\mathbf{t}^2)\cong
\mathrm{SL}_2(R)$. Note that all 
matrices in $\Gamma(1)$ have determinant
one. The elements of   
$\Gamma_0(\mathbf{t})/\Gamma(\mathbf{t}^2)$ 
are the matrices of the form $\sbmattrix abcd$,
where
$a$ is invertible, $c\in\{0,u\}$ and 
$d=a^{-1}(1+bc)$. It follows that $|\Gamma_0(\mathbf{t})/\Gamma(\mathbf{t}^2)|=16$.
The computation of  $|\Gamma_0(\mathbf{t}^2)/\Gamma(\mathbf{t}^2)|$ is similar, but in this case $c=0$, so we get
$|\Gamma_0(\mathbf{t}^2)/\Gamma(\mathbf{t}^2)|=8$.
\end{proof}

\begin{proposition}\label{p52}
When $q=2$, and $f$ is a quadratic polynomial, 
the WRFQ for $\Gamma_0(f)$ looks 
like the one in the corresponding entry of Table \ref{table2}.
\end{proposition}

Note that, for the fundamental region with the cusps given 
in Table \ref{table2}, the lifting of the type-$0$ leaf, for 
$f=\mathbf{t}^2+\mathbf{t}+1$, can only be the ball $B_{1/\mathbf{t}}^{[2]}$.
 
\begin{table}
\[
\begin{tabular}{|l|c|}
  \hline
  $f$ & $\Gamma_0(f)|\mathfrak{t}_P$ \\
\hline
  $\mathbf{t}$ &
\xygraph{
!{<0cm,0cm>;<.8cm,0cm>:<0cm,.8cm>::}
!{(4,1)}*+{\bullet}="b"
 !{(4,.5)}*+{0}
!{(2,1)}*+{\star}="i1"!{(1.8,0.8)}*+{{}_{0}}
!{(6,1)}*+{\star}="i3" !{(6.3,1)}*+{{}_\infty}
!{(3.8,1.2)}*+{{}^q}!{(4.2,1.2)}*+{{}^1}
"b"-@{.>}"i1"  "b"-@{.>}"i3" } 
\\ \hline
  $\mathbf{t}(\mathbf{t}+1)$ & 
\xygraph{
!{<0cm,0cm>;<.8cm,0cm>:<0cm,.8cm>::}
!{(3,2)}*+{\bullet}="a" !{(3,2.5)}*+{1}
!{(4,2)}*+{\bullet}="b" !{(4,2.5)}*+{0}
!{(5,2)}*+{\bullet}="c" !{(5,2.5)}*+{1}
!{(4,1)}*+{\bullet}="d" !{(4,0.5)}*+{1}
!{(5,1)}*+{\bullet}="e" !{(5,0.5)}*+{0}
!{(6,1)}*+{\bullet}="f" !{(6,0.5)}*+{1}
!{(1,2)}*+{\star}="i1" !{(0.8,1.8)}*+{{}_{\frac1{\mathbf{t}+1}}}
!{(8,2)}*+{\star}="i2" !{(8.3,2)}*+{{}_{\frac1{\mathbf{t}}}}
!{(1,1)}*+{\star}="i3" !{(0.8,0.8)}*+{{}_{0}}
!{(8,1)}*+{\star}="i4" !{(8.3,1)}*+{{}_\infty}
!{(3.8,0.9)}*+{{}_1}!{(4.2,1.1)}*+{{}^1}!{(4.8,1.1)}*+{{}^2}
!{(5.2,1.1)}*+{{}^1}!{(5.8,1.1)}*+{{}^2}!{(6.2,1.1)}*+{{}^1}
!{(2.8,2.1)}*+{{}^1}!{(5.2,2.1)}*+{{}^1}
!{(3.2,2.1)}*+{{}^2}!{(3.8,2.1)}*+{{}^1}
!{(4.2,2.1)}*+{{}^1}!{(4.8,2.1)}*+{{}^2}
!{(3.9,1.7)}*+{{}^1}!{(3.9,1.3)}*+{{}_1}
 "a"-"b" "b"-"c" "b"-"d" "d"-"e" "e"-"f"
"a"-@{.>}"i1" "c"-@{.>}"i2" "d"-@{.>}"i3" "f"-@{.>}"i4"
} 
\\  \hline
  $\mathbf{t}^2$ &
\xygraph{
!{<0cm,0cm>;<.8cm,0cm>:<0cm,.8cm>::}
!{(2,1)}*+{\bullet}="a" !{(2,0.5)}*+{1}
!{(3,1)}*+{\bullet}="b" !{(3,0.5)}*+{0}
!{(4,1)}*+{\bullet}="c" !{(4,0.5)}*+{1}
!{(5,1)}*+{\bullet}="d" !{(5,0.5)}*+{0}
!{(6,1)}*+{\bullet}="e" !{(6,0.5)}*+{1}
!{(1,1)}*+{\star}="i1" !{(0.8,0.8)}*+{{}_{\frac1{\mathbf{t}}}}
!{(4,2)}*+{\star}="i2" !{(4.3,2)}*+{{}_0}
!{(8,1)}*+{\star}="i3" !{(8.3,1)}*+{{}_\infty}
!{(1.8,1.1)}*+{{}^1}!{(2.2,1.1)}*+{{}^2}
!{(2.8,1.1)}*+{{}^1}!{(3.2,1.1)}*+{{}^2}
!{(3.8,1.1)}*+{{}^1}!{(4.2,1.1)}*+{{}^1}!{(4.8,1.1)}*+{{}^2}
!{(5.2,1.1)}*+{{}^1}!{(5.8,1.1)}*+{{}^2}!{(6.2,1.1)}*+{{}^1}
 "a"-"b" "b"-"c" "c"-"d" "d"-"e" 
"a"-@{.>}"i1" "c"-@{.>}"i2" "e"-@{.>}"i3" } 
\\ \hline
$\mathbf{t}^2+\mathbf{t}+1$ & 
\xygraph{
!{<0cm,0cm>;<.8cm,0cm>:<0cm,.8cm>::}
!{(3,1)}*+{\bullet}="b" !{(3,0.5)}*+{0}
!{(4,1)}*+{\bullet}="c" !{(4,0.5)}*+{1}
!{(5,1)}*+{\bullet}="d" !{(5,0.5)}*+{0}
!{(6,1)}*+{\bullet}="e" !{(6,0.5)}*+{1}
!{(4,2)}*+{\star}="i2" !{(4.3,2)}*+{{}_0} 
!{(8,1)}*+{\star}="i3" !{(8.3,1)}*+{{}_\infty}
!{(5.8,1.1)}*+{{}^2}!{(6.2,1.1)}*+{{}^1}!{(3.2,1.1)}*+{{}^3}
!{(3.8,1.1)}*+{{}^1}!{(4.2,1.1)}*+{{}^1}!{(4.8,1.1)}*+{{}^2}
!{(5.2,1.1)}*+{{}^1}
"b"-"c" "c"-"d" "d"-"e" "c"-@{.>}"i2" "e"-@{.>}"i3" } 
\\  \hline
\end{tabular}
\]
\caption{Some quotient graphs with its cusp representatives.}\label{table2}
\end{table}

\begin{proof}
We consider the quotient group $H=\Gamma_0(f)/\Gamma(f)$ as a subgroup
of $\mathrm{SL}_2(R)$, for $R=\mathcal{O}/f\mathcal{O}$
as before. Their elements
are the matrices of the form $\sbmattrix ab0{a^{-1}}$, 
where $a\in R^*$ and $b\in R$.
In particular, this group is isomorphic to the semidirect product $R\ltimes R^*$, where $R^*$
acts on $R$ by $a*b=a^2b$.
 Our strategy consist in identifying $H$ as a subgroup of the automorphism group 
$\mathrm{Aut}(\mathfrak{o})$ of the graph 
$\mathfrak{o}$ to compute the corresponding quotient graph.

We note that each Moebius transformation sending $\infty$ to $0$ must have the form
$\mu(z)=\frac a{z+b}$, which corresponds to a matrix that do not belong to $\Gamma_0(f)$
for $\deg(f)>0$. So $\infty$ and $0$ are representatives of different visual limits. Furthermore,
the maximal geodesic joining them has a unique vertex of type $0$, namely $B_0^{[0]}$,
and all vertices in the ray joining it with $\infty$ have valency $2$, as all translations
$z\mapsto z+h$, where $h$ is a polynomial, belong to $\Gamma_0(f)$. This is sufficient
to place the visual limits corresponding to $0$ and $\infty$ in each case. 

When $f=\mathbf{t}^2+\mathbf{t}+1$, we have $R\cong\mathbb{F}_4=\mathbb{F}_2(\omega)$, where $\omega$ is a primitive
cubic root of unity, whence $R^*$ acts non-trivially on $R$ which is isomorphic to the Klein group. We
conclude that $H\cong A_4$. We note that $\mathrm{Aut}(\mathfrak{o})\cong S_5$, and the elements
of the subgroup $A_5$, the only subgroup of order $60$, 
can be identified with symmetries of the dodecahedron in a natural way.
We recall that any copy of $A_4$ inside the latter symetry group is the stabilizer of a triplet
of mutually orthogonal pairs of opposite edges. We denote their images by a $\star$ in Figure \ref{f15}(A)
and call them $\star$-edges. Their vertices are called $a$-vertices.
 The elements of order $2$ correspond to rotations around an axis 
through a $\star$-edge, while the elements of order $3$ are rotations $\rho_b$ around an axis 
crossing a non-$a$-vertex, a $b$-vertex (c.f. Fig. \ref{f15}(A)). Each rotation $\rho_b$
permute transitively the edges adjacent to a 
$b$-vertex, i.e., permutes the three 
neighboring $a$-vertices,
and also the three $b$-vertices beyond each of 
the former. We conclude that all $a$-vertices 
are in a unique orbit,
as are the $b$-vertices.
The $\star$-edges are in a unique orbit
since there is only one of them
adjacent to each $a$-vertex, and its image
(in the fine quotient of $\mathfrak{o}$)
is a half edge for that reason. 
The remaining edges are also in a unique orbit
since the rotation around a $b$-vertex
identify all adjacent edges. 
The result follows in this
case.

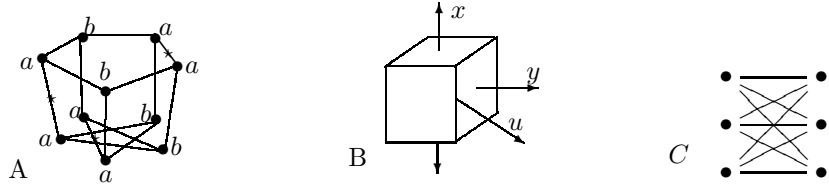
\begin{figure}
\[
\unitlength 1mm 
\linethickness{0.4pt}
\ifx\plotpoint\undefined\newsavebox{\plotpoint}\fi 
\begin{picture}(20,26.75)(65,0)
\put(69.75,21){\line(1,0){9.75}}
\put(79.5,21){\line(3,-4){3}}
\multiput(82.5,17)(-.091346154,-.033653846){104}{\line(-1,0){.091346154}}
\multiput(73,13.5)(-.065298507,.03358209){134}{\line(-1,0){.065298507}}
\multiput(64.25,18)(.06179775,.03370787){89}{\line(1,0){.06179775}}
\multiput(70,10.25)(.03353659,-.07317073){82}{\line(0,-1){.07317073}}
\multiput(72.75,4.25)(.046474359,.033653846){156}{\line(1,0){.046474359}}
\multiput(80,9.5)(-.19776119,-.03358209){67}{\line(-1,0){.19776119}}
\multiput(66.75,7.25)(.26923077,-.03365385){52}{\line(1,0){.26923077}}
\multiput(80.75,5.5)(-.078358209,.03358209){134}{\line(-1,0){.078358209}}
\multiput(66.75,7)(-.03333333,.15){75}{\line(0,1){.15}}
\multiput(72.75,4)(.03125,1.15625){8}{\line(0,1){1.15625}}
\multiput(80.75,5.25)(.03365385,.22596154){52}{\line(0,1){.22596154}}
\put(79.5,9.5){\line(0,1){11.75}}
\multiput(69.75,10.25)(-.03125,1.34375){8}{\line(0,1){1.34375}}
\put(69.05,19.8){$\bullet$}
\put(78.65,19.75){$\bullet$}
\put(72.1,12.65){$\bullet$}
\put(81.6,15.95){$\bullet$}
\put(63.65,17){$\bullet$}
\put(72.1,3.5){$\bullet$}
\put(79.7,4.93){$\bullet$}
\put(66.1,6.25){$\bullet$}
\put(78.75,8.85){$\bullet$}
\put(69.15,9.2){$\bullet$}
\put(70,21){$b$}
\put(80,21){$a$}
\put(72.1,15){$b$}
\put(83.5,16){$a$}
\put(61.5,16.5){$a$}
\put(72,1.5){$a$}
\put(81.5,5){$b$}
\put(64,6.25){$a$}
\put(77.5,9.5){$b$}
\put(68,10){$a$}
\put(80.5,18.5){${}_\star$}
\put(65,12.5){${}_\star$}
\put(70.8,7.2){${}_\star$}
\put(60,1.95){A}
\end{picture}
\qquad
\unitlength .8mm 
\linethickness{0.4pt}
\ifx\plotpoint\undefined\newsavebox{\plotpoint}\fi 
\begin{picture}(59.5,31.75)(0,0)
\put(22.75,8.5){\framebox(11.5,12.5)[cc]{}}
\multiput(22.5,21)(.048657718,.033557047){149}{\line(1,0){.048657718}}
\multiput(34,21)(.048657718,.033557047){149}{\line(1,0){.048657718}}
\put(29.75,26){\line(1,0){11}}
\put(41.2,26){\line(0,-1){12.5}}
\multiput(34.2,8.5)(.048657718,.033557047){149}{\line(1,0){.048657718}}
\put(31.5,23.75){\vector(0,1){8}}
\put(31.25,8.25){\vector(0,-1){5}}
\put(37.75,17.5){\vector(1,0){10.25}}
\put(45.75,8.25){\vector(3,-2){.07}}\multiput(34.25,15.75)(.051569507,-.033632287){223}{\line(1,0){.051569507}}
\put(34.75,30){\makebox(0,0)[cc]{$x$}}
\put(47.25,19.5){\makebox(0,0)[cc]{$y$}}
\put(44.25,11.5){\makebox(0,0)[cc]{$u$}}
\put(18,6){\makebox(0,0)[cc]{B}}
\end{picture}
\qquad
\xygraph{
!{(0,0.7)}*+{C}="ns"
!{(0.5,0.5)}*+{\bullet}="a1"
!{(0.5,1)}*+{\bullet}="a2"
!{(0.5,1.5)}*+{\bullet}="a3"
!{(1.5,0.5)}*+{\bullet}="b1"
!{(1.5,1)}*+{\bullet}="b2"
!{(1.5,1.5)}*+{\bullet}="b3"
"a1"-"b1" "a1"-"b2" "a1"-"b3"
"a2"-"b1" "a2"-"b2" "a2"-"b3"
"a3"-"b1" "a3"-"b2" "a3"-"b3"
}
\]
\caption{Three pictures used in the proof of Proposition \ref{p52}.}
\label{f15}\end{figure}

When $f=\mathbf{t}^2$, we have $R=\mathbb{F}_2[u]$
with $u^2=0$, as in the proof of Lemma \ref{l51}, and therefore 
$R^*= \{1,1+u\}\cong C_2$ acts trivially on $R$. On the other hand, 
$\Gamma(1)/\Gamma(\mathbf{t}^2)\cong\mathrm{SL}_2(R)$ has order $48$, by the
description given in the proof of Theorem \ref{t12}.
In particular, it is the full
symmetry group $G$ of the cube, which is isomorphic to 
$C_2\times S_4$. The $2$-sylow subgroup
of the latter group  is $C_2\times D_4$, which has two non-conjugate subgroups isomorphic
to $C_2\times C_2\times C_2$. In fact, if we set $D_4=\langle a,b|a^4=b^2=abab=e\rangle$ for
the last factor, and $C_2=\langle c|c^2=e\rangle$ for the first, 
then $c$ corresponds to the antipodal map, as it is central, 
the matrix $a$ to a rotation in
$90$ degrees on a axis passing by the center of a face, 
like the axis $x$ in Figure \ref{f15}(B),
and $b$ to a $180$-degrees rotation on
an orthogonal axis $y$. To fix ideas, we assume also
$y$ passes through the center of a face,
as in Fig. \ref{f15}(B). Then $ab$ is a $180$-degrees 
rotation on an axis passing
through an edge, like axis $u$ in Fig. \ref{f15}(B). Then 
$H_1=\langle c,a^2,b\rangle$ and $H_2=\langle c,a^2,ab\rangle$ 
are representatives of the two conjugacy classes of subgroups 
isomorphic to 
$C_2\times C_2\times C_2$. The group $H_1$ is normal in $G$,
as it is generated by all 180-degrees rotations around an axis 
crossing a face, plus the antipodal. However, $H_2$ is not so, 
since it contains a 180-degrees rotation around an axis crossing
an edge precisely when the axis is perpendicular to $x$.
Since $\Gamma_0(\mathbf{t}^2)/\Gamma(\mathbf{t}^2)$ is not normal 
in $\Gamma(1)/\Gamma(\mathbf{t}^2)$, it must correspond to
a group that is conjugate to $H_2$.
 Now the quotient graph is easy to compute.
 In fact, there are two orbits of vertices, the vertices 
 in one orbit are precisely
 those in a vertical plane $\Pi$ containing the axis $u$.
 Furthermore, $ca^3b$ is a reflection on $\Pi$, while
 $cab$ is a reflection on a vertical plane that is perpendicular
 to $\Pi$. The cusps at $0$ and $\infty$ are placed by the
 general argument given earlier.
To find a representative of the remaining cusp, we observe that 
$\Gamma_0(\mathbf{t})/\Gamma_0(\mathbf{t}^2)\cong C_2$ act on this graph, and a representative of
the non-trivial element is $\sbmattrix 10{\mathbf{t}}1$, 
whose associated Moebius transformation
is given by $\tau(z)=\frac z{\mathbf{t}z+1}$. We conclude that a representative of the cusp in
the left is $\tau(\infty)=\frac1{\mathbf{t}}$. 

When $f=\mathbf{t}(\mathbf{t}+1)$, we have $R\cong\mathbb{F}_2\times\mathbb{F}_2$, whence $R^*$ is trivial.
In this case $\mathrm{Aut}(\mathfrak{o})$ is a 
semidirect product of the form 
$C_2\ltimes(S_3\times S_3)$,
where $C_2$ acts by permuting the coordinates. 
Its $2$-sylow subgroup has the form
$C_2\ltimes(C_2\times C_2)\cong D_4$. As before, there are two
subgroups isomorphic to $C_2\times C_2$. One is contained in $S_3\times S_3$.
We claim that this cannot be the one corresponding to $\Gamma_0(f)$.
In fact, the element $\sbmattrix 1{\mathbf{t}}01\in\Gamma_0(f)$ stabilizes a type-$1$ vertex, 
while it switches the two neighboring type-$0$  vertices. This implies that the
image of $\Gamma_0(f)$ cannot stabilize the
natural partition of vertices in $K_{3,3}$, 
which correspond to the type-$0$ vertices.
We conclude that it is, up to conjugacy, the group generated by two automorphisms of $K_{3,3}$
that, when drawn as in Figure 7C, can be defined as the reflection 
on the vertical axis and the reflection on
the horizontal axis.  This gives the required quotient graph. Now we note that, in the geodesic whose visual limits are $0$
and $\infty$, the type-$1$ vertex next to the cusp corresponding to $0$ has the ball $B_0^{[1]}$ as
a representative, so the type-$1$ vertices on the upper line of the graph have representatives
$B_{1/\mathbf{t}}^{[3]}$ and 
$B_{1/(\mathbf{t}+1)}^{[3]}$. There is a unique ray from each of these vertices with
increasing types, and it can be seen, as in the previous case, that the corresponding visual limits are $1/\mathbf{t}$
and $1/(\mathbf{t}+1)$. 
\end{proof}

Now we evaluate the WRFQ $\mathfrak{wq}_P(\mathbb{O}_D)$ for 
the genus of Eichler orders $\mathbb{O}_D$, whose level $D=\mathrm{div}(f)$
 is the divisor of zeros of one of the four 
polynomials considered in this section. 
In order to do this, we note that 
$\Gamma_0(f)$ is precisely the ring of units of
an Eichler $A$-order $E(f)$ of level $f$. 
It suffices now to apply to each of these graphs the action of the 
quotient group $N(f)/\Gamma_0(f)$, where  $N(f)$ is the normalizer 
of $E(f)$.
Although $A$ is a polynomial ring in all of this section,
next lemma is written in more generality, as it is used in the
last example of \S8.

\begin{lemma}\label{lem83}
Assume that $P\in|X|$ is chosen in a way that the ring 
$A=\oink_X(U)$ is a principal ideal domain. Let $f\in A$,
let $E(f)$ be the Eichler $A$-order described above, $N(f)$ its normalizer
and $n\in N(f)$. 
 If conjugation by $n$ fixes every maximal order containing $E(f)$, then $n\in\Gamma_0(f)$.  
\end{lemma}

\begin{proof}
If $n$ fixes a maximal order $\Da_Q$ at a place $Q\neq P$, then 
$n\in K_Q^*\Da_Q^*$, whence
 $\det(n)\in K_Q^{*2}\mathcal{O}_Q^*$. In particular,
if $n$ fixes a maximal $A$-order $D$, then 
$\det(n)\in K_Q^{*2}\mathcal{O}_Q^*$ 
at every finite place. In other words, the principal divisor
$\big(\det(n)\big)$ has the form $2D'$ for some divisor
$D'$. Since $\big(\det(n)\big)$ is principal, its degree is $0$,
and the same holds for $D'$. As $A$ is principal, 
there exists a constant $\lambda$ whose valuation
$\nu_Q(\lambda)$ coincide with the valuation 
$\nu_Q(D')$ at every place $Q\neq P$, and therefore
also at $P$ since both divisors have degree $0$.
We can divide $n$ by the constant $\lambda$
and assume that $\det(n)$ is a unit. 
Then $n$ is a unit of $D$
locally at all places, and therefore also globally. As the order 
$D\supseteq E(f)$ is arbitrary, and an Eichler order
is an intersection of maximal orders, the result follows.
\end{proof}

\begin{proposition}\label{p52b}
When $q=2$ and $f$ is a quadratic polynomial, 
or when $f$ is linear for any $q$,
the WRFQ for $N(f)$ looks 
like the one in the corresponding 
entry of Table \ref{table3}.
\end{proposition}

\begin{proof}
It follows from the preceding lemma that an element in 
$N(f)/\Gamma_0(f)$ is completely determined by
its action on the maximal orders containing $E(f)$. 
If the divisor of zeros is
$\mathrm{div}(f)=k_1Q_1+k_2Q_2$, then,
according to \cite[Lemma 4.4]{scffgeo}, these maximal orders 
are in correspondence with the vertices on a 
$(k_1\times k_2)$-grid. This easily generalizes to 
polynomials with more than two prime divisors, 
but we do not need it here.
Action by an element in $N(f)/\Gamma_0(f)$ can be only 
a horizontal or 
vertical flip of the grid, or a combination 
thereof, since it preserves the places $Q_1$ and $Q_2$, 
whence it can 
only have $1$, $2$ or $4$ elements.
The grids for our examples are depicted on the right of 
Figure \ref{grids}.
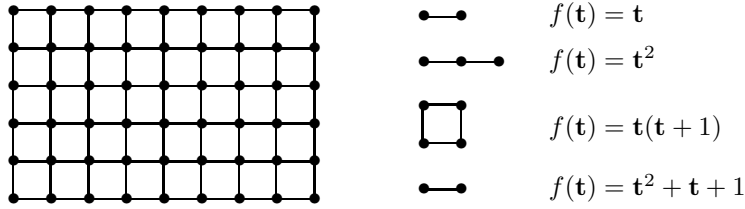
\begin{figure}
\unitlength 1mm 
\linethickness{0.4pt}
\ifx\plotpoint\undefined\newsavebox{\plotpoint}\fi 
\begin{picture}(67,29.5)(30,20)
\put(23.75,22){\framebox(40,25)[]{}}
\put(28.75,22){\line(0,1){25}}\put(33.75,22){\line(0,1){25}}
\put(38.75,22){\line(0,1){25}}
\put(43.75,22){\line(0,1){25}}\put(48.75,22){\line(0,1){25}}
\put(53.75,22){\line(0,1){25}}
\put(58.75,22){\line(0,1){25}}\put(23.75,32){\line(1,0){40}}
\put(23.75,27){\line(1,0){40}}\put(78,46.25){\line(1,0){5}}
\put(23.75,42){\line(1,0){40}}\put(23.75,37){\line(1,0){40}}
\put(23,21.25){$\bullet$}\put(28,21.25){$\bullet$}
\put(33,21.25){$\bullet$}\put(38,21.25){$\bullet$}
\put(43,21.25){$\bullet$}\put(48,21.25){$\bullet$}
\put(53,21.25){$\bullet$}\put(58,21.25){$\bullet$}
\put(63,21.25){$\bullet$}
\put(23,26.25){$\bullet$}\put(28,26.25){$\bullet$}
\put(33,26.25){$\bullet$}\put(38,26.25){$\bullet$}
\put(43,26.25){$\bullet$}\put(48,26.25){$\bullet$}
\put(53,26.25){$\bullet$}\put(58,26.25){$\bullet$}
\put(63,26.25){$\bullet$}
\put(23,31.25){$\bullet$}\put(28,31.25){$\bullet$}
\put(33,31.25){$\bullet$}\put(38,31.25){$\bullet$}
\put(43,31.25){$\bullet$}\put(48,31.25){$\bullet$}
\put(53,31.25){$\bullet$}\put(58,31.25){$\bullet$}
\put(63,31.25){$\bullet$}
\put(23,36.25){$\bullet$}\put(28,36.25){$\bullet$}
\put(33,36.25){$\bullet$}\put(38,36.25){$\bullet$}
\put(43,36.25){$\bullet$}\put(48,36.25){$\bullet$}
\put(53,36.25){$\bullet$}\put(58,36.25){$\bullet$}
\put(63,36.25){$\bullet$}
\put(23,41.25){$\bullet$}\put(28,41.25){$\bullet$}
\put(33,41.25){$\bullet$}\put(38,41.25){$\bullet$}
\put(43,41.25){$\bullet$}\put(48,41.25){$\bullet$}
\put(53,41.25){$\bullet$}\put(58,41.25){$\bullet$}
\put(63,41.25){$\bullet$}
\put(23,46.25){$\bullet$}\put(28,46.25){$\bullet$}
\put(33,46.25){$\bullet$}\put(38,46.25){$\bullet$}
\put(43,46.25){$\bullet$}\put(48,46.25){$\bullet$}
\put(53,46.25){$\bullet$}\put(58,46.25){$\bullet$}
\put(63,46.25){$\bullet$}
\put(78,40.25){\line(1,0){10}}\put(78.25,29.5){\framebox(5,5)[]{}}\put(78,23.25){\line(1,0){5}}
\put(77.5,45.5){$\bullet$}\put(82.5,45.5){$\bullet$}
\put(77.5,39.5){$\bullet$}\put(82.5,39.5){$\bullet$}\put(87.5,39.5){$\bullet$}
\put(77.5,33.5){$\bullet$}\put(82.5,33.5){$\bullet$}\put(77.5,28.5){$\bullet$}\put(82.5,28.5){$\bullet$}
\put(77.5,22.5){$\bullet$}\put(82.5,22.5){$\bullet$}
\put(95,45.5){$f(\mathbf{t})=\mathbf{t}$}\put(95,39.5){$f(\mathbf{t})=\mathbf{t}^2$}\put(95,30.5){$f(\mathbf{t})=\mathbf{t}(\mathbf{t}+1)$}\put(95,22.5){$f(\mathbf{t})=\mathbf{t}^2+\mathbf{t}+1$}
\end{picture}
\caption{A general grid for Eichler orders (left), and the particular grids for 
each examples considered here (right).}\label{grids}
\end{figure}
Generators for the group $N(f)/\Gamma_0(f)$ can be, in general, computed as Atkin-Lehner involutions. For these few cases
we can simply give the explicit matrices. The matrix $\sbmattrix0{-1}f0$ defines a flip in each direction 
(horizontal and vertical), which is the only symmetry if $f\in\{\mathbf{t},\mathbf{t}^2,\mathbf{t}^2+\mathbf{t}+1\}$.
This is equivalent to saying that it transposes the orders
$\sbmattrix AAAA$ and $\sbmattrix A{f^{-1}A}{fA}A$. 
For $f(\mathbf{t})=\mathbf{t}(\mathbf{t}+1)$ we need the additional matrix $$\bmattrix {\mathbf{t}}1{\mathbf{t}(\mathbf{t}+1)}{\mathbf{t}}=
\bmattrix 01{\mathbf{t}}0\bmattrix {(\mathbf{t}+1)}1{\mathbf{t}}1.$$
Since this matrix is congruent to $\sbmattrix1101$ modulo $\mathbf{t}+1$ it fixes
both orders locally at the place $Q_1$ defined by the relation
$\mathrm{div}(\mathbf{t}+1)=Q_1-P$.
Recall that $P$ is the place at infinity.
The factorization on the right hand side shows that it permutes them 
at the place $Q_2$ defined by $\mathrm{div}(t)=Q_2-P$. 
If we check the effect of each matrix on the Bruhat-Tits tree $\mathfrak{t}_P$, we conclude that the cusps at
$0$ and $\infty$ are always in the same orbit. This is enough to compute the quotient when $f\in\{\mathbf{t},\mathbf{t}^2,\mathbf{t}^2+\mathbf{t}+1\}$.
When $f(\mathbf{t})=\mathbf{t}(\mathbf{t}+1)$ we need to use the remaining matrix, which corresponds to the Moebius transformation
$\eta(z)=\frac{\mathbf{t}z+1}{\mathbf{t}(\mathbf{t}+1)z+\mathbf{t}}$. Since $\eta(0)=1/\mathbf{t}$ and $\eta(\infty)=1/(\mathbf{t}+1)$, we conclude that all cusps
are in the same orbit in this case. 
The result follows.
\end{proof}
 
In next section we apply Theorem \ref{t5} to the graphs 
in Table \ref{table3}
to illustrate our main results.

\begin{table}
\[
\begin{tabular}{|l|c|}
  \hline
  $f$ & $N(f)|\mathfrak{t}_P$ \\
\hline  
$\mathbf{t}$ & 
\xygraph{
!{<0cm,0cm>;<.8cm,0cm>:<0cm,.8cm>::}
!{(3,1)}*+{*}="a" 
!{(4,1)}*+{\bullet}="d" 
!{(5,1)}*+{\bullet}="e"
!{(6,1)}*+{\bullet}="f" 
!{(8,1)}*+{\star}="i4"
!{(3.8,1.2)}*+{{}^q}!{(4.2,1.2)}*+{{}^1}!{(4.8,1.2)}*+{{}^q}
!{(5.2,1.2)}*+{{}^1}!{(5.8,1.2)}*+{{}^q}
 "a"-"d" "d"-"e" "e"-"f"  "f"-@{.>}"i4"
} 
\\  \hline
  $\mathbf{t}(\mathbf{t}+1)$ & 
\xygraph{
!{<0cm,0cm>;<.8cm,0cm>:<0cm,.8cm>::}
!{(3,1)}*+{*}="a" 
!{(4,1)}*+{\bullet}="d" 
!{(5,1)}*+{\bullet}="e"
!{(6,1)}*+{\bullet}="f" 
!{(8,1)}*+{\star}="i4"
!{(3.8,1.2)}*+{{}^1}!{(4.2,1.2)}*+{{}^2}!{(4.8,1.2)}*+{{}^2}
!{(5.2,1.2)}*+{{}^1}!{(5.8,1.2)}*+{{}^2}
 "a"-"d" "d"-"e" "e"-"f"  "f"-@{.>}"i4"
} 
\\  \hline
  $\mathbf{t}^2$ &
\xygraph{
!{<0cm,0cm>;<.8cm,0cm>:<0cm,.8cm>::}
!{(2,1)}*+{\bullet}="a" 
!{(3,1)}*+{\bullet}="b"
!{(4,1)}*+{\bullet}="c" 
!{(5,1)}*+{\bullet}="d"
!{(6,1)}*+{\bullet}="e"
!{(1,1)}*+{\star}="i1"
!{(8,1)}*+{\star}="i3"
!{(2.2,1.2)}*+{{}^2}!{(2.8,1.2)}*+{{}^1}!{(3.2,1.2)}*+{{}^2}
!{(3.8,1.2)}*+{{}^1}!{(4.2,1.2)}*+{{}^2}!{(4.8,1.2)}*+{{}^2}
!{(5.2,1.2)}*+{{}^1}!{(5.8,1.2)}*+{{}^2}
 "a"-"b" "b"-"c" "c"-"d" "d"-"e" 
"a"-@{.>}"i1"  "e"-@{.>}"i3" } 
\\ \hline
$\mathbf{t}^2+\mathbf{t}+1$ & 
\xygraph{
!{<0cm,0cm>;<.8cm,0cm>:<0cm,.8cm>::}
!{(2,1)}*+{\bullet}="b" 
!{(3,1)}*+{\bullet}="c"
!{(4,1)}*+{\bullet}="d" 
!{(5,1)}*+{\bullet}="e"
!{(8,1)}*+{\star}="i3"
!{(2.2,1.2)}*+{{}^3}!{(2.8,1.2)}*+{{}^1}!{(3.2,1.2)}*+{{}^2}
!{(3.8,1.2)}*+{{}^2}!{(4.2,1.2)}*+{{}^1}!{(4.8,1.2)}*+{{}^2}
!{(5.2,1.2)}*+{{}^1}
"b"-"c" "c"-"d" "d"-"e" 
"e"-@{.>}"i3" } 
\\  \hline
\end{tabular}
\]
\caption{Some quotient graphs to be used in the main
examples.}\label{table3}
\end{table} 

\section{Some explicit quotient at $Q$}

We start this section by computing, in detail,
the WRFQ associated to the spinor genera of Eichler orders 
of level $2P_0$, where 
$P_0\in\mathbb{P}^1$ has degree 1, 
and $q=2$, at a place $Q$ of degree 
$3$. First we note that Theorem 
\ref{t5} does not apply directly, as the divisor $2P$ is 
not multiplicity free. 
However, the description of cusps in that quotient is still 
explicit enough to
determine unambiguously the vertices in each cusp. In fact, if we label vertices in
the WRFQ in the third entry of Table \ref{table3}
as shown in Figure \ref{f17}, in a way that $d_1,d_2,d_3,\dots$ correspond to the
split orders, then $c$ corresponds to an Eichler order 
contained in two maximal orders
isomorphic to $\mathfrak{D}_{P_\infty}$ and one 
isomorphic to $\mathfrak{D}_0$,
while $u_i$ corresponds, for any $i\geq1$, to an 
Eichler order contained in one
maximal order isomorphic to $\mathfrak{D}_{iP_\infty}$ and two isomorphic to
$\mathfrak{D}_{(i-1)P_\infty}$. If $\mathbf{N}$ is the associated linear map,
corresponding neighborhood matrix $N=N_{P_\infty}(\mathbb{O}_{2P_0})$, then the image, under $\mathbf{N}^3$, of the unitary charge $\delta_v$
supported on a vertex $v$ is contained in the space generated by the unitary charges 
supported on four vertices. The two neighbors of $v$ and the two vertices at distance 
three from $v$ at either side. To fix ideas, assume that $v=c$. There are three
promenades of length $3$ going from 
$c$ to $d_1$. We can go to $d_2$ and then 
take a step back. We can first go to $u_1$ and 
then go directly to $d_1$. Or we can go back 
and forth between $c$ and $d_1$ three times. The first promenade has a total weight of $2$,
since the edge from $c$ to $d_1$ has a weight of 1, the edge from $d_1$ to $d_2$ has 
a weight of 1 and the edge from $d_2$ to $d_1$ has a weight of 2. Similarly,
the second promenade has a weight of 4, 
and the last one a weight of 2. 
This gives a total
weight of 8 for the edge from $c$ to $d_1$.
We perform analogous computation for other 
vertices. If we write $\delta_v$ for the
unitary charge supported on the vertex $v$, the linear map $\mathbf{T}$
corresponding to $N$ satisfies
$$ (\mathbf{T}^3-6\mathbf{T})(\delta_c)=
(8\delta_{d_1}+16\delta_{u_1}+\delta_{d_3}+2\delta_{u_3})
-6(\delta_{d_1}+2\delta_{u_1})=
2\delta_{d_1}+4\delta_{u_1}+\delta_{d_3}+2\delta_{u_3}.$$
Analogous computations give
$$ (\mathbf{T}^3-6\mathbf{T})(\delta_{u_1})=
\delta_{u_4}+2\delta_{u_2}+2\delta_{d_2}+4\delta_{c},\qquad
(\mathbf{T}^3-6\mathbf{T})(\delta_{u_2})=
\delta_{u_5}+4\delta_{d_1}+4\delta_{u_1},$$
$$ (\mathbf{T}^3-6\mathbf{T})(\delta_{d_1})=
\delta_{d_4}+4\delta_{u_2}+4\delta_{c},\qquad
(\mathbf{T}^3-6\mathbf{T})(\delta_{d_2})=
\delta_{d_5}+8\delta_{u_1},$$
and finally
$$ (\mathbf{T}^3-6\mathbf{T})(\delta_{u_i})=
\delta_{u_{i+1}}+8\delta_{u_{i-1}},\qquad
(\mathbf{T}^3-6\mathbf{T})(\delta_{d_i})=
\delta_{d_{i+1}}+8\delta_{d_{i-1}}$$
for $i\geq3$.
\begin{figure}
$$\xygraph{
!{<0cm,0cm>;<.8cm,0cm>:<0cm,.8cm>::}
!{(2,1)}*+{\bullet}="a" !{(2,0.5)}*+{d_2} 
!{(3,1)}*+{\bullet}="b" !{(3,0.5)}*+{d_1} 
!{(4,1)}*+{\bullet}="c0" !{(4,0.5)}*+{c} 
!{(5,1)}*+{\bullet}="d" !{(5,0.5)}*+{u_1} 
!{(6,1)}*+{\bullet}="e" !{(6,0.5)}*+{u_2} 
!{(1,1)}*+{\star}="i1"
!{(8,1)}*+{\star}="i3"
!{(2.2,1.2)}*+{{}^2}!{(2.8,1.2)}*+{{}^1}!{(3.2,1.2)}*+{{}^2}
!{(3.8,1.2)}*+{{}^1}!{(4.2,1.2)}*+{{}^2}!{(4.8,1.2)}*+{{}^2}
!{(5.2,1.2)}*+{{}^1}!{(5.8,1.2)}*+{{}^2}
 "a"-"b" "b"-"c0" "c0"-"d" "d"-"e" 
"a"-@{.>}"i1"  "e"-@{.>}"i3" } $$
\caption{A suitable labeling for the vertices of the WRFQ associated to
Eichler orders of level $2P$ at a degree one place.}\label{f17}
\end{figure}
This gives the graph illustrated in Figure \ref{f19}.
\begin{figure}
$$\xygraph{
!{<0cm,0cm>;<.8cm,0cm>:<0cm,.8cm>::}
!{(4,3)}*+{\bullet}="b0" !{(3.9,2.5)}*+{{}^{u_1}} 
!{(5,3)}*+{\bullet}="d0" !{(5,2.5)}*+{{}^{d_2}} 
!{(3,3)}*+{\star}="i10"
!{(6,3)}*+{\star}="i30"
!{(3.8,3.2)}*+{{}^1} !{(4.2,3.2)}*+{{}^2}
!{(3.4,2.6)}*+{{}^4} !{(4.3,2.6)}*+{{}^2}
!{(4.8,3.2)}*+{{}^8} !{(5.2,3.2)}*+{{}^1}
!{(4,-1)}*+{\bullet}="b1" !{(4,-0.7)}*+{{}^{d_1}} 
!{(5,-1)}*+{\bullet}="d1" !{(5.3,-0.7)}*+{{}^{u_2}} 
!{(3,-1)}*+{\star}="i11"
!{(6,-1)}*+{\star}="i31"
!{(3.8,-1.2)}*+{{}^1} !{(4.2,-1.2)}*+{{}^4}
!{(3.4,-0.7)}*+{{}^4} !{(4.7,-0.7)}*+{{}^4}
!{(4.8,-1.2)}*+{{}^4} !{(5.2,-1.2)}*+{{}^1}
!{(1,1)}*+{\bullet}="b" !{(1,0.5)}*+{{}^{d_3}} 
!{(2,1)}*+{\bullet}="c0" !{(2,0.5)}*+{{}^c} 
!{(3,1)}*+{\bullet}="d" !{(3,0.5)}*+{{}^{u_3}} 
!{(0,1)}*+{\star}="i1"
!{(4,1)}*+{\star}="i3"
!{(0.8,1.2)}*+{{}^1}!{(1.2,1.2)}*+{{}^8} !{(2.2,1.4)}*+{{}^4}
!{(1.8,1.2)}*+{{}^1}!{(2.4,0.75)}*+{{}^2}!{(2.8,1.2)}*+{{}^8}
!{(3.2,1.2)}*+{{}^1}
 "b"-"c0" "c0"-"d" 
"b"-@{.>}"i1"  "d"-@{.>}"i3"
"b0"-"d0" "b0"-@{.>}"i10"  "d0"-@{.>}"i30"
"b1"-"d1" "b1"-@{.>}"i11"  "d1"-@{.>}"i31"
"c0"-"b0" "c0"-"b1" "b0"-"d1"}
\qquad\xygraph{
!{<0cm,0cm>;<.8cm,0cm>:<0cm,.8cm>::}
!{(3,1)}*+{\bullet}="b" !{(3,0.5)}*+{d_i} 
!{(4,1)}*+{\bullet}="c" !{(4,0.5)}*+{d_{i+3}} 
!{(5,1)}*+{\bullet}="d" !{(5,0.5)}*+{d_{i+6}} 
!{(6,1)}*+{\bullet}="e" !{(6,0.5)}*+{d_{i+9}} 
!{(8,1)}*+{\star}="i3"
!{(3.8,3.2)}*+{{}^8}!{(4.2,3.2)}*+{{}^1}!{(4.8,3.2)}*+{{}^8}
!{(5.2,3.2)}*+{{}^1}!{(5.8,3.2)}*+{{}^8}
!{(3,3)}*+{\bullet}="b1" !{(3,2.5)}*+{u_i} 
!{(4,3)}*+{\bullet}="c1" !{(4,2.5)}*+{u_{i+3}} 
!{(5,3)}*+{\bullet}="d1" !{(5,2.5)}*+{u_{i+6}} 
!{(6,3)}*+{\bullet}="e1" !{(6,2.5)}*+{u_{i+9}} 
!{(8,3)}*+{\star}="i31"
!{(3.8,1.2)}*+{{}^8}!{(4.2,1.2)}*+{{}^1}!{(4.8,1.2)}*+{{}^8}
!{(5.2,1.2)}*+{{}^1}!{(5.8,1.2)}*+{{}^8}
!{(4.8,-0.7)}*+{{}_{i=1,2,3}}
 "b"-"c" "c"-"d" "d"-"e" "e"-@{.>}"i3" 
  "b1"-"c1" "c1"-"d1" "d1"-"e1" "e1"-@{.>}"i31" } 
$$
\caption{The WRFQ $\mathfrak{wq}_Q(\mathbb{O}_{2P})$, where $Q$ has degree $3$
and $P$ has degree $1$. Here $X=\mathbb{P}^1$ and $\mathbb{F}=\mathbb{F}_2$. 
The vertex $c$ has two adjacent
edges of weight $2$, but the number $2$ is only written once between them for
the sake of space. The doted lines denote cusps, whose structure
is shown on the right hand side of the picture.}\label{f19}
\end{figure}
Analogous computations produced the graphs in Figure \ref{f20}, using the other entries
in Table \ref{table3}. For $f=\mathbf{t}$ we label the actual
vertices, in order, as $d_1,d_2,\dots$.
For $f=\mathbf{t}(\mathbf{t}+1)$ we use $a,d_1,d_2,\dots$ and for $f=\mathbf{t}^2+\mathbf{t}+1$ is $a,b,d_1,d_2,\dots$.
Note that $d_1,d_2,\dots$ are the vertices in the cusp in every case.
Note that Theorem \ref{t5} applies directly to all three cases.
\begin{figure}
$$\xygraph{
!{<0cm,0cm>;<.8cm,0cm>:<0cm,.8cm>::}
!{(0,2)}*+{(A)}
!{(3,4.2)}*+{*}="b2" !{(5,4.2)}*+{*}="d2" 
!{(3,3)}*+{\bullet}="b0" !{(3.3,3.3)}*+{{}^{d_1}} 
!{(5,3)}*+{\bullet}="d0" !{(5,2.5)}*+{{}^{d_2}} 
!{(1,3)}*+{\star}="i10"
!{(7,3)}*+{\star}="i30"
!{(2.6,3.1)}*+{{}^1} !{(2.5,3.8)}*+{{}^{q^3-q^2}}
!{(5.2,3.8)}*+{{}^{q^2}} !{(3.6,2.75)}*+{{}^{q^2-q}}
!{(4.5,3.2)}*+{{}^{q^3-q^2}} !{(2.7,2.5)}*+{{}^{q-1}}
!{(5.3,3.1)}*+{{}^1}
!{(3,1)}*+{\bullet}="c0" !{(3,0.5)}*+{{}^{d_3}} 
!{(5,1)}*+{\bullet}="d" !{(5,0.5)}*+{{}^{d_6}} 
!{(7,1)}*+{\star}="i3"
!{(3.2,0.7)}*+{{}^1}
!{(2.8,1.2)}*+{{}^{q^3}} !{(4.8,1.2)}*+{{}^{q^3}}
!{(5.2,1.2)}*+{{}^1}
 "c0"-"d"  "d"-@{.>}"i3"
"b0"-"d0" "b0"-@{.>}"i10"  "d0"-@{.>}"i30"
"c0"-"b0" "b0"-"b2" "d0"-"d2"}
\qquad\xygraph{
!{<0cm,0cm>;<.8cm,0cm>:<0cm,.8cm>::}
!{(3,3)}*+{\bullet}="b" !{(3,2.5)}*+{d_i} 
!{(4,3)}*+{\bullet}="c" !{(4,2.5)}*+{d_{i+3}} 
!{(5,3)}*+{\bullet}="d" !{(5,2.5)}*+{d_{i+6}} 
!{(6,3)}*+{\bullet}="e" !{(6,2.5)}*+{d_{i+9}} 
!{(8,3)}*+{\star}="i3"
!{(3.8,3.2)}*+{{}^{q^3}}!{(4.2,3.2)}*+{{}^1}!{(4.8,3.2)}*+{{}^{q^3}}
!{(5.2,3.2)}*+{{}^1}!{(5.8,3.2)}*+{{}^{q^3}} !{(6.2,3.2)}*+{{}^1}
!{(4.8,1.7)}*+{{}_{i=1,2,6}}
 "b"-"c" "c"-"d" "d"-"e" "e"-@{.>}"i3"  } 
$$
$$
(B)\ \xygraph{
!{<0cm,0cm>;<.8cm,0cm>:<0cm,.8cm>::}
!{(2,2)}*+{*}="a1" !{(0,0)}*+{*}="a2"
!{(1,-1)}*+{\bullet}="b1" !{(1,-1.5)}*+{{}^{d_1}} 
!{(2,-1)}*+{\bullet}="d1" !{(2,-1.5)}*+{{}^{d_2}} 
!{(0,-1)}*+{\star}="i11"
!{(3,-1)}*+{\star}="i31"
!{(0.8,-1.2)}*+{{}^1} !{(1.2,-1.2)}*+{{}^2}
!{(0.4,-0.7)}*+{{}^4} !{(2.2,-0.7)}*+{{}^4}
!{(1.8,-1.2)}*+{{}^4} !{(2.2,-1.2)}*+{{}^1}
 !{(1.1,-0.5)}*+{{}^2}
!{(2,1)}*+{\bullet}="c0" !{(1.7,1.2)}*+{{}^a} 
!{(3,1)}*+{\bullet}="d" !{(3,0.5)}*+{{}^{d_3}} 
!{(4,1)}*+{\star}="i3"
 !{(2.2,1.4)}*+{{}^3}
!{(1.7,0.7)}*+{{}^2} !{(1.9,0.4)}*+{{}^2}
!{(2.4,0.75)}*+{{}^2}!{(2.8,1.2)}*+{{}^8}
!{(3.2,1.2)}*+{{}^1}
  "c0"-"d" "d"-@{.>}"i3" "b1"-"c0" "d1"-"c0"
"b1"-"d1" "b1"-@{.>}"i11"  "d1"-@{.>}"i31"
"a1"-"c0" "a2"-"b1"}
\qquad\qquad (C)\ 
\xygraph{
!{<0cm,0cm>;<.8cm,0cm>:<0cm,.8cm>::}
!{(0,0)}*+{\bullet}="a2" !{(-0.3,0)}*+{{}^a}
!{(0.3,0.3)}*+{{}^3} !{(0.5,-0.3)}*+{{}^6}
!{(1,-1)}*+{\bullet}="b1" !{(1,-1.5)}*+{{}^{d_2}} 
!{(2,-1)}*+{\bullet}="d1" !{(2,-1.5)}*+{{}^{d_1}} 
!{(0,-1)}*+{\star}="i11"
!{(3,-1)}*+{\star}="i31"
!{(0.8,-1.2)}*+{{}^1} !{(1.2,-1.2)}*+{{}^4}
!{(0.85,-0.7)}*+{{}^4} !{(2.2,-0.7)}*+{{}^6}
!{(1.8,-1.2)}*+{{}^2} !{(2.2,-1.2)}*+{{}^1}
!{(2,1)}*+{\bullet}="c0" !{(2,1.3)}*+{{}^b} 
!{(3,1)}*+{\bullet}="d" !{(3,0.5)}*+{{}^{d_3}} 
!{(4,1)}*+{\star}="i3"
!{(1.5,0.9)}*+{{}^1} !{(1.9,0.4)}*+{{}^6}
!{(2.4,0.75)}*+{{}^2}!{(2.8,1.2)}*+{{}^8}
!{(3.2,1.2)}*+{{}^1}
  "c0"-"d" "d"-@{.>}"i3"  "d1"-"c0" "c0"-"a2"
"b1"-"d1" "b1"-@{.>}"i11"  "d1"-@{.>}"i31"
"a2"-"b1"}
$$
\caption{The WRFQ $\mathfrak{wq}_Q(\mathbb{O}_{D})$, where $Q$ has degree $3$.
We assume $X=\mathbb{P}^1$ and $\mathbb{F}=\mathbb{F}_q$, with $q=2$ in (B) and (C). 
The divisor is $D=P_1$ in (A), $D=P_1+P_2$ in (B) and $D=P'$ in (C), where
$\deg(P_1)=\deg(P_2)=1$, while $\deg(P')=2$. All cusp in (B) and (C) follow the 
pattern in Figure \ref{f19}.}\label{f20}
\end{figure}

For the final example we let $X$ be the elliptic curve defined over $\mathbb{F}_2$
by the equation 
$$\mathbf{y}^2+ \mathbf{y} + \mathbf{x}^3+ \mathbf{x} +1=0.$$
As usual, $K=\mathbb{F}_2(X)$.
Note that this curve has a unique rational place, the place $P=P_\infty$ at infinity.
In particular, $\mathrm{Pic}(X)\cong\mathbb{Z}$, whence the ring $A$ of rational 
functions that are regular outside $P_\infty$ is a principal ideal domain. 
In particular, if $D$ is a maximal $A$-order, the normalizer of $D$ in 
$\mathrm{GL}_2(K)$ is $K^*D^*$ by 
Lemma \ref{lem83},
whence the C-graph and the S-graph of 
maximal orders
coincide. We conclude that the WRFQ associated 
to the spinor genera $\mathbb{O}_0$
is as depicted in Figure \ref{f21}(A), as such
is the S-graph according to Example 2.4.4 in 
\S II.2.4 of \cite{trees}. Let $V$ be 
the underlying graph, and
let $\mathbb{T}:C_V\rightarrow C_V$ be 
the linear map associated to the WRFQ.
The general Theory tells us that 
$W=\{n_1,n_2,m_1,m_2,d_0\}$ is a shell for 
the obstruction space $O_{\mathbf{T}}$. Let $\mathbf{F}$ be a map with shell $W$ that
commutes with $\mathbf{T}$. 
Let  $\mathbf{P}$ be the projection onto $C_W$ with
kernel $C_{W^c}$, which also commutes with $\mathbf{F}$. 
Then the matrix of $\mathbf{P}\mathbf{T}|_{C_W}$
with respect to the basis $\{\delta_{n_1},\delta_{n_2},\delta_{m_1},
\delta_{m_2},\delta_{d_0}\}$ of $C_W$ is
$$\left(\begin{array}{ccccc}
0&0&1&0&0\\0&0&1&0&0\\3&3&0&2&0\\0&0&1&0&0\\0&0&0&0&0\end{array}\right),$$
which is a diagonalizable matrix with eigenvalues $\pm2\sqrt2$, each with
multiplicity $1$, and $0$ with multiplicity $3$. The map $\mathbf{F}$ must
leave invariant each of the corresponding eigenspaces. In particular,
The eigenvector corresponding to the eigenvalue $2\sqrt2$, namely the charge
$\mu=\delta_{n_1}+\delta_{n_2}+2\sqrt2\delta_{m_1}+\delta_{m_2}$, must also 
be an eigenvector of $\mathbf{F}$. Let $\rho$ be such that
$\mathbf{F}(\mu)=\rho\mu$. Note that $\mathbf{T}(\mu)=2\sqrt2\mu+\delta_{d_1}$,
so a computation gives 
$$\rho2\sqrt2\mu=\rho2\sqrt2\mu+0=
\mathbf{F}(2\sqrt2\mu)+\mathbf{F}(\delta_{d_1})=
\mathbf{F}\mathbf{T}(\mu)=\mathbf{T}\mathbf{F}(\mu)=
\rho2\sqrt2\mu+\rho\delta_{d_1},$$
which can happen only if $\rho=0$. For the same reason, $\mathbf{F}$
annihilates the eigenvector corresponding to $-2\sqrt2$, so its image is contained 
in the null space of $\mathbf{P}\mathbf{T}$. It follows that $\mathbf{F}\mathbf{P}\mathbf{T}=\mathbf{P}\mathbf{T}\mathbf{F}=0$. 
This conditions imply that $\mathbf{F}$ has a matrix of the form
$$\left(\begin{array}{ccccc}
2x_1&2x_2&0&-2(x_1+x_2)&2x_3\\2y_1&2y_2&0&-2(y_1+y_2)&2y_3\\
0&0&0&0&0\\-3(x_1+y_1)&-3(x_2+y_2)&0&3(x_1+x_2+y_1+y_2)&3(x_3+y_3)\\
z_1&z_2&0&-(z_1+z_2)&z_3\end{array}\right).$$
The condition that the columns add up to $0$ reads $z_i=x_i+y_i$,
for $i=1,2,3$. In particular, the set $T$ in Theorem \ref{t2} is
$T=\{n_1,n_2,m_2,d_0\}$. In other words,  the only edges whose weight 
cannot be determined from the WRFQ at $P_\infty$ are those joining two
of these four vertices. If $Q$ is a place of degree $3$, for instance,
the place corresponding to the point $(\alpha,0)$ when $\alpha$ is a root
of $x^3+x+1=0$, the WRFQ at $Q$ must corresponds to a linear map of the 
form $\mathbf{T}^3-6\mathbf{T}+\mathbf{T}'$, where $\mathbf{T}'$ has 
$T$ as a shell. 
The WRFQ corresponding to $\mathbf{T}^3-6\mathbf{T}$ is depicted
in Figure \ref{f21}(B). Since this graph has no edges joining any pair 
of vertices in $T$, the condition that each column of $\mathbf{T}'$
adds up to $1$ implies that  $\mathbf{T}'=0$, since weights cannot
be negative. We conclude that Figure \ref{f21}(B) is indeed the right
picture.\begin{figure}
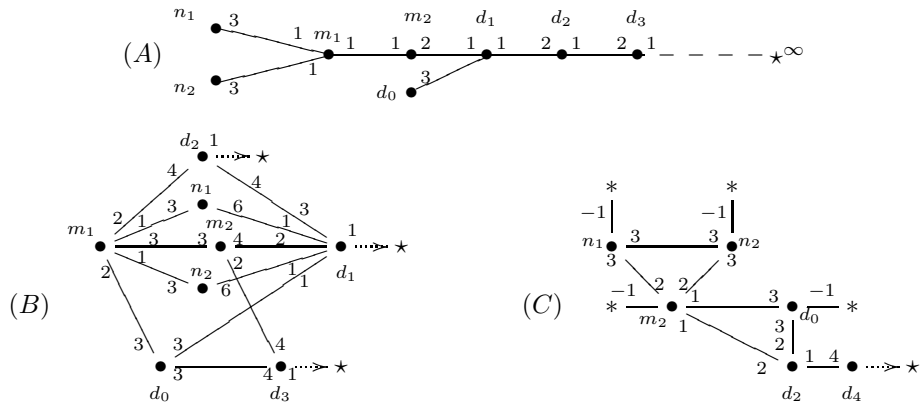

\[
(A)\xygraph{!{<0cm,0cm>;<1cm,0cm>:<0cm,1cm>::}
!{(8,0) }*+{\star^{\infty}}="vl2"
!{(-.2,0) }*+{}="v0"
!{(0,0.5) }*+{{}^{n_1}}
!{(0,-0.5) }*+{{}^{n_2}}
!{(0.35,0.38) }*+{}="v1v"
!{(0.35,-0.38) }*+{}="v2v"
!{(1.95,0) }*+{}="v0v"
!{(1.8,0) }*+{}="v0"
!{(3,0.0) }*+{\bullet}
!{(4,0) }*+{\bullet}
!{(3,-0.5) }*+{\bullet}
!{(5,0) }*+{\bullet}
!{(6,0) }*+{\bullet}
!{(1.9,0) }*+{\bullet}
!{(6.2,0) }*+{}="v7i"
!{(0.4,0.35) }*+{\bullet}
!{(0.4,-0.35) }*+{\bullet}
!{(1.9,0.2) }*+{{}^{m_1}}
!{(3.1,0.45) }*+{{}^{m_2}}
!{(4,0.45) }*+{{}^{d_1}}
!{(2.68,-0.5) }*+{{}_{d_0}}
!{(4.15,0.05) }*+{}="v5v"
!{(2.96,-0.5) }*+{}="v5bv"
!{(5,0.45) }*+{{}^{d_2}}
!{(6,0.45) }*+{{}^{d_3}}
!{(6.2,0) }*+{}="v7"
!{(.65,0.4)}*+{{}^3}!{(.65,-0.5)}*+{{}^3}!{(3.2,-0.35)}*+{{}^3}
!{(2.2,0.1)}*+{{}^1}!{(1.5,0.3)}*+{{}_1}!{(1.7,-0.25)}*+{{}^1}
!{(2.8,.1)}*+{{}^1}!{(3.2,.1)}*+{{}^2}!{(3.8,.1)}*+{{}^1}!{(4.2,.1)}*+{{}^1}
!{(4.8,.1)}*+{{}^2}!{(5.2,.1)}*+{{}^1}!{(5.8,.1)}*+{{}^2}!{(6.2,.1)}*+{{}^1}
"v0"-"v7" "vl2"-@{--}"v7i" "v0v"-"v1v" "v0v"-"v2v" "v5v"-"v5bv" 
}
\]
\[
(B)\ \xygraph{
!{<0cm,0cm>;<.8cm,0cm>:<0cm,.8cm>::}
!{(0,-1)}*+{\bullet}="b1" !{(0,-1.5)}*+{{}^{d_0}} 
!{(2,-1)}*+{\bullet}="d1" !{(2,-1.5)}*+{{}^{d_3}} 
!{(3,-1)}*+{\star}="i31"
!{(0.3,-1.2)}*+{{}^3} !{(0.3,-0.7)}*+{{}^3} 
!{(-0.35,-0.7)}*+{{}^3} !{(2,-0.7)}*+{{}^4}
!{(1.8,-1.2)}*+{{}^4} !{(2.2,-1.2)}*+{{}^1}
 !{(0.7,2.5)}*+{\bullet}="d2" !{(0.5,2.7)}*+{{}^{d_2}} 
 !{(0.2,2.2)}*+{{}^4}!{(1.6,2)}*+{{}^4} !{(0.9,2.7)}*+{{}^1}
!{(-1,1)}*+{\bullet}="c1" !{(-1.3,1.2)}*+{{}^{m_1}} 
!{(-0.1,1.05)}*+{{}^3} !{(-0.3,1.35)}*+{{}^1}
!{(-0.3,0.75)}*+{{}^1} !{(-0.7,1.4)}*+{{}^2}
!{(-0.9,0.5)}*+{{}^2}
!{(1,1)}*+{\bullet}="c2" !{(1,1.3)}*+{{}^{m_2}} 
!{(1.3,1.05)}*+{{}^4} !{(0.7,1.05)}*+{{}^3}
!{(1.3,0.65)}*+{{}^2} 
!{(0.7,1.7)}*+{\bullet}="z1" !{(0.7,1.9)}*+{{}^{n_1}} 
 !{(0.2,1.6)}*+{{}^3}!{(1.3,1.6)}*+{{}^6} 
!{(0.7,0.3)}*+{\bullet}="z2" !{(0.7,0.5)}*+{{}^{n_2}} 
 !{(0.2,0.25)}*+{{}^3}!{(1.1,0.2)}*+{{}^6} 
!{(3,1)}*+{\bullet}="d" !{(3.1,0.5)}*+{{}^{d_1}} 
!{(4,1)}*+{\star}="i3" !{(1.7,2.5)}*+{\star}="i1"
 !{(2.4,1.55)}*+{{}^3}  !{(2.1,1.35)}*+{{}^1}
!{(2.2,0.55)}*+{{}^1} !{(2.4,0.35)}*+{{}^1}
!{(3.2,1.2)}*+{{}^1} !{(2,1.05)}*+{{}^2}
  "c1"-"c2" "c2"-"d" "d"-@{.>}"i3" "d2"-@{.>}"i1" "b1"-"c1" "d1"-"c2"
  "z1"-"c1" "z1"-"d"   "z2"-"c1" "z2"-"d" "c1"-"d2" "d2"-"d" "b1"-"d"
"b1"-"d1" "d1"-@{.>}"i31"} \qquad\qquad
(C)\ \xygraph{
!{<0cm,0cm>;<.8cm,0cm>:<0cm,.8cm>::}
!{(1,1)}*+{\bullet}="n1" !{(0.7,1)}*+{{}^{n_1}} 
!{(3,1)}*+{\bullet}="n2" !{(3.3,1)}*+{{}^{n_2}} 
!{(1,2)}*+{*}="a1" !{(3,2)}*+{*}="a2" 
!{(1,0)}*+{*}="a3" !{(5,0)}*+{*}="a4" 
!{(2,0)}*+{\bullet}="m2" !{(1.7,-0.3)}*+{{}^{m_2}} 
!{(4,0)}*+{\bullet}="d0" !{(4.3,-0.2)}*+{{}^{d_0}} 
!{(4,-1)}*+{\bullet}="d2" !{(4,-1.5)}*+{{}^{d_2}} 
!{(5,-1)}*+{\bullet}="d4" !{(5,-1.5)}*+{{}^{d_4}} 
!{(6,-1)}*+{\star}="i31"
!{(0.7,1.5)}*+{{}^{-1}} !{(2.7,1.5)}*+{{}^{-1}}
!{(1.2,0.2)}*+{{}^{-1}} !{(4.5,0.2)}*+{{}^{-1}}
!{(2.4,0.1)}*+{{}^{1}} !{(3.7,0.1)}*+{{}^{3}}
!{(1.4,1.1)}*+{{}^{3}} !{(2.7,1.1)}*+{{}^{3}}
!{(1,0.7)}*+{{}^{3}} !{(3,0.7)}*+{{}^{3}}
!{(1.8,0.3)}*+{{}^{2}} !{(2.2,0.3)}*+{{}^{2}}
!{(4.3,-0.9)}*+{{}^{1}} !{(4.7,-0.9)}*+{{}^{4}}
!{(2.2,-0.4)}*+{{}^{1}} !{(3.5,-1.1)}*+{{}^{2}}
!{(3.8,-0.4)}*+{{}^{3}} !{(3.8,-0.7)}*+{{}^{2}}
"n1"-"m2" "n2"-"m2" "m2"-"d0" "m2"-"a3" "d0"-"a4"
"d0"-"d2" "m2"-"d2"
"n1"-"a1" "n2"-"a2" "n1"-"n2" "d2"-"d4" "d4"-@{.>}"i31"}
\]
\caption{
In (A) we have the WRFQ $\mathfrak{wq}_P(\mathbb{O}_0)$ at
a point $P$ of degree one,
for an elliptic curve $X$, taken from a reference.
In (B) we have $\mathfrak{wq}_Q(\mathbb{O}_0)$
for a point $Q$ of degree $3$.
In (C) we show a first approximation
for one connected component
in $\mathfrak{wq}_Q(\mathbb{O}_0)$
when $Q$ is a point of degree $2$.
The coefficient $-1$ appearing in some
positions show that it
differ non-trivially from the actual 
graph by some edges joining the four 
upper points.}\label{f21}
\end{figure}

Now assume that $Q$ is a place of degree $2$. Then one connected
component of the WRFQ corresponding to $\mathbf{T}^2-4\mathbf{T}$,
computed analogously, is depicted in Figure 
\ref{f21}(C). Note that the appearance of negative weights at all half
edges, which forces the error term $\mathbf{T}'$ to be nontrivial.
In fact, the fact that weights must be nonnegative forces the inequalities
$2x_1\geq1$, $2y_1\geq-3$ and $0\leq x_1+y_1\leq1$ whence we conclude
$$(x_1,y_1)\in\left\{\left(\frac12,\frac{-1}2\right),
\left(\frac12,\frac12\right),(1,-1),(1,0),
\left(\frac32,\frac{-3}2\right),\left(\frac32,\frac{-1}2\right),(2,-1),\left(\frac52,\frac{-3}2\right)\right\},$$
and similar lists can be given for the remaining columns.

The picture in Page 61 of \cite{Alvetall},
suggest that we have either 
$(x_1,y_1,y_2,x_2,x_3,y_3)=(2,-1,-1,1,1,0)$
or 
$(x_1,y_1,y_2,x_2,x_3,y_3)=(1,-1,-1,2,0,1)$,
depending on the choice of
the particular degree $2$ place $Q$.
Note that the curve $X$ has the automorphism
$(x,y)\mapsto(x,y+1)$ that permutes the two
degree $2$ places, and it can be shown that
this automorphism also permutes the maximal 
orders corresponding to the vertices $n_1$
and $n_2$, as either of them is associated
to a different degree $2$ place. See the 
reference for details.

\section{Acknowledgments}
This work was supported by Fondecyt, Grant No 1180471.

\end{document}